\documentclass[11pt]{article}

\usepackage{parskip}
\usepackage{lmodern}
\usepackage{anyfontsize}

\usepackage{hyperref}
\usepackage{fullpage}
\usepackage{amsmath}
\usepackage{amsthm}
\usepackage{thm-restate}
\usepackage{mathtools}
\usepackage{newpxmath}
\usepackage{newpxtext}

\usepackage{amsfonts}
\usepackage{bm}
\usepackage{enumitem}
\usepackage{color}
\usepackage{comment}
\usepackage[capitalize]{cleveref}
\usepackage[table,xcdraw,dvipsnames]{xcolor}
\usepackage{float}
\usepackage[T1]{fontenc}
\usepackage{todonotes}
\usepackage{asymptote}
\usepackage{mdframed}
\usepackage[most]{tcolorbox}
\usepackage{hyperref}
\usepackage{enumitem}
\usepackage{framed}
\usepackage{mdframed}
\usepackage{bbm}
\usepackage{tablefootnote}

\usepackage{wrapfig}
\usepackage{multicol}
\usepackage{cutwin}

\usepackage[capitalize]{cleveref}

\usepackage{mathrsfs}

\usepackage[T1]{fontenc}

\definecolor{asparagus}{rgb}{0.53, 0.66, 0.42}
\definecolor{sacramentostategreen}{rgb}{0.0, 0.3, 0.15}
\definecolor{teal}{rgb}{0.0, 0.5, 0.5}
\definecolor{forestgreen}{rgb}{0.13, 0.6, 0.13}

\usepackage{hyperref}
\hypersetup{
    colorlinks=true,
    linkcolor=sacramentostategreen,
    filecolor=sacramentostategreen,
    citecolor=teal,
    urlcolor=teal,
}

\newtheorem{theorem}{Theorem}[section]

\newtheorem{definition}[theorem]{Definition}

\newtheorem{lemma}[theorem]{Lemma}

\newtheorem{corollary}[theorem]{Corollary}

\newtheorem{problem}{Problem}
\newtheorem*{problem*}{Problem}
\newtheorem{fact}[theorem]{Fact}

\theoremstyle{remark}

\Crefname{theorem}{Theorem}{Theorems}
\Crefname{claim}{Claim}{Claims}
\Crefname{lemma}{Lemma}{Lemmas}
\Crefname{proposition}{Proposition}{Propositions}\usepackage[symbol]{footmisc}
\Crefname{corollary}{Corollary}{Corollaries}
\Crefname{definition}{Definition}{Definitions}

\newcommand{\E}{\mathbb{E}}

\newcommand{\N}{\mathbb{N}}
\newcommand{\R}{\mathbb{R}}
\newcommand{\Z}{\mathbb{Z}}

\newcommand{\cD}{\mathcal{D}}

\newcommand{\cG}{\mathcal{G}}

\newcommand{\cL}{\mathcal{L}}
\newcommand{\cM}{\mathcal{M}}
\newcommand{\cN}{\mathcal{N}}

\newcommand{\mA}{\mathbf{A}}
\newcommand{\mB}{\mathbf{B}}

\newcommand{\mD}{\mathbf{D}}

\newcommand{\mF}{\mathbf{F}}

\newcommand{\mI}{\mathbf{I}}

\newcommand{\mM}{\mathbf{M}}

\newcommand{\mR}{\mathbf{R}}

\newcommand{\mU}{\mathbf{U}}
\newcommand{\mV}{\mathbf{V}}
\newcommand{\mW}{\mathbf{W}}

\newcommand{\vmu}{\bm{\mu}}

\newcommand{\va}{\bm{a}}
\newcommand{\vb}{\bm{b}}

\newcommand{\ve}{\bm{e}}
\newcommand{\vf}{\bm{f}}
\newcommand{\vg}{\bm{g}}

\newcommand{\vm}{\bm{m}}

\newcommand{\vs}{\bm{s}}

\newcommand{\vu}{\bm{u}}
\renewcommand{\vv}{\bm{v}}
\newcommand{\vw}{\bm{w}}
\newcommand{\vx}{\bm{x}}
\newcommand{\vy}{\bm{y}}
\newcommand{\vz}{\bm{z}}

\def\to{\rightarrow}
\def\eps{\varepsilon}

\def\eps{\varepsilon}

\usepackage{algorithm}
\usepackage[noend]{algpseudocode}

\newcommand{\onev}{\mathbbm{1}}

\newcommand{\argmin}[1]{\underset{#1}{\arg\min}}

\newcommand{\suchthat}{{\;\; : \;\;}}

\newcommand{\inparen}[1]{\left(#1\right)}             
\newcommand{\inbraces}[1]{\left\{#1\right\}}           
\newcommand{\insquare}[1]{\left[#1\right]}             

\newcommand{\abs}[1]{\ensuremath{\left\lvert #1 \right\rvert}}

\newcommand{\norm}[1]{\ensuremath{\left\lVert #1 \right\rVert}}

\newcommand{\ip}[1]{\left\langle #1 \right\rangle}

\newcommand{\myfunc}[3]{#1\colon#2\to#3}

\usepackage{nicefrac}

\let\nfrac=\nicefrac

\renewcommand{\Pr}{\mathsf{Pr}}

\newcommand{\exv}[1]{\E\insquare{#1}}
\newcommand{\exvv}[2]{\underset{#1}{\E}\insquare{#2}}
\newcommand{\expv}[1]{\mathsf{exp}\inparen{#1}}
\newcommand{\prv}[1]{\Pr\insquare{#1}}
\newcommand{\prvv}[2]{\underset{#1}{\Pr}\insquare{#2}}

\newcommand{\logv}[1]{\log\inparen{#1}}
\newcommand{\lnv}[1]{\ln\inparen{#1}}

\makeatletter
\newcommand{\customlabel}[2]{%
   \protected@write \@auxout {}{\string \newlabel {#1}{{#2}{\thepage}{#2}{#1}{}} }%
   \hypertarget{#1}{#2}
}
\makeatother

\newcounter{casenum}

\newcounter{subcasenum}

\newcounter{casenump}

\newcommand{\casep}[2]{
    \ifthenelse{\equal{\value{casenump}}{0}}{
    \vskip.5\baselineskip\par\noindent
    }{}
    {\it Case \arabic{casenump}:} {\it #1}
    \vskip0.1\baselineskip
    \begin{addmargin}[1.5em]{1em}
    #2
    \end{addmargin}
    \addtocounter{casenump}{1}
}

\newcounter{subcasenump}

\newcommand{\lecture}[4]{
  \pagestyle{myheadings}
  \thispagestyle{plain}
  \newpage
  \setcounter{page}{1}
  \noindent
  \begin{center}
    \framebox{
      \vbox{\vspace{2mm}
      \hbox to 6.28in {{\bf Hypergraph Sparsification
                          \hfill Note #4}}
      \vspace{6mm}
      \hbox to 6.28in {{\Large \hfill #1  \hfill }}
      \vspace{5mm}
      \hbox to 6.28in {{#2 \hfill #3}}
      \vspace{2mm}}
    }
  \end{center}
  \markboth{#1}{#1}
  \vspace*{4mm}
}

\newcommand{\ceil}[1]{\left\lceil #1 \right\rceil}
\newcommand{\rank}[1]{\mathsf{rank}\inparen{#1}}

\usepackage{cleveref}

\usepackage{ulem}

\newcommand{\xhat}{\widehat{\vx}}

\renewcommand{\vmu}{\bm{\mu}}

\newcommand{\dtwo}{d_2}

\newcommand{\vbrho}{\eps^{-2}\inparen{\log n}^2\logv{\nfrac{n}{\eps}}}
\newcommand{\vlambda}{\bm{\lambda}}
\newcommand{\mLambda}{\mathbf{\Lambda}}
\newcommand{\mSigma}{\mathbf{\Sigma}}

\newcommand{\zstar}{\vz^{\star}}
\newcommand{\Fstar}{F^{\star}}
\newcommand{\mtilde}{\widetilde{m}}

\newcommand{\vrho}{\bm{\rho}}
\newcommand{\valpha}{\bm{\alpha}}

\newcommand{\polynorm}[1]{\norm{#1}_{\cG,\vrho,\infty, S}}
\newcommand{\subgnorm}[1]{\norm{#1}_{\psi_2}}
\newcommand{\gnorm}[2]{\norm{#1}_{\cG_{#2}}}
\newcommand{\gnorml}[2]{\norm{#1}_{\cG_{#2}(\vlambda)}}
\newcommand{\nnz}[1]{\mathsf{nnz}\inparen{#1}}

\newcommand{\xstar}{\vx^{\star}}
\newcommand{\opt}{\mathsf{OPT}}

\definecolor{ForestGreen}{RGB}{34, 139, 34}

\definecolor{forest}{RGB}{0,155,85}

\newcommand{\fixout}{\bgroup\markoverwith{\textcolor{forest}{\rule[0.5ex]{2pt}{0.4pt}}}\ULon}

\renewcommand{\phi}{\varphi}

\newif\ifnotes
\makeatletter
\newcommand{\nnote}[1]{\@bsphack\ifnotes{$\ll$\textsf{\color{blue} Naren: { #1}}$\gg$}\fi\@esphack}
\newcommand{\mnote}[1]{\@bsphack\ifnotes{$\ll$\textsf{\color{magenta} Max: { #1}}$\gg$}\fi\@esphack}
\newcommand{\anote}[1]{\@bsphack\ifnotes{$\ll$\textsf{\color{green} Antares: { #1}}$\gg$}\fi\@esphack}
\makeatother
\notesfalse

\usepackage{tikz}

\usepackage[margin=1in]{geometry}

\allowdisplaybreaks

\usepackage[
backref=true,
natbib=true,
style=trad-alpha,
maxalphanames=8,
sorting=nyt
]{biblatex}
\numberwithin{equation}{section}

\notestrue 

\begin{document}

\title{The Change-of-Measure Method, Block Lewis Weights, and Approximating Matrix Block Norms}
\author{Naren Sarayu Manoj\thanks{TTIC. Email: \url{nsm@ttic.edu}. Supported by NSF Graduate Research Fellowship and NSF Award ECCS-2216899}\and Max Ovsiankin\thanks{TTIC. Email: \url{maxov@ttic.edu}. Supported by NSF Award ECCS-2216899}}
\maketitle
\begin{abstract}\noindent
Given a matrix $\mathbf{A} \in \mathbb{R}^{k \times n}$, a partitioning of $[k]$ into groups $S_1,\dots,S_m$, an outer norm $p$, and a collection of inner norms such that either $p \ge 1$ and $p_1,\dots,p_m \ge 2$ or $p_1=\dots=p_m=p \ge 1/\log n$, we prove that there is a sparse weight vector $\bm{\beta} \in \mathbb{R}^{m}$ such that $\sum_{i=1}^m \bm{\beta}_i \cdot \|\mathbf{A}_{S_i}\bm{x}\|_{p_i}^p \approx_{1\pm\varepsilon} \sum_{i=1}^m \|\mathbf{A}_{S_i}\bm{x}\|_{p_i}^p$, where the number of nonzero entries of $\bm{\beta}$ is at most $O_{p,p_i}(\varepsilon^{-2}n^{\max(1,p/2)}(\log n)^2(\log(n/\varepsilon)))$. When $p_1\dots,p_m \ge 2$, this weight vector arises from an importance sampling procedure based on the \textit{block Lewis weights}, a recently proposed generalization of Lewis weights. Additionally, we prove that there exist efficient algorithms to find the sparse weight vector $\bm{\beta}$ in several important regimes of $p$ and $p_1,\dots,p_m$. Our results imply a $\widetilde{O}(\varepsilon^{-1}\sqrt{n})$-linear system solve iteration complexity for the problem of minimizing sums of Euclidean norms, improving over the previously known $\widetilde{O}(\sqrt{m}\log({1/\varepsilon}))$ iteration complexity when $m \gg n$.

Our main technical contribution is a substantial generalization of the \textit{change-of-measure} method that Bourgain, Lindenstrauss, and Milman used to obtain the analogous result when every group has size $1$. Our generalization allows one to analyze change of measures beyond those implied by D. Lewis's original construction, including the measure implied by the block Lewis weights and natural approximations of this measure.
\end{abstract}

\newpage
\tableofcontents
\newpage

\section{Introduction}

Suppose we are given a large dataset that is computationally inconvenient to work with in a downstream task. To alleviate this, we can try to randomly sample a small representative subset of the original dataset. The design and analysis of randomized sampling algorithms for this purpose is well-explored (for example, see \cite{ss08,mmwy21,wy22_oneshot,wy23_sens} for preserving $\ell_p$ objectives, \cite{fl11} for preserving objectives for $k$-median, projective clustering, subspace approximation, and more, \cite{ss08,kkty21a,jls22,lee22} for preserving graph and hypergraph $\ell_2$-energy, and \cite{jlls23} for sums (of powers) of general norms).

In order to design randomized sampling algorithms, we first need to understand the properties of the original dataset we want to preserve. To this end, we study the problem of preserving \textit{block $p$-norm objectives}. Let $\cG = (\mA \in \R^{k\times n}, S_1,\dots,S_m, p_1,\dots,p_m)$ be a dataset consisting of a matrix $\mA \in \R^{k\times n}$. Consider a partitioning of $[k]$ into groups $S_1,\dots,S_m$ and consider positive numbers $p_1,\dots,p_m$. Let $\mA$ have rows $\va_1,\dots,\va_k$ and denote by $\mA_{S_i}$ the matrix in $\R^{\abs{S_i} \times n}$ whose rows are the rows of $\mA$ indexed by $S_i$. Consider the function $\gnorm{\mA\vx}{p}$ on some input vector $\vx \in \R^n$:
\begin{align}
    \gnorm{\mA\vx}{p}^p \coloneqq \sum_{i=1}^m \norm{\mA_{S_i}\vx}_{p_i}^p \label{eq:energy}
\end{align}
We use the norm notation because we can easily verify that for $p \ge 1$ and $p_i \ge 1$ for all $i$, $\gnorm{\cdot}{p}$ is a norm. We remark that objectives of the form of \eqref{eq:energy} are widely studied in geometric functional analysis, theoretical computer science, and data science. In \Cref{sec:results}, we go over one important application of the objective \eqref{eq:energy}. We defer a broader discussion of more applications and connections to \Cref{sec:related_works}.

Our goal in this paper is to design and analyze randomized sampling algorithms to output a weighted subset that preserves \eqref{eq:energy} for all $\vx \in \R^n$. We give a formal problem statement for the general problem we study in \Cref{prob:main_problem}. 

\begin{problem}[$\ell_p$ block norm sampling]
\label{prob:main_problem}
We are given as input $\cG = \inparen{\mA \in \R^{k\times n}, S_1,\dots,S_m, p_1,\dots,p_m}$, $p > 0$, and an error parameter $\eps$. For all $i \in [m]$, we must output a probability distribution $\vrho_1,\dots,\vrho_m$ over $[m]$ such that if we choose a collection of groups $\cM = (i_1,\dots,i_{\mtilde})$ where each $i_h$ is independently distributed according to $\vrho_i$, then the following holds with probability $\ge 1-\delta$:
\begin{align}
    \text{for all } \vx \in \R^n:\quad\quad\inparen{1-\eps}\gnorm{\mA\vx}{p}^p \le \frac{1}{\mtilde}\sum_{i\in \cM} \frac{1}{\vrho_i}\cdot\norm{\mA_{S_i}\vx}_{p_i}^p \le \inparen{1+\eps}\gnorm{\mA\vx}{p}^p\label{eq:main_objective}
\end{align}
We would like $\mtilde$ to be small with probability $1-\delta$ (for example, $\mtilde$ should not depend on $m$ and the dependence on $\delta^{-1}$ should be polylogarithmic).
\end{problem}
Observe that the formulation of \Cref{prob:main_problem} is an instantiation of an \textit{importance sampling} framework. Specifically, we can think of the distribution $\cD$ as consisting of importance scores for each group. We form our sparse approximation by sampling group $i$ with probability $\vrho_i$ and reweighting appropriately so that the function we return is an unbiased estimator of $\gnorm{\mA\vx}{p}$. We call $\mtilde$ the \textit{sparsity} of the procedure described in \Cref{prob:main_problem}. Additionally, in the statement of our results, we will assume that $p$ is a constant (and thus any function solely of $p$ will treated as a constant in any $O\inparen{\cdot}$ or $\Omega\inparen{\cdot}$ terms). 

In this paper, we give new results for Problem \ref{prob:main_problem} and show how these imply faster algorithms for commonly implemented optimization problems.

\subsection{Our results}
\label{sec:results}

For a quick summary of our existence results for the block norm sampling problem (\Cref{prob:main_problem}), see \Cref{table:comparison}. 

We begin with stating our main result\footnote{In the statement of \Cref{thm:one_shot_lewis}, writing the lower bound $p \ge 1/\log n$ instead of $p > 0$ is somewhat arbitrary -- we choose this lower bound to make our calculations easier later on.}, \Cref{thm:one_shot_lewis}.

\begin{restatable}[Block Lewis weight sampling]{mainthm}{oneshotlewis}
\label{thm:one_shot_lewis}
Let $\cG = (\mA \in \R^{k \times n}, S_1,\dots,S_m, p_1,\dots,p_m)$ where $S_1,\dots,S_m$ form a partition of $[k]$. Suppose at least one of the following holds:
\begin{itemize}
    \item $1 \le p < \infty$ and $p_1,\dots,p_m \ge 2$;
    \item $1/\log n \le p_1 = \dots = p_m = p < \infty$;
    \item $p_1 = \dots = p_m = 2$ and $1/\log n \le p < \infty$.
\end{itemize}
Let $P \coloneqq \max\inparen{1, \max_{i \in [m]} \min(p_i,\log\abs{S_i})}$. Then, there exists a probability distribution $\cD = \inparen{\vrho_1,\dots,\vrho_m}$ such that if
\begin{align*}
    \mtilde &= \Omega\inparen{\logv{\nfrac{1}{\delta}}\vbrho P \cdot  n^{\max(1,p/2)}},
\end{align*}
and if we sample $\cM \sim \cD^{\mtilde}$, then, with probability $\ge 1-\delta$,
\begin{align*}
    \text{for all } \vx \in \R^n,\quad (1-\eps)\gnorm{\mA\vx}{p}^p \le \frac{1}{\mtilde}\sum_{i \in \cM} \frac{1}{\vrho_i} \cdot \norm{\mA_{S_i}\vx}_{p_i}^p \le (1+\eps)\gnorm{\mA\vx}{p}^p.
\end{align*}
\end{restatable}

We prove \Cref{thm:one_shot_lewis} in \Cref{sec:lewis}. It will follow from \Cref{thm:concentration} (stated and proven in \Cref{sec:generic_chaining}), which is a more general but more technical statement that also includes a description of the relevant distributions $\cD$.

We remark that when $p \ge 1$ and $p_1,\dots,p_m \ge 2$, the sampling probabilities \(\vrho\) mentioned in \Cref{thm:one_shot_lewis} can be found using the optimality conditions of a particular optimization problem that was stated and analyzed by \citet[Section 4]{jlls23}. That problem itself can be viewed as the natural generalization of the determinant maximization problem that yields the existence of Lewis's measure (see \cite[Section 2]{sz01} for details). However, \cite{jlls23} did not address the question of whether sparsification guarantees could be obtained with these weights beyond the case where the ``outer norm'' satisfies $p=2$.

Additionally, although \cite{jlls23} study sparsification of sums of norms and sums of powers $p > 1$ of uniformly smooth norms, we obtain an improved sparsity in the case entailed by \Cref{prob:main_problem} (by a factor of $\psi_n\logv{\nfrac{n}{\eps}}^{\min(p-1,2)}$, where $\psi_n$ is the KLS ``constant'' in $n$ dimensions). We defer a more detailed comparison of our existence results to \Cref{sec:related_works}.

Furthermore, it is well-known that the polynomial terms in the sparsities in \Cref{thm:one_shot_lewis} are optimal. In particular, \citet[Corollary 1.6 and Theorem 1.7]{lww19} show that $\Omega(n^{\max(1,p/2)}+\eps^{-2}\mathrm{polylog}(\eps^{-1})n)$ rows must be chosen in order to satisfy the requirement imposed by (\ref{eq:main_objective}).

Finally, the setting where $p = p_1 = \dots = p_m$ is a particularly important case of \Cref{prob:main_problem}. Here, we see that $\norm{\mA\vx}_p = \gnorm{\mA\vx}{p}$, and so \Cref{prob:main_problem} amounts to finding an $\ell_p$ subspace embedding under a group constraint (that certain rows must be kept together in the subsample). This might be a useful notion in practice, where the $S_i$ denote related observations that should be kept together for some downstream application. Moreover, this can be viewed as a higher-rank analog of $\ell_p$ row sampling, somewhat similarly to how the matrix Chernoff bound gives a higher-rank analog for the concentration of sums of bounded random matrices when compared to the rank-$1$ variant of \citet{rudelson1996}.

\paragraph{Computing sampling probabilities.}
The previous results show the existence of sampling probabilities \(\vrho_1, \ldots, \vrho_m\) such that sampling using those probabilities gives a sparsifier in the setting of \Cref{prob:main_problem}.
To get a sparsification \textit{algorithm}, we need to also compute (or approximate) the sampling probabilities.

We give efficient algorithms to do so in natural cases.

\begin{restatable}[Computation of block Lewis weights]{mainthm}{computeblw}
\label{thm:computeblw}
Consider the setting of \Cref{thm:one_shot_lewis} and suppose at least one of the following holds:
\begin{itemize}
    \item $p = 2$ and $p_1,\dots,p_m \ge 2$;
    \item $1/\log n \le p_1=\dots=p_m=p < \infty$;
    \item $p_1 = \dots = p_m = 2$ and $1/\log n \le p < \infty$.
\end{itemize}
Let $P = \max\inparen{1, \max_{i \in [m]} \min(p_i,\log\abs{S_i})}$ and set
\begin{align*}
    \mtilde &= O\inparen{\logv{\nfrac{1}{\delta}}\vbrho P \cdot n^{\max(1,p/2)}}.
\end{align*}
Then, there is an algorithm that outputs a probability distribution $\cD = (\vrho_1,\dots,\vrho_m)$ such that sampling a multiset $\cM \sim \cD^{\mtilde}$ satisfies, with probability \(1 - \delta\),
\begin{align*}
    \text{for all } \vx \in \R^n,\quad (1-\eps)\gnorm{\mA\vx}{p}^p \le \frac{1}{\mtilde}\sum_{i \in \cM} \frac{1}{\vrho_i} \cdot \norm{\mA_{S_i}\vx}_{p_i}^p \le (1+\eps)\gnorm{\mA\vx}{p}^p,
\end{align*}
Further, the algorithm to find $\cD$ performs at most \(\text{polylog}(k,n,m)\) leverage score overestimate computations or linear system solves.
\end{restatable}

We prove \Cref{thm:computeblw} in \Cref{sec:alg}.

We formally define a \textit{leverage score overestimate computation} in \Cref{def:lev_score_over}. Alternately, these can be implemented using linear system solvers that solve systems of the form $\mA^{\top}\mD\mA\vy=\vz$ for diagonal $\mD$ (see \cite{ls19} for details). Although the runtime of this primitive depends on the structure of the input, each such iteration runs in $\widetilde{O}(\mathsf{nnz}(\mA) + n^{\omega})$ time. Moreover, in the special case where the matrix $\mA$ is a graph edge-incidence matrix, the runtime improves to $\widetilde{O}(\mathsf{nnz}(\mA))$.


Finally, we note that our algorithms are faster than the log-concave sampling-based routines given in \cite{jlls23} for calculating sparse approximations to sums (of powers of) more general norms, when the outer norm $p$ satisfies $1 \le p \le 2$ (they do not give algorithms for the case where $p > 2$). In particular, while their algorithm applies to a more general setting, the runtime is $\widetilde{O}(m + n^5)$. In contrast, since our algorithms only depend on a polylogarithmic number of leverage score overestimate computations or linear system solves, we can obtain much faster runtimes (in particular improved powers of $n$). This means that we can apply our algorithms to downstream optimization tasks where the main computational primitive is a linear system solver (as is the case for many general frameworks for convex programming).

\paragraph{Applications to minimizing sums of Euclidean norms.} A well-studied regression task is the \textit{minimizing sums of Euclidean norms} (MSN) problem. We are given \(\mA \in \R^{k \times n}\) and \(\vb \in \R^k\), and a partition \(S_1, \ldots, S_m\) of \([k]\).
In this problem, we would like to find
\begin{align}
    \quad\quad\min_{\vx\in\R^n} \sum_{i=1}^m \norm{\mA_{S_i} \vx-\vb_{S_i}}_2. \label{eq:msn}
\end{align}
Solving the MSN objective \eqref{eq:msn} subsumes several widely implemented optimization problems such as variants of Euclidean single facility location, Euclidean multifacility location, Euclidean Steiner minimum tree under a given topology, and plastic collapse analysis. See the long line of work on this problem \cite{and96,xy97,acco00,qsz02} for a more detailed discussion. Additionally, observe that if all $\abs{S_i}=1$, then \eqref{eq:msn} is nothing but $\ell_1$ regression (i.e., $\min_{\vx\in\R^n} \norm{\mA\vx-\vb}_1$). Thus, \eqref{eq:msn} is a generalization of $\ell_1$ regression. Finally, notice that \eqref{eq:msn} subsumes the \textit{stochastic robust approximation problem} when the norm in question is the Euclidean norm and the design $\mA$ assumes a finite number of values -- see \cite[Section 6.4.1]{boyd2004convex} for further discussion.

In this paper, we will be interested in algorithms that return a $(1+\eps)$-multiplicative approximation to the objective -- namely, we desire a point $\widehat{\vx} \in \R^n$ such that
\begin{align*}
    \sum_{i=1}^m \norm{\mA_{S_i}\widehat{\vx}-\vb_{S_i}}_2 \le (1+\eps)\min_{\vx\in\R^n} \sum_{i=1}^m \norm{\mA_{S_i} \vx-\vb_{S_i}}_2.
\end{align*}
To our knowledge, the best known algorithms based on interior point methods output a $(1+\eps)$-approximate solution to \eqref{eq:msn} $\widetilde{O}(\sqrt{m}\logv{\nfrac{1}{\eps}})$ calls to a linear system solver \cite{and96,xy97} for matrices of the form $\mA^{\top}\mD\mA$ for block-diagonal matrices $\mD$.

By applying \Cref{thm:computeblw} on the matrices $\insquare{\mA_{S_i} \vert \vb_{S_i}} \in \R^{(n+1) \times \abs{S_i}}$ with $p_1=\dots=p_m=2$ and $p=1$, observe that within $\widetilde{O}(1)$ linear system solves in matrices $\mA^{\top}\mD\mA$ for nonnegative diagonal $\mD$, we obtain an objective with $\widetilde{O}(\eps^{-2} \cdot n)$ terms that approximates \eqref{eq:msn} up to a $(1\pm\eps)$ multiplicative factor on all vectors $\vx \in \R^{n+1}$ whose last coordinate is $1$. This immediately implies \Cref{thm:msnalg}.

\begin{restatable}[Minimizing sums of Euclidean norms]{mainthm}{msnalg}
\label{thm:msnalg}
Let $\mA \in \R^{k \times n}$ and \(\vb \in \R^k\), and $S_1, \dots, S_m$ be a partition of $k$. There exists an algorithm that, with probability $\ge 1-\delta$, returns $\widehat{\vx}$ such that
\begin{align*}
    \sum_{i=1}^m \norm{\mA_{S_i}\widehat{\vx}-\vb_{S_i}}_2 \le (1+\eps)\min_{\vx\in\R^n} \sum_{i=1}^m \norm{\mA_{S_i} \vx-\vb_{S_i}}_2.
\end{align*}
The algorithm runs in $\widetilde{O}\inparen{\nfrac{\sqrt{n}}{\eps} \cdot \sqrt{\logv{\nfrac{1}{\delta}}}}$ calls to a linear system solver in matrices of the form $\mA^{\top}\mD\mA$ for block-diagonal matrices $\mD$, where each block has size $(\abs{S_i}+1) \times (\abs{S_i}+1)$.
\end{restatable}

We prove \Cref{thm:msnalg} in \Cref{sec:applications_msn}.

\Cref{thm:msnalg} improves over the best-known iteration complexities for solving \eqref{eq:msn} when the number of summands is much larger than the input dimension, i.e., $m \gg n$. Furthermore, the iteration complexity stated in \Cref{thm:msnalg} matches the iteration complexity for $\ell_1$ regression up to the $\eps^{-1}$ term \cite{bllssw21}. It is an interesting (but probably challenging) open problem to design and analyze an algorithm for \eqref{eq:msn} with iteration complexity $\widetilde{O}(\sqrt{n}\logv{\nfrac{1}{\eps}})$, which would exactly match what is known for $\ell_1$ regression.

Finally, we note that in the special case of the geometric median, where all the $\mA_{S_i} = \mI_d$ for some fixed dimension \(d\), an algorithm with runtime $\widetilde{O}(\mathsf{nnz}(\vb)\logv{\nfrac{1}{\eps}}^3)$ is known due to \citet{clmps16}. The algorithm is a long-step interior point method with a custom analysis and follows from different techniques from ours.

\paragraph{Outline.} The rest of this paper is organized as follows. In the remainder of this section, we establish notation that we use throughout the rest of the paper (\Cref{sec:notation}), give an overview of our technical methods (\Cref{sec:overview}), and discuss some prior and related works (\Cref{sec:related_works}). In \Cref{sec:background}, we give background from linear algebra, convex geometry, and probability that we rely on for the rest of the paper. In \Cref{sec:coverings}, we prove bounds on geometric quantities known as \textit{covering numbers}. These play a crucial role in our concentration arguments. In \Cref{sec:generic_chaining}, we prove that our general sampling scheme concentrates and therefore preserves the original objective on all $\vx \in \R^n$, with high probability. In \Cref{sec:applications}, we show how to apply our general sampling scheme to the problems we discuss in \Cref{sec:results}. Finally, in \Cref{sec:alg}, we describe our algorithmic results. 

\subsection{Notation and definitions}
\label{sec:notation}

\paragraph{General notation.} For positive integer $N$, we let $[N]$ denote the set $\inbraces{i \in \Z \suchthat 1 \le i \le N}$. All $\log$s are base $2$; we use $\ln$ to denote the natural logarithm. We let $\ve_1,\dots,\ve_n$ denote the standard basis vectors in $\R^n$. When we write $a \lesssim b$, we mean that $a \le Cb$ for some universal constant $C > 0$.

\paragraph{Linear algebra notation.} In this paper, we work extensively with matrices and vectors. We always denote matrices with capital letters in boldface (e.g. $\mA$) and vectors with lowercase letters in boldface (e.g. $\vx$). With a few exceptions, we write the rows of a matrix using the lowercase boldface version of the same letter used to write the matrix along with a subscript denoting which index the row corresponds to. For example, $\va_i$ denotes the $i$th row of matrix $\mA$. In a slight abuse of notation, for a symmetric matrix $\mM$, we let $\mM^{-1} \coloneqq \sum_{i=1}^{\rank{\mM}} \lambda_i^{-1}\vu_i\vu_i^{\top}$, where $\vu_i$ is the $i$th eigenvector of $\mM$. In other words, we write $\mM^{-1}$ to denote the pseudoinverse of $\mM$ when $\mM$ is symmetric. We will never use the inverse notation $\mM^{-1}$ for a non-symmetric matrix $\mM$.

\subsection{Technical overview}
\label{sec:overview}

In this subsection, we give a bird's eye view of the technical methods behind our proof of \Cref{thm:one_shot_lewis}.

\subsubsection{Concentration}
\label{sec:overview_concentration}

We begin with an explanation of our concentration proof. This type of argument has become standard in the line of work on sparsification (particularly in \cite{lee22,jlls23}), but we include a description for completeness.

Let $B_p \coloneqq \inbraces{\vx\in\R^n\suchthat\gnorm{\mA\vx}{p}\le1}$. By a standard symmetrization reduction, it suffices to fix $i_1,\dots,i_{\mtilde}$ (not necessarily distinct) and argue that for independent $R_1,\dots,R_{\mtilde}$ where $R_h \sim \mathsf{Unif}\inparen{\pm 1}$ we have
\begin{align}
    \exvv{R_h}{\sup_{\vx \in B_{p}} \abs{\sum_{h=1}^{\mtilde} R_h \cdot \frac{\norm{\mA_{S_{i_h}}\vx}_{p_{i_h}}^p}{\vrho_{i_h}}}} \le \mtilde\cdot\eps.\label{eq:intro_empirical_process}
\end{align}
Intuitively, satisfying \eqref{eq:intro_empirical_process} means that for the rebalancing of the groups given by the $\vrho_{i_h}$, a Rademacher average of the groups evaluated on every point in $\vx$ is close to $0$. It is straightforward to check that the above instantiation is a subgaussian process under an appropriately chosen distance function $\myfunc{\dtwo}{\R^{n} \times \R^{n}}{\R}$. Thus, we will apply chaining \cite{tal21}, which can be thought of as simultaneously controlling \eqref{eq:intro_empirical_process} on $\eps$-nets of $B_p$ using the metric $\dtwo$, for all $\eps>0$.

To apply chaining, the main technical task is to understand the \textit{entropy numbers} $e_N(B_{p},\dtwo)$. The entropy numbers $e_N(B_p, \dtwo)$ are the values $\eta$ that answer the question, ``what is the smallest $\eta$ such that $B_p$ can be covered by at most $2^{2^N}$ balls of radius $\eta$ in $\dtwo$?'' (or see \Cref{defn:entropy_numbers}).

\subsubsection{Covering numbers}
\label{sec:overview_geometry}

In this subsection, we explain how to control the entropy numbers as required by \Cref{sec:overview_concentration}. We first define the sampling body (\Cref{defn:sampling_polyhedron}).

\begin{definition}[Sampling body]
\label{defn:sampling_polyhedron}
Let $S$ be some subset of $[m]$ and $\vrho_1,\dots,\vrho_m$ be a probability distribution. Define the norm $\polynorm{\vx} \coloneqq \max_{i \in S} \vrho_i^{-1/p}\norm{\mA_{S_i}\vx}_{p_i}$. We call the unit norm ball of $\polynorm{\vx}$ the \textit{sampling body}.
\end{definition}

Recall that the covering number $\cN(K_1,K_2)$ for two symmetric convex bodies $K_1$ and $K_2$ is the minimum number of translates of $K_2$ required to cover $K_1$. Additionally, recall from the previous subsection the notion of \textit{entropy numbers} (which we will define in \Cref{defn:entropy_numbers}). We will reduce controlling \eqref{eq:intro_empirical_process} to bounding the entropy numbers
\begin{align*}
    e_N\inparen{\inbraces{\vx\in\R^n\suchthat\gnorm{\mA\vx}{p}\le 1}, \inbraces{\vx\in\R^n\suchthat\polynorm{\vx}\le 1}}
\end{align*}
when $N$ is small. This places us in the setting where a simple volume-based argument becomes suboptimal. In this range, the dual Sudakov inequality (\Cref{fact:sudakov}) is the technical workhorse that allows us to get sharper bounds than what we would get if we applied just a volume-based bound. It states that if $B$ is the Euclidean ball in $n$ dimensions and $K$ is some symmetric convex body in $n$ dimensions, then we have
\begin{align*}
    \log\cN(B,\eta K) \lesssim \eta^{-2} \exvv{\vg\sim\cN(0,\mI_n)}{\norm{\vg}_K}^2,
\end{align*}
where $\norm{\cdot}_K$ is the \textit{gauge norm} for $K$, defined by $\norm{\vx}_K \coloneqq \inf\inbraces{t > 0 \suchthat \vx/t \in K}$.

However, applying the dual Sudakov inequality requires that we analyze covering numbers of the form $\log\cN(B, K)$ where $B$ is the Euclidean ball in $n$ dimensions and $K$ is some symmetric convex body. Denoting $\{\vx\in\R^n\suchthat\polynorm{\vx}\le 1\}$ by $K$, we see that we cannot immediately apply the dual Sudakov inequality to bound $\log\cN(B_p,\eta K)$. This is because $B_p$ is not a (linear transformation of a) Euclidean ball. The work of \cite{jlls23} resolve this by generalizing the dual Sudakov inequality to cover arbitrary symmetric convex bodies. Unfortunately, this approach is not optimal in every setting. One source of the loss arises from exploiting the concentration of Lipschitz functionals of isotropic log-concave random vectors -- improving the bounds on this concentration depends on further progress on the KLS conjecture. Another is that the one-dimensional marginals of isotropic log-concave random variables, without any further assumptions, are only subexponential.

To escape these inefficiencies, we will want to try to find a way to apply the dual Sudakov inequality as-is. We may then exploit the concentration of Lipschitz functionals of Gaussian random vectors, which we do have a tight understanding of (for a precise statement, see \Cref{fact:lip_concentration}). A natural attempt is to first observe that for any $t > 0$,
\begin{align}
    \log\cN\inparen{B_p, \eta K} \le \log\cN\inparen{B_p, t\widehat{B_2}} + \log\cN\inparen{t\widehat{B_2}, \eta K} = \log\cN\inparen{B_p, t\widehat{B_2}} + \log\cN\inparen{\widehat{B_2}, \frac{\eta}{t}\cdot K}.\label{eq:intro_split_covering}
\end{align}
We will choose $\widehat{B_2}$ to be a linear transformation of a Euclidean ball so that we can control $\log\cN\inparen{\widehat{B_2}, \eta/t \cdot K}$ using the dual Sudakov inequality. 

Here, we will split our argument based on whether $p \ge 2$. When $p \ge 2$, it will become clear later on that it will be sufficient to choose $\widehat{B_2}$ so that $B_p \subseteq \widehat{B_2}$. Then, it is easy to see that when $t=1$, we get $\log\cN(B_p,t\widehat{B_2}) = 0$. Hence, we have $\log\cN(B_p, \eta K) \le \log\cN(\widehat{B_2}, \eta K)$, and the required bound will follow from exploiting the concentration of Lipschitz functionals of Gaussian random vectors and then applying the dual Sudakov inequality.

However, when $p < 2$, we are still left with a pesky $\log\cN(B_p,t \widehat{B_2})$ term. Loosely, this is almost dual to the statement of the dual Sudakov inequality. Now, because it is known that covering number duality does hold when one of the bodies in question is the Euclidean ball, it may be tempting to simply write $\log\cN(B_p, t\widehat{B_2}) = \log\cN(\widehat{B_2},t B_q)$ where $B_q$ is the dual ball to $B_p$ after applying some linear transformation to map $\widehat{B_2}$ to $B$. The challenge here is that we do not believe that the gauge of the resulting $B_q$ has a form that is amenable to analysis. We will therefore need to be more careful, and we describe our alternative approach in \Cref{sec:overview_measure}.

\subsubsection{The change-of-measure principle and norm interpolation}
\label{sec:overview_measure}

Recall from the previous part that our goal is to bound $\log\cN(B_p,t\widehat{B_2})$ when $p < 2$.

We are now ready to introduce our main conceptual message -- \textit{by changing the measure under which we take norms, we can almost automatically identify a linear transformation of a Euclidean ball $\widehat{B_2}$ that is a good approximation to $B_p$}. This sort of idea has already been used by \citet{blm89} and \citet{sz01} to obtain the required $\widehat{B_2}$ in the special case where all the $S_i$ are singletons. We will generalize this machinery to give similar results for the block norm sampling problem.

Let us describe this idea further. Let $\vlambda_1,\dots,\vlambda_m$ denote a probability measure over the groups. Let $\mLambda \in \R^{k\times k}$ be the diagonal matrix such that if $j \in S_i$, then $\mLambda_{jj} = \vlambda_i$. Finally, for any $r > 0$ and $\vy\in\R^k$, let $\gnorml{\vy}{r} = (\sum_{i \le m} \vlambda_i\norm{\vy_{S_i}}_{p_i}^{r})^{1/r}$ and $B_r \coloneqq \{\vx\in\R^n \suchthat \gnorml{\mLambda^{-1/p}\mA\vx}{r} \le 1\}$. Notice that under this definition, we  still have \(B_p\) as before. We will first describe the argument when the groups are singletons, then explain how to move onto the general case. We will take $\widehat{B_2} = B_2$; it is easy to see that this is a linear transformation of a Euclidean ball.

Next, notice that by log-convexity of norms, if we choose $0 < \theta < p$ and $r > 2$ for which $1/2 = (\theta/2)/p + (1-\theta/2)/r$, we have
\begin{align}
    \gnorml{\vy}{2}^2 \le \gnorml{\vy}{p}^{\theta} \cdot \gnorml{\vy}{r}^{2-\theta}.\label{intro:compact_rounding_interpolation}
\end{align}
We will exploit this observation as follows. For all integers $h \ge 0$, we will show that there exists a set $\cL_h$ that is (a subset of) the unit ball of $B_2$ such that every pair of points in $\cL_h$ is $\delta_h$-separated according to $\gnorml{\cdot}{r}$. We will find $\delta_h$ according to the interpolation inequality \eqref{intro:compact_rounding_interpolation}. Furthermore, we will generate $\cL_h$ using a sort of compactness argument arising from a $B_2$-maximally separated subset of $B_p$. This means we get, for every $h \ge 0$,
\begin{align*}
    \log\cN(\widehat{B_2}, \delta_h B_r) \ge \log\abs{\cL_h} \ge \logv{\frac{\cN(B_p,8^h t \widehat{B_2})}{\cN(B_p, 8^{h+1} t \widehat{B_2})}}.
\end{align*}
Then, summing over $h \ge 0$ (noting that once $h$ is sufficiently large, $\cN(B_p, 8^ht \widehat{B_2}) = 1$), we have
\begin{align*}
    \log\cN(B_p, t\widehat{B_2}) \le \sum_{h \ge 0} \log\cN(\widehat{B_2},\delta_h B_r).
\end{align*}
Notice that the right hand side can be evaluated using the dual Sudakov inequality\footnote{For technical reasons that will be clearer in \Cref{sec:coverings}, we will have to do this after another interpolation step.} (recall the previous section), so it suffices to show that $\log\cN(\widehat{B_2},\eta B_r)$ is small.

This is where the choice of measure becomes crucial. Since both $\widehat{B_2}$ and $B_r$ are dependent on our choice of measure $\vlambda$, we will need to carefully choose the measure so that our covering numbers are well-behaved. A classical result of \citet{Lewis1978} establishes the existence of a change-of-measure under which we simultaneously get:
\begin{equation}\label{intro:rounding}
  \begin{aligned}
    n^{1/2-1/r}B_r &\subset \widehat{B_2} \subset B_r & \text{for all } r < 2 \\
    B_r &\subset \widehat{B_2} \subset n^{1/2-1/r}B_r & \text{for all } r > 2
  \end{aligned}
\end{equation}
This change-of-measure corresponds to the ``$\ell_p$ Lewis weights'' of $\mA$ (in particular, if $\vw_i$ is the $i$th $\ell_p$ Lewis weight, then we set $\vlambda_i = \vw_i/n$). It will turn out that this choice of $\vlambda$ is enough for us to ensure that $\cN(\widehat{B_2}, \eta B_r)$ is sufficiently small for our purposes, which eventually follows from \eqref{intro:rounding}.

\paragraph{Handling general $S_i$.} The main challenge with directly porting this argument to the block norm sampling problem is that $B_2$ is not a linear transformation of a Euclidean ball unless $p_1=\dots=p_m=2$. We will therefore have to choose $\widehat{B_2}$ to be a ``rounding'' of $B_2$ such that $B_2 \subseteq \widehat{B_2}$. Observe that the interpolation step \eqref{intro:compact_rounding_interpolation} will continue to hold here, as we will get $\norm{\vy}_{\widehat{B_2}} \le \gnorml{\vy}{2}$. However, if $\widehat{B_2}$ is chosen suboptimally, then there could a large loss in the interpolation step \eqref{intro:compact_rounding_interpolation}.

To understand what we need from our measure and rounding, let us try to derive a version of \eqref{intro:rounding} for general $S_i$. We show an example of this calculation for $r = p \le 2$; the other cases follow similarly. Let $\vlambda \in \R^m_{\ge 0}$ denote a probability measure. Let $\mW$ be a diagonal ``rounding matrix'' so that for all $\vx \in \R^n$, we have
\begin{align*}
    \norm{\mW^{1/2}\mLambda^{1/2-1/p}\mA\vx}_2 \le \gnorm{\mLambda^{1/2-1/p}\mA\vx}{2} = \gnorml{\mLambda^{-1/p}\mA\vx}{2}.
\end{align*}
Letting $\widehat{B_2} = \inbraces{\vx\in\R^n\suchthat \norm{\mW^{1/2}\mLambda^{1/2-1/p}\mA\vx}_2 \le 1}$, the above inequality gives $B_2 \subseteq \widehat{B_2}$, as desired. Next, observe that since $\vlambda$ is a probability measure, we get $B_2 \subseteq B_p$ for free. For the other direction, we write
\begin{align*}
    \gnorml{\mLambda^{-1/p}\mA\vx}{2}^2 &= \gnorm{\mLambda^{1/2-1/p}\mA\vx}{2}^2 = \sum_{i=1}^m \vlambda_i\norm{\vlambda_i^{-1/p}\mA\vx}_{p_i}^2 = \sum_{i=1}^m \vlambda_i\norm{\vlambda_i^{-1/p}\mA\vx}_{p_i}^{p}\norm{\vlambda_i^{-1/p}\mA\vx}_{p_i}^{2-p} \\
    &\le \sum_{i=1}^m \vlambda_i\norm{\vlambda_i^{-1/p}\mA\vx}_{p_i}^{p} \cdot \max_{i \in [m]} \norm{\vlambda_i^{-1/p}\mA\vx}_{p_i}^{2-p} = \gnorm{\mA\vx}{p}^p \cdot \max_{i \in [m]} \norm{\vlambda_i^{-1/p}\mA\vx}_{p_i}^{2-p} \\
    &\le \gnorm{\mA\vx}{p}^p \cdot \max_{i \in [m]} \inparen{\max_{\vx \in \R^n} \frac{\norm{\vlambda_i^{-1/p}\mA\vx}_{p_i}^{2}}{\gnorm{\mLambda^{1/2-1/p}\mA\vx}{2}^2}}^{1-p/2} \cdot \gnorm{\mLambda^{1/2-1/p}\mA\vx}{2}^{2-p}.
\end{align*}
We combine the $\gnorm{\mLambda^{1/2-1/p}\mA\vx}{2}$ terms and take the $p$th root of both sides, giving
\begin{align*}
    \gnorm{\mLambda^{1/2-1/p}\mA\vx}{2} &\le \gnorm{\mA\vx}{p} \cdot \max_{i \in [m]} \inparen{\max_{\vx \in \R^n} \frac{\norm{\vlambda_i^{-1/p}\mA\vx}_{p_i}^{2}}{\gnorm{\mLambda^{1/2-1/p}\mA\vx}{2}^2}}^{1/p-1/2} \\
    &= \gnorm{\mA\vx}{p} \cdot \max_{i \in [m]} \inparen{\frac{1}{\vlambda_i} \cdot \underbrace{\max_{\vx \in \R^n} \frac{\norm{\vlambda_i^{1/2-1/p}\mA\vx}_{p_i}^{2}}{\gnorm{\mLambda^{1/2-1/p}\mA\vx}{2}^2}}_{\widehat{\tau}_i(\mLambda^{1/2-1/p}\mA)}}^{1/p-1/2}.
\end{align*}
We may think of the quantity $\widehat{\tau}_i$ as a generalized \textit{leverage score}. Specifically, it upper bounds the contribution of the term $\|\vlambda_i^{1/2-1/p}\mA\vx\|_{p_i}^{2}$ to the objective $\gnorm{\mLambda^{1/2-1/p}\mA\vx}{2}^2$. The above calculation shows us that if we make $\widehat{\tau}_i/\vlambda_i$ small for all $i$, then we can get a tight relationship between $B_2$ and $B_p$. A slight weakening of the definition of the $\widehat{\tau}_i$ motivates the notion of a \textit{block Lewis overestimate} that we use in the remainder of the paper.

\begin{restatable}[Block Lewis overestimate]{definition}{blocklewisoverestimate}
\label{defn:block_lewis_overestimate}
Let $\tau_j(\mM)$ denote the leverage score of the $j$th row of $\mM$. Let \(F^{\star} > 0\). For $p > 0$ and $p_i > 0$, we say the probability measure $\vlambda$ and rounding $\mW$ form an $\Fstar$-block Lewis overestimate if for all $i \in [m]$, we have
\begin{align*}
    \frac{1}{\vlambda_i}\inparen{\sum_{j \in S_i} \inparen{\frac{\tau_j(\mW^{1/2}\mLambda^{1/2-1/p}\mA)}{\vw_j}}^{p_i/2}}^{2/p_i} \le \Fstar.
\end{align*}
\end{restatable}
Following the above argument, establishing a probability measure $\vlambda$ and a rounding matrix $\mW$ that form an $\Fstar$-block Lewis overestimate will imply
\begin{equation}\label{intro:blw_rounding}
  B_2 \subseteq \widehat{B_2}\quad \text{ and }\quad
  \begin{aligned}
    (\Fstar)^{1/2-1/r}B_r &\subset B_2 \subset B_r & \text{for all } r < 2 \\
    B_r &\subset B_2 \subset (\Fstar)^{1/2-1/r}B_r & \text{for all } r > 2
  \end{aligned}.
\end{equation}
With \eqref{intro:blw_rounding} in hand, we at least have enough reason to believe that establishing $(\vlambda, \mW)$ that form an $\Fstar$-block Lewis overestimate may yield the requisite control over $\log\cN(\widehat{B_2},\eta B_r)$. To actually get this, by the dual Sudakov inequality, we estimate $\norm{\vg}_{B_r}$ for $\vg \sim \cN(0,\mI_n)$. Doing so is a matter of applying again the fact that the concentration of Lipschitz functionals of Gaussian vectors is determined entirely by the Lipschitz parameter of the functional. 

To actually find $\vlambda$ and $\mW$ with a small value of $\Fstar$, we split into cases. When $p \ge 1$ and $p_1=\dots=p_m \ge 2$, we extract the relevant $\vlambda$ and $\mW$ from the analysis in the proof of \cite[Lemma 4.2]{jlls23}, which yields $\Fstar = n$. When $p_1=\dots=p_m=p\ge1/\log n$ or $p_1=\dots=p_m=2$ and $p\ge1/\log n$, we separately prove that we can find $\vlambda$ and $\mW$ satisfying \Cref{defn:block_lewis_overestimate}, again with $\Fstar = n$. Hence, in all cases, the control we get over $\log\cN(\widehat{B_2},\eta B_r)$ is essentially as good as what we get in the case where all the $S_i$ are singletons. 

\paragraph{Change-of-measures in functional analysis.} We note that other change-of-measure arguments are used throughout the study of finite-dimensional subspaces of $L_p$, as they are a very useful way to compare the $L_p$ norm to some other norm of interest. See the survey by \citet[Section 1.2]{js01} for more information.

\subsection{Prior results, related works, and connections}
\label{sec:related_works}

\paragraph{Relevance of matrix block norms.} We discuss the importance of the matrix block norm objective \eqref{eq:energy} to functional analysis, theoretical computer science, and data science, beyond our previous discussion of the MSN problem \eqref{eq:msn}.

In the special case where all the $p_i$ are equal to one another (call this value $q$), the set of $\vx$ for which $\gnorm{\mA\vx}{p} < \infty$ yields a subspace of a mixed $p,q$ norm space (sometimes notated as $\ell_p(\ell_q)$). Observe that \Cref{thm:one_shot_lewis} implies a finite-dimensional subspace embedding result for finite-dimensional subspaces of infinite-dimensional $L_p(\ell_q^r)$ (where $r$ is finite).  Spaces of the form $\ell_p(\ell_q)$ are widely studied in the geometric functional analysis and approximation theory communities; see, e.g., \cite{ps12,kv15,mu19,jkb22} and the references therein. 

As mentioned in those works, a central motivation for studying $\ell_p(\ell_q)$ is that they are natural testbeds with which to evaluate and further our understanding of the geometry of symmetric convex bodies in high dimensions. Consequently, studying block norm subspace embedding problems (\Cref{prob:main_problem}) is a fruitful direction through which to improve our geometric handle of subspaces of $\ell_p(\ell_q)$ and symmetric convex bodies in general. We note that understanding the correct $\mathrm{polylog}$ dependencies in $n$ for this problem typically requires new geometric insights. For instance, the necessity of additional polylogarithmic dependencies on the dimension $n$ is not even totally understood when $\abs{S_i} = 1$ and $p\neq 2$, and resolving them likely requires significant new geometric ideas \cite[Conjecture 2]{hrr22}. 

The matrix block norms are also ubiquitous in both theoretical computer science and data science. For example, the block norm objective has been studied in the context of hypergraph Laplacians. One recovers this by choosing $p=2$ and $p_1=\dots=p_m=\infty$; see the discussion in \cite[Section 1.2]{jls22} to see how to rewrite the hypergraph Laplacian in the form of \eqref{eq:energy}. Within data science, the block norms are used to encourage structured solutions to underdetermined linear systems (i.e., in a noiseless setting, we can set up and solve the convex optimization problem\footnote{This is similar to how basis pursuit can be seen as encouraging sparsity in a noiseless setting, while LASSO does so in the presence of noise \cite[Section 7.2]{Wainwright_2019}.}, ``find a vector $\vy$ in the affine space $\mB\vy = \vb$ minimizing $\gnorm{\vy}{p}^p$''); see, e.g., \cite{yy06,bach07,nie2010efficient,sfht13} and other applications mentioned in \cite{sra12}. As a concrete candidate application of our results to such settings, inspired by \cite{cd21,mmwy21}, we believe that our results can be used as subroutines to give runtime and query-efficient algorithms for active $\gnorm{\cdot}{p}$ regression when $p > 0$ and $p_1=\dots=p_m=2$ (generalizing the basis pursuit equivalent of the group Lasso objective) or when $p=2$ and $p_1,\dots,p_m \ge 2$. 

On a more conceptual level, we are optimistic that some of our results will be useful for designing faster algorithms for norm-constrained optimization problems. This is partly motivated by our discussion around the MSN problem \eqref{eq:msn} and is similar to how an improved geometric understanding of Lewis weights improved the iteration complexities for linear programming and $\ell_p$ regression \cite{ls19,jls21}.

\paragraph{Lewis weights for $\ell_p$ row sampling.} When each group $S_i$ has size $1$, notice that we have $\gnorm{\mA\vx}{p}^p = \sum_{i=1}^m \abs{\ip{\va_i,\vx}}^p = \norm{\mA\vx}_p^p$. Consequently, in this special case, satisfying (\ref{eq:main_objective}) is exactly equivalent to computing an \textit{$\ell_p$ subspace embedding} for $\mA$. There is a long line of work studying computing $\ell_p$ subspace embeddings using Lewis weights, starting with that of \citet{blm89}. For the details of this argument, see \cite[Section 7]{blm89} and \cite{sz01}. 


\paragraph{Sparsifying sums of norms.} The work perhaps most closely related to ours is \cite{jlls23}. There, the authors give existence results for sparse approximations to sums (of powers) of norms. It is easy to see that this is a more general problem than the one we study. However, this generality comes at a cost. In particular, the sparsity given by our \Cref{thm:one_shot_lewis} improves over theirs by a factor of $\psi_n\logv{\nfrac{n}{\eps}}^{\min(p-1,2)}$, where $\psi_n$ denotes the KLS ``constant'' in $n$ dimensions (which is currently $\sqrt{\log n}$, due to \citet{kla23}). See their Theorem 1.3 for more details. And, as mentioned earlier, we believe understanding \Cref{prob:main_problem} down to the correct polylogarithmic dependencies in $n$ is an important geometric question.

The authors also define the block Lewis weights as the natural generalization of the determinant-maximization program that \citet{sz01} use to prove the existence of Lewis's measure for all $p > 0$. They use this to give results for sparsification of sums of certain powers of arbitrary norms (obtaining a sparsity of $\sim n^{2-1/p}$ when $1 \le p \le 2$) and for \Cref{prob:main_problem} when the outer norm $p=2$ (obtaining a sparsity of $\sim n$). We note that the result they obtain when $p=2$ provides logarithmic-factor improvements over ours, as the main technical primitive they use is a chaining estimate developed by \citet{lee22} that meaningfully exploits the fact that the space of events is a subset of a $2$-uniformly convexity set. However, they did not address whether the block Lewis weights could yield to sparsification guarantees for \Cref{prob:main_problem}. Additionally, their construction of the block Lewis weights does not yield a change-of-measure that allows for sparsification when the inner norms $p_i \le 2$ or when the outer norm $p \le 1$.

\paragraph{Summary for sparsifying sums of norms.} See \Cref{table:comparison} for a comparison between our new results and a selection of the most relevant prior work on sparsifying sums of norms. We focus on results concerning \(\ell_p\)-norms specifically, although \cite{jlls23} has results for more general classes of norms.
By ``sampling'', we mean the work provides an analysis that shows how sampling according to some sampling probabilities gives a sparsifier of size $\widetilde{O}(\eps^{-2}n^{\max(1,p/2)})$ with good probability.
By ``fast computation'', we mean the work provides an algorithm to compute sampling probabilities with \(\text{polylog}(k, m,n )\) leverage score computations or linear system solves (or some other primitive that can be implemented in time $\widetilde{O}(\mathsf{nnz}(\mA) + n^{\omega})$). For works that only explicitly handle \(\abs{S_i}=1\), we leave \(p_1, \ldots, p_m\) blank because the choice of inner norms does not affect
the objective.
\begin{table}[H]
\centering
\begin{tabular}{l c c c c r}
\hline
\textbf{Block size} & \(p\) & \(p_1, \ldots, p_m\) & \textbf{Sampling} & \textbf{Fast computation} & \\
\hline
\(1\) & \(1 \le p < \infty \) &  & \checkmark & & \cite{blm89} \\
\(1\) & \(0 < p < 1\) &  & \checkmark & & \cite{sz01} \\
\(1\) & \(0 < p < 4 \) & & & \checkmark & \cite{cp15} \\
\(\ge 1\) & \(p = 2\) & \(p_1 = \cdots = p_m = \infty\) & \checkmark & & \cite{lee22} \\
\(\ge 1\) & \(p=2\) & \(p_1=  \cdots = p_m = \infty\) & \checkmark & \checkmark & \cite{jls22} \\
\(\ge 1\) & \(1 \le p < \infty\) & \(p_1, \ldots, p_m \ge 2\) & \checkmark & & \cite{jlls23} \\
\rowcolor{blue!5} \(\ge 1\) & \(1 \le p < \infty\) & \(p_1, \ldots, p_m \ge 2\) & \checkmark & & This work \\
\rowcolor{blue!5} \(\ge 1\) & \(\nfrac{1}{\log n} \le p < \infty\) & \(p_1, \ldots, p_m = p\) & \checkmark & \checkmark & This work \\
\rowcolor{blue!5} \(\ge 1\) & \(\nfrac{1}{\log n} \le p < \infty\) & \(p_1, \ldots, p_m = 2\) & \checkmark & \checkmark & This work \\
\rowcolor{blue!5} \(\ge 1\) & \(p=2\) & \(p_1, \ldots, p_m \ge 2\) & \checkmark & \checkmark & This work \\
\hline
\end{tabular}
\label{table:comparison}
\end{table}

\section{Preliminaries}
\label{sec:background}

In this section, we set up and review definitions and existing facts that will play crucial roles in our analyses. In Section \ref{sec:linalg_background}, we review material from linear algebra, and in Section \ref{sec:convex_geo_background}, we review material from convex geometry.

\subsection{Linear algebra background}
\label{sec:linalg_background}

We introduce a few definitions concerning \textit{leverage scores} (Definition \ref{defn:lev_score}). 
\begin{definition}
\label{defn:lev_score}
For a matrix $\mA \in \R^{m\times n}$, we let $\tau_i(\mA) \coloneqq \va_i^{\top}\inparen{\mA^{\top}\mA}^{-1}\va_i$ denote the \textit{leverage score} of row $\va_i$ with respect to the matrix $\mA$. When $\mA$ is clear from context, we omit it and simply write $\tau_i$ in place of $\tau_i(\mA)$.
\end{definition}

The following are well-known properties of leverage scores.

\begin{fact}
\label{fact:lev_score_facts}
For a matrix $\mA \in \R^{m\times n}$, we have:
\begin{itemize}
    \item $\sum_{i=1}^m \tau_i(\mA) = \rank{\mA}$;
    \item $\tau_i(\mA) = \max_{\vx\in\R^n\setminus\inbraces{0}} \frac{\abs{\ip{\va_i,\vx}}^2}{\norm{\mA\vx}_2^2}$ for all $i$;
    \item $0 \le \tau_i(\mA) \le 1$;
    \item For any positive constant $C$, we have $\tau_i(C\mA) = \tau_i(\mA)$ for all $i$.
\end{itemize}
\end{fact}


We will also need the following fact relating the leverage scores of a matrix $\mA$ to its singular value decomposition.

\begin{fact}
\label{fact:whitening}
Let $\mA\in\R^{n\times m}$ and $\mU\mSigma\mV^{\top}$ be a singular value decomposition for $\mA$, where $\mU\in\R^{m\times n}$ and $\mSigma,\mV\in\R^{n\times n}$. Then, $\tau_i(\mA) = \norm{\vu_i}_2^2$.
\end{fact}
\begin{proof}[Proof of \Cref{fact:whitening}]
To understand why this equality might hold, observe that we can think of $\mU$ as the resulting matrix from applying a statistical whitening transform to $\mA$. More precisely, recall that
\begin{align*}
    \tau_i(\mA) &= \va_i^{\top} \inparen{\mV\mSigma^2\mV^{\top}}^{-1}\va_i^{\top} = \va_i^{\top}\inparen{\sum_{j=1}^n \frac{1}{\sigma_i^2} \cdot \vv_j\vv_j^{\top}}\va_i^{\top} = \sum_{j=1}^n \frac{1}{\sigma_j^2}\ip{\va_i,\vv_j}^2.
\end{align*}
We now calculate $\ip{\va_i,\vv_j}$. Notice that
\begin{align*}
    \ip{\va_i,\vv_j} = \ip{\vv_j, \ve_i^{\top}\sum_{j'=1}^n \sigma_{j'}\mU\ve_{j'}\vv_{j'}^{\top}} = \ip{\vv_j, \sum_{j'=1}^n \sigma_{j'}\ip{\ve_i,\mU\ve_{j'}}\vv_{j'}^{\top}} = \sigma_{j}\ip{\ve_i,\mU\ve_j}.
\end{align*}
Substituting this back in gives
\begin{align*}
    \tau_i(\mA) = \sum_{j=1}^n \ip{\ve_i,\mU\ve_j}^2 = \norm{\vu_i}_2^2.
\end{align*}
This concludes the proof of \Cref{fact:whitening}.
\end{proof}

\subsection{Convex geometry background}
\label{sec:convex_geo_background}

In this subsection, we review foundational facts regarding convex geometry we use throughout the remainder of this paper. 

We will need the notions of covering and entropy numbers. 
\begin{definition}[Covering numbers {\cite[p. 69]{rothvoss}}]
\label{defn:covering_numbers}
Let $X, Y \subset \R^n$. The covering number $\cN(X,Y)$ is the minimum number of translates of $Y$ required to cover $X$. Formally, we have
\begin{align*}
    \cN(X,Y) \coloneqq \min\inbraces{N \in \N \suchthat \text{ there exists } \vx_1,\dots,\vx_N \in \R^n \text{ such that } X \subseteq \bigcup_{i=1}^N \inparen{\vx_i + Y}}.
\end{align*}
\end{definition}

\begin{definition}[Entropy numbers {\cite[Definition 2.1]{vh15}}]
\label{defn:entropy_numbers}
Let $X, Y \subset \R^n$. The entropy number $e_N(X,Y)$ is the minimum radius $\eta$ such that $\log\cN(X,\eta \cdot Y) \le 2^N$.
\end{definition}

Sometimes, when writing $e_N$, we will write $e_N(X,\norm{\cdot})$ for some quasi-norm $\norm{\cdot}$. Here, we take $Y$ to be the object formed by the unit ball of $\norm{\cdot}$.

Finally, we state the dual Sudakov inequality.
\begin{fact}[Dual Sudakov inequality, due to \citet{ptj87}]
\label{fact:sudakov}
For a symmetric convex body $K \subset \R^n$, define
\begin{align*}
    \norm{\vx}_K \coloneqq \inf\inbraces{t > 0 \suchthat \vx/t \in K}.
\end{align*}
Let $K$ be a symmetric convex body in $\R^n$. We have the below.
\begin{align}
    \log\cN(B_2^n, \eta \cdot K) &\lesssim \eta^{-2} \cdot \exvv{\vg\sim\cN(0,\mI_n)}{\norm{\vg}_K}^2.\label{eq:sudakov_dual}
\end{align}
\end{fact}

\subsection{Probability background}

In this subsection, we review a few facts about subgaussian random variables. These are mostly derived from the presentation of \citet{vershynin_2018}.

\begin{definition}[$\subgnorm{\cdot}$ and subgaussian random variable {\cite[Definition 2.5.6]{vershynin_2018}}]
\label{defn:subgaussian}
Let $X$ be a random variable. Define $\subgnorm{X} \coloneqq \inf\inbraces{t > 0 \suchthat \exv{\expv{\nfrac{X^2}{t^2}}} \le 2}$. If $\subgnorm{X} < \infty$, we say $X$ is subgaussian.
\end{definition}

\begin{fact}[Properties of subgaussian random variables {\cite[Proposition 2.5.2]{vershynin_2018}}]
\label{fact:subgaussian_properties}
The following properties equivalently characterize a subgaussian random variable $X$ up to constants:
\begin{itemize}
    \item For all $t \ge 0$, $\prv{\abs{X} \ge t} \le 2\expv{-\frac{t^2}{\subgnorm{X}^2K_1^2}}$;
    \item For all $r \ge 1$, $\exv{\abs{X}^r} \lesssim r^{r/2}$;
    \item For all $\lambda$ such that $\abs{\lambda} \lesssim \subgnorm{X}^{-1}$, we have $\exv{\expv{\lambda^2X^2}} \le \expv{\subgnorm{X}^2\lambda^2}$.
\end{itemize}
\end{fact}

\begin{fact}[Maximum of subgaussian random variables {\cite[Exercise 2.5.10]{vershynin_2018}}]
\label{fact:max_subgaussians}
Let $X_1,\dots,X_N$ be a sequence of (not necessarily independent) subgaussian random variables. Then
\begin{align*}
    \exv{\max_{i \in [N]} \abs{X_i}} \lesssim \max_{i \in [N]} \subgnorm{X_i}\sqrt{\log N}.
\end{align*}
\end{fact}

\begin{fact}[Decentering]
\label{fact:uncentering}
We have
\begin{align*}
    \subgnorm{X} \lesssim \subgnorm{X - \exv{X}} + \exv{\abs{X}}.
\end{align*}
\end{fact}
\begin{proof}[Proof of \Cref{fact:uncentering}]
By the triangle inequality, we get
\begin{align*}
    \subgnorm{X} \le \subgnorm{X - \exv{X}} + \subgnorm{\exv{X}} \lesssim \subgnorm{X-\exv{X}} + \exv{\abs{X}},
\end{align*}
which is exactly the statement of \Cref{fact:uncentering}.
\end{proof}

\begin{fact}[Lipschitz functionals of Gaussians are subgaussian {\cite[Theorem 5.2.2]{vershynin_2018}}]
\label{fact:lip_concentration}
If $\vg \sim \cN(0,\mI_n)$ and if $\myfunc{f}{\R^d}{\R}$, then
\begin{align*}
    \subgnorm{f(\vg) - \exv{f(\vg)}} \lesssim \norm{f}_{\mathsf{Lip}}.
\end{align*}
\end{fact}

\section{Covering number estimates}
\label{sec:coverings}

In this section, we develop our metric entropy estimates. It will be helpful to keep in mind the context and outline from \Cref{sec:overview}.

\subsection{Notation and general formula}

We begin with some definitions that are necessary for our results.

\begin{definition}[Block-constant diagonal matrix]
\label{defn:block_constant}
We say that a vector $\vm \in \R^k$ and corresponding diagonal matrix $\mM\in\R^{k\times k}$ is block-constant or ``constant down the blocks'' if for every $j_1,j_2 \in S_i$, we have $\vm_{j_1}=\vm_{j_2}$.
\end{definition}

For instance, when we define a probability measure $\vlambda \in \R^{m}$ over $[m]$, we will find it useful to extend it to a block-constant diagonal matrix $\mLambda\in\R^{k \times k}$.

\begin{definition}[Rounding matrix]
\label{defn:rounding_matrix}
For a probability measure $\vlambda$ over $[m]$, we say that a positive diagonal matrix $\mW \in \R^{k \times k}$ rounds the measure matrix $\mLambda \in \R^{k\times k}$ if for all $\vx\in\R^n$ we have $\norm{\mW^{1/2}\mLambda^{1/2-1/p}\mA\vx}_2 \le \gnorm{\mLambda^{1/2-1/p}\mA\vx}{2}$. We also denote
\begin{align*}
    \widehat{B_2} \coloneqq \inbraces{\vx\in\R^n \suchthat \norm{\mW^{1/2}\mLambda^{1/2-1/p}\mA\vx}_2 \le 1}.
\end{align*}
\end{definition}

The $\widehat{B_2}$ defined in \Cref{defn:rounding_matrix} is the linear transformation of the Euclidean ball that we will ``pass through'' to get our covering number estimates (recall \eqref{eq:intro_split_covering}).

Next, recall our notion of measure overestimates. This is a generalization of prior definitions of Lewis measure overestimates (see e.g. \cite[Definition 2.4]{jls21}, \cite[Definition 2.3]{wy22}) and group leverage score overestimates (\cite[Definition 1.1]{jls22}).

\blocklewisoverestimate*

For example, observe that when all the $S_i$ have size $1$, then \Cref{defn:block_lewis_overestimate} corresponds to standard definitions of Lewis weight overestimates, and there exist weights such that $\Fstar \lesssim n$. Furthermore, we will see that there exist a $\vlambda$ and $\vw$ such that $\Fstar = n$ (it will follow from \Cref{lemma:block_instantiation_small}).

Next, we define the vector $\valpha$, whose entries capture a notion of group importance.

\begin{definition}
\label{defn:alpha}
Let $\vlambda$ be a probability measure over $[m]$ and let $\mW$ be a rounding matrix for $\mLambda$ (\Cref{defn:rounding_matrix}). If $p_1,\dots,p_m \ge 2$, then let $\valpha \in \R^{m}$ be the vector such that for all $i \in [m]$, we have
\begin{align*}
    \valpha_i^p \coloneqq \vlambda_i^{1-p/2}\inparen{\sum_{j\in S_i}\inparen{\frac{\tau_j\inparen{\mW^{1/2}\mLambda^{1/2-1/p}\mA}}{\vw_j}}^{p_i/2}}^{p/p_i}.
\end{align*}
Equivalently, if $\mU$ is a matrix whose columns consist of the left singular vectors of $\mW^{1/2}\mLambda^{1/2-1/p}\mA$, and if we denote by $\vu_j$ the $j$th row of $\mU$ and let $\vf_j \coloneqq \vlambda_{i}^{-1/2}\vw_j^{-1/2}\vu_j$, then by \Cref{fact:whitening}, we may also write
\begin{align*}
    \valpha_i^p &\coloneqq \vlambda_i\inparen{\sum_{j \in S_i} \norm{\vf_j}_2^{p_i}}^{p/p_i}.
\end{align*}
On the other hand, if $p_1=\dots=p_m=p<2$, then let $\widehat{\vlambda}$ be a probability measure over $[k]$ and let $\widehat{\valpha}\in\R^k$ be defined as above accordingly. Finally, let $\valpha\in\R^m$ be such that
\begin{align*}
    \valpha_i^p \coloneqq \inparen{\sum_{j \in S_i}\widehat{\valpha}_j^p}^{1/p}.
\end{align*}
\end{definition}

To help ground \Cref{defn:alpha}, notice that combining \Cref{defn:block_lewis_overestimate} with \Cref{defn:alpha} gives us, for $p_1,\dots,p_m\ge 2$ and $p_1,\dots,p_m=p < 2$, respectively,
\begin{equation}
\begin{aligned}
    \valpha_i^p &\le \vlambda_i^{1-p/2}\inparen{\vlambda_i\Fstar}^{p/2} = \vlambda_i\inparen{\Fstar}^{p/2} \\
    \valpha_i^p &= \sum_{j \in S_i} \widehat{\valpha}_j^p = \sum_{j \in S_i} \widehat{\vlambda}_j^{1-p/2}\tau_j\inparen{\widehat{\mLambda}^{1/2-1/p}\mA}^{p/2} \le \sum_{j \in S_i}\widehat{\vlambda}_j\inparen{\Fstar}^{p/2} = \inparen{\sum_{j\in S_i} \widehat{\vlambda_j}}\inparen{\Fstar}^{p/2},
\end{aligned}
\label{eq:alpha_to_fstar}
\end{equation}
and that when $\Fstar \le 2n$ (say), we get $\norm{\valpha}_p^p \lesssim n^{p/2}$. Thus, at least when $p\ge 2$, we can think of $\norm{\valpha}_p^p$ as encoding the sparsity we should expect when we sample with probabilities proportional to the $\valpha_i$. Although this does not quite work when $p < 2$, a minor modification of it will.

\begin{definition}[Notation for unit balls and norms under change-of-measure]
\label{defn:unit_ball}
Let $\vlambda$ be a probability measure over $[m]$ and $\mLambda\in\R^{k\times k}$ be its corresponding block-constant diagonal matrix (\Cref{defn:block_constant}). For any $r > 0$, and $\vy\in\R^k$ we define
\begin{align*}
    \gnorml{\vy}{r} \coloneqq \inparen{\sum_{i=1}^m \vlambda_i\inparen{\sum_{j \in S_i} \abs{\vy_j}^{p_i}}^{r/p_i}}^{1/r}.
\end{align*}
We also define
\begin{align*}
    B_r \coloneqq\inbraces{\vx \in\R^n \suchthat \gnorml{\mLambda^{-1/p}\mA\vx}{r} \le 1}.
\end{align*}
\end{definition}

From \Cref{defn:unit_ball}, it is easy to verify that $\gnorml{\mLambda^{-1/p}\mA\vx}{p} = \gnorm{\mA\vx}{p}$. Indeed, we have
\begin{align*}
    \gnorml{\mLambda^{-1/p}\mA\vx}{p}^p = \sum_{i=1}^m \vlambda_i\norm{\mLambda_{S_i}^{-1/p}\mA_{S_i}\vx}_{p_i}^{p} = \sum_{i=1}^m \norm{\mA_{S_i}\vx}_{p_i}^p = \gnorm{\mA\vx}{p}^p.
\end{align*}
We will also require the crucial property that $\gnorml{\vy}{r}$ is log-convex in $1/r$. To see this, note that the vector in $\R^m$ formed by calculating all the inner norms $p_1,\dots,p_m$ is constant regardless of the outer norm, and then we can use the fact that for a fixed measure $\vmu$, the $\ell_r^m(\vmu)$ norms are log-convex in $1/r$.

We now have the language to state the main result of this section, \Cref{thm:general_covering}.

\begin{theorem}
\label{thm:general_covering}
Let $\vlambda$ be a probability measure and $\mW$ be a rounding matrix (\Cref{defn:rounding_matrix}) so that $\vlambda$ and $\mW$ form an $\Fstar$-block Lewis overestimate (\Cref{defn:block_lewis_overestimate}). Suppose at least one of the following holds:
\begin{itemize}
    \item $p > \frac{1}{\log n}$ and $\abs{S_1}=\dots=\abs{S_m}=1$;
    \item $p = p_1 = \dots = p_m$ and $p < 2$;
    \item $p \ge 1$ and $p_1,\dots,p_m \ge 2$;
    \item $p_1=\dots=p_m=2$ and $1/\log n \le p < \infty$.
\end{itemize}
If $H \ge 1$ is such that the sampling probabilities $\vrho_i$ satisfy $H\vrho_i \ge \valpha_i^p/\norm{\valpha}_p^p$ for all $i \in [m]$, and if we write $p^{\star}\coloneqq\max\inbraces{1,\max_{i}\min\inbraces{p_i,\log\abs{S_i}}}$, then (recall \Cref{defn:sampling_polyhedron} for the definition of $\polynorm{\cdot}$)
\begin{align*}
    \log\cN\inparen{B_p, \eta\inbraces{\vx\in\R^n \suchthat \polynorm{\vx} \le 1}} \lesssim \eta^{-\min(2,p)} \cdot H^{2/\max(p,2)}\cdot C(p)p^{\star}\Fstar\log\max\inbraces{\mtilde,\Fstar},
\end{align*}
where $C(p)$ is a constant that only depends on $p$.
\end{theorem}

Although \Cref{thm:general_covering} is stated abstractly, we will see that there exists a convenient instantiation for all the parameters stated.

\begin{corollary}
\label{corollary:specific_cover}
In the same cases as in \Cref{thm:general_covering}, there exists a probability measure $\vlambda$ over $[m]$ and a rounding $\mW$ for which in the same setting as \Cref{thm:general_covering}, we have
\begin{align*}
    \log\cN\inparen{B_p, \eta\inbraces{\vx\in\R^n \suchthat \polynorm{\vx} \le 1}} \lesssim_p \eta^{-\min(2,p)}\cdot{\underset{i \in S}{\max}\ \min\inbraces{p_i,\log\abs{S_i}}n\log\mtilde}.
\end{align*}
\end{corollary}

\subsection{Block Lewis weights}
\label{sec:blw_defs_covering}

For the sake of motivation, let us first prove \Cref{corollary:specific_cover} given \Cref{thm:general_covering}. We first need \Cref{lemma:block_lewis}, which is derived from the \textit{block Lewis weights} of \citet{jlls23}.

For a nonnegative diagonal matrix $\mV$, let $\beta_i(\mV) \coloneqq \inparen{\sum_{j \in S_i} \inparen{\va_j^{\top}(\mA^{\top}\mV\mA)^{-1}\va_j}^{p_i/2}}^{1/p_i}$. We call the $\beta_i(\mV)^p$ the \textit{block Lewis weights}.

\begin{lemma}
\label{lemma:block_lewis}
If $p_i \in [2, \infty]$ and $p \in [1, \infty)$, then there exist diagonal $\mV, \mLambda \in \R^{k \times k}$ such that $\vlambda$ is a probability measure over $[m]$ and the corresponding \(\mLambda\in\R^{k\times k}\) is constant on the blocks, then
\(
\sum_{i=1}^m \beta_i(\mV)^p = n
\) and for all $\vx\in\R^n$,
\begin{align*}
    \norm{\mV^{1/2}\mA\vx}_2 \le n^{1/2-1/p}\gnorm{\mLambda^{1/2-1/p}\mA\vx}{2} \le n^{\max(0,1/2-1/p)}\gnorm{\mA\vx}{p}
\end{align*}
\end{lemma}

\begin{proof}[Proof of \Cref{lemma:block_lewis}]
The reader familiar with the work of \citet{jlls23} will notice that \Cref{lemma:block_lewis} is a strengthened variant of Lemma 4.2 from that work. 

Indeed, consider the context of the proof of Lemma 4.2 from \cite{jlls23}. There, notice that $\mW$ is initially chosen so that $\sum_{i=1}^m \beta_i(\mW)^p = n$ and \(\mU = (\mA^\top \mW \mA)^{-1/2}\). We choose $\mV$ in the same way. Next, using their choice of $\vu$, we have $\norm{\vu_{S_i}}_{p_i}^{p-2} = \inparen{\beta_i(\mV)^{p}}^{1-2/p}$. 

Restating (4.8) from \cite{jlls23} in our notation, we have for all $\vx\in\R^n$ that
\begin{align*}
    \norm{\mV^{1/2} \mA \vx}_2^2 \le \sum_{i=1}^m \norm{\vu_{S_i}}_{p_i}^{p-2} \norm{\mA_{S_i} \vx}_{p_i}^2.
\end{align*}
For each \(i \in [m]\) let \(\vlambda_i = \frac{\beta_i(\mV)^p}{n}\), so that \(\vlambda\) is a probability measure.
Then
\begin{align*}
    \norm{\mV^{1/2} \mA \vx}_2^2 &\le n^{1-2/p} \sum_{i=1}^m \vlambda_i^{1-2/p} \norm{\mA_{S_i} \vx}_{p_i}^2.
\end{align*}
Let \(\mLambda\) be a \(k \times k\) diagonal matrix, where for every \(i \in [m]\) and \(j \in S_i\),
we define \(\mLambda_{jj} = \vlambda_i\).
Because \(\vlambda_i^{1-2/p} \norm{\mA_{S_i} \vx}_{p_i}^2 = \norm{(\mLambda^{1/2-1/p} \mA)_{S_i} \vx}_{p_i}^2\), we obtain
\begin{equation}\label{eq:vax_atmost_nlax}
    \norm{\mV^{1/2}\mA\vx}_2 \le n^{1/2-1/p}\gnorm{\mLambda^{1/2-1/p}\mA\vx}{2}.
\end{equation}
Since $p$-norms taken with respect to a probability measure are increasing in $p$ we immediately get for all $p \ge 2$ that
\begin{align*}
    \norm{\mV^{1/2}\mA\vx}_2 &\stackrel{\eqref{eq:vax_atmost_nlax}}{\le} n^{1/2-1/p}\gnorm{\mLambda^{1/2-1/p}\mA\vx}{2} = n^{1/2-1/p}\gnorml{\mLambda^{-1/p}\mA\vx}{2} \\
    &\le n^{1/2-1/p}\gnorml{\mLambda^{-1/p}\mA\vx}{p} = n^{1/2-1/p}\gnorm{\mA\vx}{p}.
\end{align*}
The case where $p \le 2$ follows from the ``$1 \le q \le 2$'' subcase of the proof of Lemma 4.2 from \cite{jlls23}, which yields
\begin{align*}
    \norm{\mV^{1/2}\mA\vx}_2 \stackrel{\eqref{eq:vax_atmost_nlax}}{\le} n^{1/2-1/p}\gnorm{\mLambda^{1/2-1/p}\mA\vx}{2} \le \gnorm{\mA\vx}{p}.
\end{align*}
We therefore conclude the proof of \Cref{lemma:block_lewis}.
\end{proof}

We use \Cref{lemma:block_lewis} to give an instantiation for the parameters in \Cref{thm:general_covering}.

\begin{lemma}
\label{lemma:block_instantiation_small}
Let \(\mV, \mLambda\) be the matrices from \Cref{lemma:block_lewis} and let $\mU$ and $\vf_j$ be as defined in \Cref{defn:alpha}. Let $p \ge 1$ and $p_i \ge 2$ for all $i \in [m]$. If we choose \(\mW\) such that
\begin{align*}
    \frac{\mV^{1/2}}{n^{1/2-1/p}} = \mW^{1/2}\mLambda^{1/2-1/p},
\end{align*}
then:
\begin{itemize}
    \item $\norm{\mW^{1/2}\mLambda^{1/2-1/p}\mA\vx}_2 \le \gnorm{\mLambda^{1/2-1/p}\mA\vx}{2}$;
    \item for all $i$, $\frac{\valpha_i}{n^{1/2-1/p}} = \beta_i(\mV)$;
    \item $\norm{\valpha}_p^p = n^{p/2}$;
    \item for all $i$, $\inparen{\sum_{j \in S_i} \norm{\vf_j}_2^{p_i}}^{1/p_i} = \norm{\valpha}_p = n^{1/2}$.
    \item The rounding matrix $\mW$ and measure $\vlambda$ are an $\Fstar$-block Lewis overestimate (\Cref{defn:block_lewis_overestimate}) with $\Fstar = n$.
\end{itemize}
\end{lemma}
\begin{proof}[Proof of \Cref{lemma:block_instantiation_small}]
The first property follows immediately from \Cref{lemma:block_lewis}. Using \Cref{fact:whitening}, notice that
\begin{align*}
    \va_j^{\top}(\mA^{\top}\mV\mA)^{-1}\va_j = \frac{\tau_j(\mV^{1/2}\mA)}{\vv_j} = \frac{\tau_j(\mW^{1/2}\mLambda^{1/2-1/p}\mA)}{n^{1-2/p}\vw_j\vlambda_i^{1-2/p}} = \frac{\norm{\vu_j}_2^2}{n^{1-2/p}\vw_j\vlambda_i^{1-2/p}} = \frac{\vlambda_i^{2/p}\norm{\vf_j}_2^2}{n^{1-2/p}},
\end{align*}
so after substituting into the formula for $\beta_i(\mV)$,
\begin{align*}
    \beta_i(\mV) &= \inparen{\sum_{j \in S_i} \inparen{\va_j^{\top}(\mA^{\top}\mV\mA)^{-1}\va_j}^{p_i/2}}^{1/p_i} = \inparen{\sum_{j \in S_i} \inparen{\frac{\vlambda_i^{2/p}\norm{\vf_j}_2^2}{n^{1-2/p}}}^{p_i/2}}^{1/p_i} \\
    &= \frac{\vlambda_i^{1/p}\inparen{\sum_{j \in S_i} \norm{\vf_j}_2^{p_i}}^{1/p_i}}{n^{1/2-1/p}} = \frac{\valpha_i}{n^{1/2-1/p}},
\end{align*}
where the last equality follows from the formula for $\valpha$ stated in \Cref{thm:general_covering}. This also implies that
\begin{align*}
    \norm{\valpha}_p^p = \sum_{i=1}^m \valpha_i^p = \sum_{i=1}^p \beta_i(\mV)^pn^{p/2-1} = n^{p/2}.
\end{align*}
Finally, observe that the above calculation shows that $\vlambda_i \propto \valpha_i^p$, since we have defined $\vlambda_i \propto \beta_i(\mV)^p$ and we have just seen that $\beta_i(\mV)^p \propto \valpha_i^p$. This means we can write $\vlambda_i = \valpha_i^p/\norm{\valpha}_p^p$. Using this, we have
\begin{align*}
    \valpha_i = \vlambda_i^{1/p}\inparen{\sum_{j \in S_i} \norm{\vf_j}_2^{p_i}}^{1/p_i} = \frac{\valpha_i}{\norm{\valpha}_p} \cdot \inparen{\sum_{j \in S_i} \norm{\vf_j}_2^{p_i}}^{1/p_i}.
\end{align*}
After rearranging, we have
\begin{align*}
    \inparen{\sum_{j \in S_i} \norm{\vf_j}_2^{p_i}}^{1/p_i} = \norm{\valpha}_p = n^{1/2},
\end{align*}
and so we may take $\Fstar = n$. This concludes the proof of \Cref{lemma:block_instantiation_small}.
\end{proof}

We now handle the cases that are not covered by the block Lewis weight construction of \citet{jlls23}.

\begin{lemma}
\label{lemma:block_instantiation_alleq}
If $0 < p_1=\dots=p_m=p < 2$ or if $p_1=\dots=p_m=2$ and $1/\log m \le p < \infty$, then there exists a probability measure $\widehat{\vlambda}$ over $[k]$ and corresponding $\widehat{\valpha} \in \R^k$ such that $\widehat{\vlambda}$ is an $\Fstar$-block Lewis overestimate for $\Fstar=n$.
\end{lemma}
\begin{proof}
For the case where $0 < p_1=\dots=p_m=p < 2$, we simply use the fact that Lewis's measure tells us that there exists a measure $\widehat{\vlambda}$ such that
\begin{align*}
    \frac{\tau_j\inparen{\widehat{\mLambda}^{1/2-1/p}\mA}}{\widehat{\vlambda}_j} \le n.
\end{align*}
In the other case, we will see later that the guarantee of a natural contraction mapping (\Cref{alg:blw_lewismap} and \Cref{lemma:contract_alg}) imply that $\mW = \mI_k$ and the resulting $\vlambda$ form an $n$-block Lewis overestimate, thereby concluding the proof of \Cref{lemma:block_instantiation_alleq}.
\end{proof}

\Cref{lemma:block_instantiation_small} and \Cref{lemma:block_instantiation_alleq} easily imply \Cref{corollary:specific_cover}.

\begin{proof}[Proof of \Cref{corollary:specific_cover}]
We combine \Cref{thm:general_covering} with the instantiations in \Cref{lemma:block_instantiation_small} and \Cref{lemma:block_instantiation_alleq}, directly yielding \Cref{corollary:specific_cover}.
\end{proof}

In light of \Cref{corollary:specific_cover}, the goal of the remainder of this section is to prove \Cref{thm:general_covering}. 

It will be useful to consider a corresponding change-of-basis that arises from our setting of $\vlambda$. Let $\mU\mSigma\mV^{\top}$ be a singular value decomposition of $\mW^{1/2}\mLambda^{1/2-1/p}\mA$ where $\mU \in \R^{m \times n}$ and $\mSigma, \mV \in \R^{n \times n}$. Let $\mR$ be the invertible matrix $\mV\mSigma^{-1}$ (we assume without loss of generality that $\rank{\mA} = n$, and it is easy to extend the results of this section to the case where $\rank{\mA} < n$). We take $\mR$ as our change-of-basis matrix. Using this, it is easy to see that $\mW^{1/2}\mLambda^{1/2-1/p}\mA\mR = \mU$ consists of orthonormal columns. Furthermore, we have $\mLambda^{-1/p}\mA\mR = \mW^{-1/2}\mLambda^{-1/2}\mU$.
 
\subsection{Covering numbers for \texorpdfstring{$0 < p < 2$}{0 < p < 2}}

The goal of this section is to prove \Cref{lemma:covering_two_to_one_clean} under the notion of overestimate given by \Cref{defn:block_lewis_overestimate}. 

We are now ready to state the main result of this subsection.

\begin{lemma}
\label{lemma:covering_two_to_one_clean}
Let $\vlambda$ and $\vw$ be such that they form an $\Fstar$-block Lewis overestimate. Then,
\begin{align*}
    \log\cN(B_p,\eta \widehat{B_2}) \lesssim \eta^{-\frac{2p}{2-p}} \cdot C(p)\max_{i}\min\inparen{p_i,\log\abs{S_i}}\Fstar\log\Fstar,
\end{align*}
where $C(p)$ is a constant that only depends on $p$.
\end{lemma}

The goal of the rest of this subsection is to prove \Cref{lemma:covering_two_to_one_clean}. We follow the outline detailed in \Cref{sec:overview_measure}. In short, our plan is the following:
\begin{enumerate}
    \item We first reduce bounding $\log\cN(B_p, \eta\widehat{B_2})$ to bounding $\log\cN(\widehat{B_2}, \delta_h B_r)$ for all $h \ge 0$ and appropriate choices of $r$ and $\delta_h$.
    \item We then control each term $\log\cN(\widehat{B_2}, \delta_h B_r)$. To do so, we will apply the dual Sudakov inequality (\Cref{fact:sudakov}, \eqref{eq:sudakov_dual}). To actually estimate $\E\norm{\vg}_{B_r}$ where $\vg \sim \cN(0,\mI_n)$, we need to prove that every resulting summand of the form $\norm{\cdot}_{p_i}$ is subgaussian with a parameter that only depends on $p_i$. To do so, we exploit the fact that these summands are Lipschitz and then apply \Cref{fact:lip_concentration}.
    \item We finally assemble all the previous pieces together to get the desired handle on $\log\cN(B_p, \eta\widehat{B_2})$. 
\end{enumerate}

\subsubsection{Reduction to bounding \texorpdfstring{$\log\cN(\widehat{B_2},\eta B_r)$}{log N(\^{B2}, n Br)}}

As stated in \Cref{sec:overview_measure}, we begin with reducing the calculation of $\log\cN(B_p, \eta \widehat{B_2})$ to calculating $\log\cN(\widehat{B_2}, \eta B_r)$ (for a different $\eta$).

\begin{lemma}
\label{lemma:poor_mans_dualization}
Let $\theta$ and $r$ be such that $r=(2-\theta)p/(p-\theta)$. Define
\begin{align*}
    \delta_h \coloneqq \inparen{\frac{8^{h+1}\eta}{2\cdot 8^{2/\theta}}}^{\frac{\theta}{2-\theta}} = \eta^{\frac{\theta}{2-\theta}}\cdot 8^{(h+1)\cdot\frac{\theta}{2-\theta}} \cdot \inparen{2\cdot 8^{2/\theta}}^{-\frac{\theta}{2-\theta}}
\end{align*}
Then, we have
\begin{align*}
     \log\cN(B_p,\eta \widehat{B_2}) &\le \sum_{h\ge 0} \log\cN(\widehat{B_2},\delta_hB_r).
\end{align*}
\end{lemma}
\begin{proof}[Proof of \Cref{lemma:poor_mans_dualization}]
For $h \in \N_{\ge 0}$, let $\cN_h$ be a maximal subset of $B_p$ such that for any two distinct elements $\vz_1, \vz_2 \in B_p$, we have $\norm{\mW^{1/2}\mLambda^{-1/p}\mA(\vz_1-\vz_2)}_{2(\vlambda)} \ge 8^h\eta$ (where by $\norm{\cdot}_{p(v\lambda)}$ we mean the $\ell_p$ norm taken with respect to the measure given by $\vlambda$). This yields $\abs{\cN_h} \ge \cN(B_p, 8^h \eta \widehat{B_2})$.

Next, since for every $h$ there are $\vz_i \in B_p$ for which $B_p \subseteq \bigcup_{i=1}^{\cN(B_p,8^{h+1}\eta \widehat{B_2})} \inbraces{\vz_i + 8^{h+1}\eta \widehat{B_2}}$, for every $h$ there must exist a $\zstar_h \in B_p$ for which
\begin{align*}
    \abs{\inbraces{\zstar_h + 8^{h+1} \eta \widehat{B_2}} \cap \cN_h} \ge \frac{\abs{\cN_h}}{\cN(B_p,8^{h+1} \eta \widehat{B_2})} \ge \frac{\cN(B_p,8^h \eta \widehat{B_2})}{\cN(B_p,8^{h+1} \eta \widehat{B_2})}.
\end{align*}
Let
\begin{align*}
    \cL_h \coloneqq \inbraces{\frac{\vz-\zstar_h}{8^{h+1} \eta} \suchthat \vz \in \inbraces{\zstar_h + 8^{h+1} \eta \widehat{B_2}} \cap \cN_h}
\end{align*}
from which we get by the sub-triangle inequality that
\begin{align*}
    \norm{\mW^{1/2}\mLambda^{-1/p}\mA\vz}_{2(\vlambda)} \le 1 \text{ and } \gnorml{\mLambda^{-1/p}\mA\vz}{p} \le \frac{\max\inbraces{2^{1/p},2}}{8^{h+1} \eta} \text{ for any } \vz \in \cL_h
\end{align*}
and
\begin{align*}
    \norm{\mW^{1/2}\mLambda^{-1/p}\mA\inparen{\vz_1-\vz_2}}_{2(\vlambda)} \ge \frac{1}{8} \text{ for any distinct } \vz_1,\vz_2 \in \cL_h.
\end{align*}
We now apply an interpolation estimate. Let $\vz_1$ and $\vz_2$ be distinct elements from $\cL_h$, set $0<\theta<2$ and $r=(2-\theta)p/(p-\theta)$, and observe that $\theta = p(r-2)/(r-p)$ and 
\begin{align*}
    \frac{1}{8^2} &\le \norm{\mW^{1/2}\mLambda^{-1/p}\mA\inparen{\vz_1-\vz_2}}_{2(\vlambda)}^2 \\
    &\le \gnorml{\mLambda^{-1/p}\mA\inparen{\vz_1-\vz_2}}{2}^2 \\
    &\le \gnorml{\mLambda^{-1/p}\mA\inparen{\vz_1-\vz_2}}{p}^{\theta}\cdot\gnorml{\mLambda^{-1/p}\mA\inparen{\vz_1-\vz_2}}{r}^{2-\theta} \\
    &\le \inparen{\frac{\max\inbraces{2^{1/p},2}}{8^{h+1}\eta}}^{\theta}\gnorml{\mLambda^{-1/p}\mA\inparen{\vz_1-\vz_2}}{r}^{2-\theta} 
\end{align*}
which means that after rearranging we have
\begin{align*}
    \gnorml{\mLambda^{-1/p}\mA\inparen{\vz_1-\vz_2}}{r} \ge \inparen{\inparen{\frac{8^{h+1}\eta}{\max\inbraces{2^{1/p},2}}}^{\theta} \cdot \frac{1}{8^2}}^{\frac{1}{2-\theta}} \ge \delta_h.
\end{align*}
The above argument gives
\begin{align*}
    \log\cN(\widehat{B_2},\delta_h B_r) \ge \log\abs{\cL_h} \ge \log\cN\inparen{B_p,8^h\eta\widehat{B_2}}-\log\cN\inparen{B_p,8^{h+1}\eta\widehat{B_2}}.
\end{align*}
We sum these inequalities over all $h \ge 0$ (noting that when $h$ is sufficiently large, we have $\log\cN(B_p,8^{h+1} \eta \widehat{B_2})=0$), and get
\begin{align*}
    \log\cN(B_p,\eta \widehat{B_2}) &\le \sum_{h\ge 0} \log\cN(\widehat{B_2},\delta_hB_r).
\end{align*}
This concludes the proof of \Cref{lemma:poor_mans_dualization}.
\end{proof}

\subsubsection{Bounding \texorpdfstring{$\log\cN(\widehat{B_2},\eta B_r)$}{log N(\^{B2}, n Br)}}

As we saw in \Cref{lemma:poor_mans_dualization}, it will be enough to understand the behavior of $\log\cN(\widehat{B_2},\eta B_r)$. Since $\widehat{B_2}$ is a linear transformation of a Euclidean ball, we will be able to apply the dual Sudakov inequality (\Cref{fact:sudakov}, \eqref{eq:sudakov_dual}).

To prepare for an application of the dual Sudakov inequality, we bound the Gaussian width of the ball $\inbraces{\vx \in \R^n \suchthat \gnorml{\mLambda^{-1/p}\mA\mR\vx}{r} \le 1}$. As we will see in a moment, the relevance of this ball arises from the fact that it is the $r$-ball with respect to the $\vlambda$ measure after a suitable linear transformation of the underlying space. In particular, it is under the invertible mapping $\vx \mapsto \mR\vx$ that we get $\mW^{1/2}\mLambda^{1/2-1/p}\mA\vx \mapsto \mW^{1/2}\mLambda^{1/2-1/p}\mA\mR\vx = \mU\vx$.

\begin{lemma}
\label{lemma:one_group_subgaussian}
Let $\vg \sim \cN(0,\mI_n)$. We have
\begin{align*}
    \subgnorm{\inparen{\sum_{j \in S_i} \abs{\ip{\vf_j,\vg}}^{p_i}}^{1/p_i}} \lesssim \inparen{1+\sqrt{p_i}}\inparen{\sum_{j \in S_i} \norm{\vf_j}_{2}^{p_i}}^{1/p_i}.
\end{align*}
\end{lemma}
\begin{proof}[Proof of \Cref{lemma:one_group_subgaussian}]
Observe the following Lipschitzness bound, i.e., for any $\vx$, by Cauchy-Schwarz, we have
\begin{align*}
   \inparen{\sum_{j \in S_i} \abs{\ip{\vf_j,\vx}}^{p_i}}^{1/p_i} \le \inparen{\sum_{j \in S_i} \norm{\vf_j}_{2}^{p_i}}^{1/p_i}\norm{\vx}_2
\end{align*}
which means by \Cref{fact:lip_concentration}, we get
\begin{align*}
    \subgnorm{\inparen{\sum_{j \in S_i} \abs{\ip{\vf_j,\vg}}^{p_i}}^{1/p_i}-\exvv{\vg\sim\cN(0,\mI_n)}{\inparen{\sum_{j \in S_i} \abs{\ip{\vf_j,\vg}}^{p_i}}^{1/p_i}}} \lesssim \inparen{\sum_{j \in S_i} \norm{\vf_j}_{2}^{p_i}}^{1/p_i}.
\end{align*}
Now, observe that
\begin{align*}
    \exvv{\vg\sim\cN(0,\mI_n}{\inparen{\sum_{j \in S_i} \abs{\ip{\vf_j,\vg}}^{p_i}}^{1/p_i}} \le \inparen{\sum_{j \in S_i} \norm{\vf_j}_2^{p_i}\exvv{\vg\sim\cN(0,\mI_n)}{\abs{\ip{\frac{\vf_j}{\norm{\vf_j}_2},\vg}}^{p_i}}}^{1/p_i} \asymp p_i^{1/2}\inparen{\sum_{j \in S_i} \norm{\vf_j}_{2}^{p_i}}^{1/p_i},
\end{align*}
and by \Cref{fact:uncentering},
\begin{align*}
    \subgnorm{\inparen{\sum_{j \in S_i} \abs{\ip{\vf_j,\vg}}^{p_i}}^{1/p_i}} \lesssim \inparen{1+\sqrt{p_i}}\inparen{\sum_{j \in S_i} \norm{\vf_j}_{2}^{p_i}}^{1/p_i},
\end{align*}
completing the proof of \Cref{lemma:one_group_subgaussian}.
\end{proof}

Next, we estimate the Gaussian width of $\inbraces{\vx\in\R^n\suchthat \gnorml{\mLambda^{-1/p}\mA\mR\vx}{r} \le 1}$.

\begin{lemma}
\label{lemma:khintchine}
For $r \ge 2$, we have
\begin{align*}
    \exvv{\vg\sim\cN(0,\mI_n)}{\gnorml{\mLambda^{-1/p}\mA\mR\vg}{r}} \lesssim \begin{cases} r^{1/2}\inparen{1+\sqrt{\max_i p_i}}\inparen{\sum_{i=1}^m \vlambda_i \inparen{\sum_{j \in S_i} \norm{\vf_j}_{2}^{p_i}}^{r/p_i}}^{1/r} & \text{ if } r \le \log m \\ \sqrt{\log m} \cdot \max_i \inparen{1+\sqrt{p_i}}\inparen{\sum_{j \in S_i} \norm{\vf_j}_{2}^{p_i}}^{1/p_i} & \text{ otherwise} \end{cases}
\end{align*}
\end{lemma}
\begin{proof}[Proof of \Cref{lemma:khintchine}]
Let $\vu_j$ denote the rows of $\mU$. Note that by \Cref{fact:whitening}, we have
\begin{align*}
    \norm{\vu_j}_2^2 = \tau_j(\mW^{1/2}\mLambda^{1/2-1/p}\mA).
\end{align*}
Now, observe that $\mLambda^{-1/p}\mA\mR = \mW^{-1/2}\mLambda^{-1/2}\mU$. By \Cref{lemma:one_group_subgaussian}, we know that
\begin{align*}
    \exvv{\vg\sim\cN(0,\mI_n)}{\inparen{\sum_{j \in S_i} \abs{\ip{\vf_j,\vg}}^{p_i}}^{r/p_i}} \lesssim r^{r/2}\inparen{1+\sqrt{p_i}}^r\inparen{\sum_{j \in S_i} \norm{\vf_j}_{2}^{p_i}}^{r/p_i}
\end{align*}
We first handle the case where $r \lesssim \log m$. Notice that
\begin{align*}
    \exvv{\vg\sim\cN(0,\mI_n)}{\gnorml{\mLambda^{-1/p}\mA\mR\vg}{r}} &= \exvv{\vg\sim\cN(0,\mI_d))}{\inparen{\sum_{i=1}^m \vlambda_i\inparen{\sum_{j \in S_i}\abs{\ip{\vw_j^{-1/2}\vlambda_i^{-1/2}\vu_j,\vg}}^{p_i}}^{r/p_i}}^{1/r}} \\
    &\le \inparen{\sum_{i=1}^m \vlambda_i\exvv{\vg\sim\cN(0,\mI_d))}{\inparen{\sum_{j \in S_i}\abs{\ip{\vw_j^{-1/2}\vlambda_i^{-1/2}\vu_j,\vg}}^{p_i}}^{r/p_i}}}^{1/r} \\
    &\lesssim \inparen{\sum_{i=1}^m \vlambda_i\inparen{r^{r/2}\inparen{1+\sqrt{p_i}}^r\inparen{\sum_{j \in S_i} \norm{\vf_j}_{2}^{p_i}}^{r/p_i}}}^{1/r} \\
    &\le r^{1/2}\inparen{1+\sqrt{\max_i p_i}}\inparen{\sum_{i=1}^m \vlambda_i \inparen{\sum_{j \in S_i} \norm{\vf_j}_{2}^{p_i}}^{r/p_i}}^{1/r}.
\end{align*}
We now handle the case where $r \gtrsim \log m$. We have
\begin{align*}
    \exvv{\vg\sim\cN(0,\mI_n)}{\gnorml{\mLambda^{-1/p}\mA\mR\vg}{r}} &= \exvv{\vg\sim\cN(0,\mI_d))}{\inparen{\sum_{i=1}^m \vlambda_i\inparen{\sum_{j \in S_i}\abs{\ip{\vw_j^{-1/2}\vlambda_i^{-1/2}\vu_j,\vg}}^{p_i}}^{r/p_i}}^{1/r}} \\
    &\lesssim \exvv{\vg\sim\cN(0,\mI_d))}{\max_i \inparen{\sum_{j \in S_i}\abs{\ip{\vw_j^{-1/2}\vlambda_i^{-1/2}\vu_j,\vg}}^{p_i}}^{1/p_i}} \\
    &\lesssim \sqrt{\log m} \cdot \max_i \inparen{1+\sqrt{p_i}}\inparen{\sum_{j \in S_i} \norm{\vf_j}_{2}^{p_i}}^{1/p_i}
\end{align*}
and conclude the proof of \Cref{lemma:khintchine} (the last line follows from \Cref{fact:max_subgaussians}).
\end{proof}

Now, we show how to relate $(\sum_{j \in S_i} \norm{\vf_j}_2^{p_i})^{1/p_i}$ to $\Fstar$.

\begin{lemma}
\label{lemma:convert_f_to_fstar}
For all $i \in [m]$, we have
\begin{align*}
    \inparen{\sum_{j \in S_i} \norm{\vf_j}_2^{p_i}}^{2/p_i} \le \Fstar.
\end{align*}
\end{lemma}
\begin{proof}[Proof of \Cref{lemma:convert_f_to_fstar}]
Recall \Cref{fact:whitening}; this gives us
\begin{align*}
    \inparen{\sum_{j \in S_i} \norm{\vf_j}_{2}^{p_i}}^{2/p_i} = \inparen{\sum_{j \in S_i} \inparen{\frac{\tau_j\inparen{\mW^{1/2}\mLambda^{1/2-1/p}\mA}}{\vw_j\vlambda_i}}^{p_i/2}}^{2/p_i} = \frac{1}{\vlambda_i}\inparen{\sum_{j \in S_i} \inparen{\frac{\tau_j\inparen{\mW^{1/2}\mLambda^{1/2-1/p}\mA}}{\vw_j}}^{p_i/2}}^{2/p_i}.
\end{align*}
We recall that $\Fstar$ satisfies \Cref{defn:block_lewis_overestimate} and conclude the proof of \Cref{lemma:convert_f_to_fstar}.
\end{proof}

We now have enough tools to build a na\"ive estimate of $\log\cN\inparen{\widehat{B_2},\eta B_r}$ via directly applying the dual Sudakov inequality.

\begin{lemma}
\label{lemma:r_to_two}
We have
\begin{align*}
    \log\cN\inparen{\widehat{B_2},\eta B_r} \lesssim \eta^{-2} \cdot  \begin{cases} r\inparen{1+\sqrt{\max_i p_i}}^2\inparen{\sum_{i=1}^m \vlambda_i \inparen{\sum_{j \in S_i} \norm{\vf_j}_{2}^{p_i}}^{r/p_i}}^{2/r} & \text{ if } r \le \log m \\ \log m \cdot \max_i \inparen{1+\sqrt{p_i}}^2\inparen{\sum_{j \in S_i} \norm{\vf_j}_{2}^{p_i}}^{2/p_i} & \text{ otherwise} \end{cases}
\end{align*}
Simply put, we may also write
\begin{align*}
    \log\cN\inparen{\widehat{B_2},\eta B_r} \lesssim \eta^{-2} \cdot r\max_{i}\min\inparen{p_i, \log\abs{S_i}}\Fstar.
\end{align*}
\end{lemma}
\begin{proof}[Proof of \Cref{lemma:r_to_two}]
Since $\mR$ is invertible, it will be enough to bound the covering number
\begin{align*}
    \cN \coloneqq \cN\inparen{\inbraces{\vx\in\R^n \suchthat \norm{\mU\vx}_2 \le 1}, \eta\inbraces{\vx\in\R^n \suchthat \gnorml{\mLambda^{-1/p}\mA\mR\vx}{r} \le 1}}.
\end{align*}
Because $\norm{\mU\vx}_2 = \norm{\vx}_2$, we can apply the dual Sudakov Inequality (\Cref{fact:sudakov}, \eqref{eq:sudakov_dual}). This means we get
\begin{align*}
    \log\cN \lesssim \eta^{-2} \inparen{\exvv{\vg\sim\cN(0,\mI_n)}{\gnorml{\mLambda^{-1/p}\mA\mR\vg}{r}}}^2.
\end{align*}
We plug in the result from \Cref{lemma:khintchine} and conclude the proof of \Cref{lemma:r_to_two}. The statement after the ``simply put'' follows from \Cref{lemma:convert_f_to_fstar}.
\end{proof}

Although the calculation in \Cref{lemma:r_to_two} works pretty well for small $r$, this degrades quite rapidly once $r$ is large (say, larger than $\log n$).

To resolve this, we build another estimate for $\log\cN\inparen{\widehat{B_2},\eta B_r}$ that performs better when $r$ is larger than $\log n$ or so. We will be able to do this after an interpolation step and a simple geometric observation relating $\widehat{B_2}$ and $B_r$.

\begin{lemma}
\label{lemma:r_ellipsoid}
Let $\Delta_i$ be defined such that 
\begin{align*}
    \Delta_i^{1/2} \coloneqq \max_{\vx\in\R^n\setminus\inbraces{0}} \frac{\vlambda_i^{-1/p}\norm{\mA_{S_i}\vx}_{p_i}}{\norm{\mW^{1/2}\mLambda^{1/2-1/p}\mA\vx}_2},
\end{align*}
and let $\Delta \coloneqq \max_{i \in [m]} \Delta_i$.

For all $\vx\in\R^n$ and $r > 2$, if $\abs{S_i} = 1$ for all $i$, then we have
\begin{align*}
    \norm{\mLambda^{-1/2}\mU\vx}_{r(\vlambda)} \le \Delta^{1/2-1/r} \cdot \norm{\mU\vx}_2 \le \Delta^{1/2-1/r} \cdot \norm{\mLambda^{-1/2}\mU\vx}_{r(\vlambda)}.
\end{align*}
Moreover, if there exists at least one $S_i$ for which $\abs{S_i} > 1$, then we have
\begin{align*}
    \gnorml{\mLambda^{-1/p}\mA\vx}{r} \le \Delta^{1/2}\norm{\mW^{1/2}\mLambda^{1/2-1/p}\mA\vx}_2.
\end{align*}
\end{lemma}
\begin{proof}[Proof of \Cref{lemma:r_ellipsoid}]
For the sake of intuition and an interpretation, the reader may think $\Delta \approx n$.

Note that for the case where all the $S_i$ are singletons, we may assume $\mW = \mI_m$.

Since $\vlambda$ is a probability measure, we have for any $r \ge 2$ and for all $\vx\in\R^n$ that
\begin{align*}
    \norm{\mU\vx}_2 = \norm{\mLambda^{-1/2}\mU\vx}_{2(\vlambda)} \le \norm{\mLambda^{-1/2}\mU\vx}_{r(\vlambda)}.
\end{align*}
We now prove the lower bound. We have
\begin{align*}
    \norm{\mLambda^{-1/2}\mU\vx}_{r(\vlambda)} &= \inparen{\sum_{i=1}^m \vlambda_i\abs{\ip{\vlambda_i^{-1/2}\vu_i,\vx}}^{r}}^{1/r} =  \inparen{\sum_{i=1}^m \vlambda_i\abs{\ip{\vlambda_i^{-1/2}\vu_i,\vx}}^{2}\cdot\abs{\ip{\vlambda_i^{-1/2}\vu_i,\vx}}^{r-2}}^{1/r} \\
    &\le \inparen{\norm{\mU\vx}_2^2 \cdot \max_i \abs{\ip{\vlambda_i^{-1/2}\vu_i,\vx}}^{r-2}}^{1/r} \le \inparen{\norm{\mU\vx}_2^2 \cdot \inparen{\Delta^{1/2}\norm{\vx}_2}^{r-2}}^{1/r} \\
    &= \Delta^{1/2-1/r} \norm{\vx}_2.
\end{align*}
We now move onto the more general case where the $S_i$ are allowed to have multiple elements. We write
\begin{align*}
    \gnorml{\mLambda^{-1/p}\mA\vx}{r} &= \inparen{\sum_{i=1}^m \vlambda_i\norm{\mLambda_{S_i}^{-1/p}\mA_{S_i}\vx}_{p_i}^r}^{1/r} \\
    &\le \inparen{\sum_{i=1}^m \vlambda_i\Delta_i^{r/2}\norm{\mW^{1/2}\mLambda^{1/2-1/p}\mA\vx}_{2}^{r}}^{1/r} \le \Delta^{1/2}\norm{\mW^{1/2}\mLambda^{1/2-1/p}\mA\vx}_2,
\end{align*}
which concludes the proof of \Cref{lemma:r_ellipsoid}.
\end{proof}

At last, we have the tools we need to give a characterization of $\log\cN\inparen{\widehat{B_2},\eta B_r}$ when $r \gtrsim \log n$.

\begin{lemma}
\label{lemma:r_to_two_complicated}
If all the $S_i$ have size $1$, then
\begin{align*}
    \log\cN\inparen{\widehat{B_2}, \eta B_r} \lesssim \inparen{\frac{\eta}{2}}^{-\frac{2r}{r-2}} \cdot r\Fstar\log\Fstar,
\end{align*}
and if there is at least one $S_i$ larger than $1$, then
\begin{align*}
    \log\cN\inparen{\widehat{B_2}, \eta B_r} \lesssim \inparen{\frac{\eta}{2}}^{-\frac{2r}{r-2}} \cdot \frac{2r^2}{r-2}\max_i \min(p_i+1, \log\abs{S_i})\Fstar\log\Fstar.
\end{align*}
\end{lemma}
\begin{proof}[Proof of \Cref{lemma:r_to_two_complicated}]
The reader familiar with the work of \citet{blm89} can think of the present Lemma as a generalization of (7.13) of Proposition 7.2 of that work.

Define
\begin{align*}
    M(\mF,r) &\coloneqq \begin{cases} r\inparen{1+\sqrt{\max_i p_i}}^2\inparen{\sum_{i=1}^m \vlambda_i \inparen{\sum_{j \in S_i} \norm{\vf_j}_{2}^{p_i}}^{r/p_i}}^{2/r} & \text{ if } r \le \log m \\ \log m \cdot \max_i \inparen{1+\sqrt{p_i}}^2\inparen{\sum_{j \in S_i} \norm{\vf_j}_{2}^{p_i}}^{2/p_i} & \text{ otherwise} \end{cases}.
\end{align*}

Let $q > r$ and $0 < \theta < 1$ be such that
\begin{align*}
    \frac{1}{r} = \frac{1-\theta}{2} + \frac{\theta}{q}.
\end{align*}
By interpolation, observe that we have
\begin{align*}
    \gnorml{\mLambda^{-1/p}\mA(\vx_1-\vx_2)}{r} &\le \gnorml{\mLambda^{-1/p}\mA(\vx_1-\vx_2)}{2}^{1-\theta} \cdot \gnorml{\mLambda^{-1/p}\mA(\vx_1-\vx_2)}{q}^{\theta} \\
    &\le 2\gnorml{\mLambda^{-1/p}\mA(\vx_1-\vx_2)}{q}^{\theta}
\end{align*}
which means that
\begin{align*}
    \log\cN\inparen{B_2,\eta B_r} \le \log\cN\inparen{B_2,(\eta/2)^{1/\theta}B_q} \le \inparen{\frac{\eta}{2}}^{-2/\theta}M(\mF,q).
\end{align*}
Let us set $q = r\log D$, where we will choose $D$ in a moment. Then, notice that
\begin{align*}
    \inparen{\frac{\eta}{2}}^{-2/\theta} = \inparen{\frac{\eta}{2}}^{-2r(q-2)/(q(r-2))} = \inparen{\frac{\eta}{2}}^{-\frac{2r}{r-2}\inparen{1-\frac{2}{r\log D}}} = \inparen{\frac{\eta}{2}}^{-\frac{2r}{r-2} + \frac{4}{(r-2)\log D}}.
\end{align*}
It is now sufficient to identify $D$ such that whenever $\eta$ is small enough to have $\log\cN > 0$, we have
\begin{align*}
    \inparen{\frac{\eta}{2}}^{\frac{4}{(r-2)\log D}} \lesssim 1.
\end{align*}
To identify this $D$, notice that \Cref{lemma:r_ellipsoid} implies that if all the $S_i$ have size $1$, then only values of $\eta$ such that $\eta \le \Delta^{1/2-1/r}$ contribute to $\log\cN$. In the more general setting, observe only $\eta \le \Delta^{1/2}$ counts.

Hence, if all the $S_i$s are singletons, we choose $D = \Delta$. For any $\eta \le 2\Delta^{1/2-1/r} = 2D^{1/2-1/r}$, we see that
\begin{align*}
    \inparen{\frac{\eta}{2}}^{\frac{4}{(r-2)\log D}} \le \Delta^{\frac{r-2}{2r} \cdot \frac{4}{(r-2)\logv{\Delta}}} = \Delta^{\frac{2}{r\log\Delta}} = 2^{2/r} \le 2.
\end{align*}
Similarly, for the case where the $S_i$ are more generally sized, we choose $D = \Delta^{2r/(r-2)}$. Now, for any $\eta \le 2\Delta^{1/2}$, we get
\begin{align*}
    \inparen{\frac{\eta}{2}}^{\frac{4}{(r-2)\log D}} \le \Delta^{\frac{4}{(r-2)\logv{\Delta^{(2r/(r-2))}}}} = \Delta^{\frac{2}{r\log\Delta}} = 2^{2/r} \le 2.
\end{align*}
Putting everything together, if all the $S_i$s are singletons, we get
\begin{align*}
    \log\cN\inparen{\widehat{B_2}, \eta B_r} \lesssim \inparen{\frac{\eta}{2}}^{-\frac{2r}{r-2}} \cdot M(\mF,q) \lesssim \inparen{\frac{\eta}{2}}^{-\frac{2r}{r-2}} \cdot r\Fstar\log\Fstar,
\end{align*}
and in the more general case,
\begin{align*}
    \log\cN\inparen{\widehat{B_2}, \eta B_r} \lesssim \inparen{\frac{\eta}{2}}^{-\frac{2r}{r-2}} \cdot M(\mF,q) \lesssim \inparen{\frac{\eta}{2}}^{-\frac{2r}{r-2}} \cdot \frac{2r^2}{r-2}\max_i \min(p_i+1, \log\abs{S_i})\Fstar\log\Fstar.
\end{align*}
This concludes the proof of \Cref{lemma:r_to_two_complicated}.
\end{proof}

\subsubsection{Putting everything together}

We are finally ready to combine all the tools we have built in the last few subsections to prove our entropy estimate when $0 < p < 2$.

Below, we state \Cref{lemma:covering_two_to_one}, which more precisely characterizes the behavior of the dependence on $p$ referred to by \Cref{lemma:covering_two_to_one_clean}. 

\begin{lemma}
\label{lemma:covering_two_to_one}
We have
\begin{align*}
    \log\cN(B_p,\eta \widehat{B_2}) \lesssim \eta^{-\frac{2p}{2-p}} \cdot C(p)\max_{i}\min\inparen{p_i+1,\log\abs{S_i}}\Fstar\log\Fstar,
\end{align*}
where $C(p)$ is a constant that only depends on $p$. The constant $C(p)$ is defined as follows. For $0 < \theta < p$, let $r=(2-\theta)p/(p-\theta)$. Then, we define
\begin{align}
    \widehat{C}(p,\theta) &\coloneqq \inparen{\frac{\inparen{2\cdot 8^{2/\theta}}^{-\frac{\theta}{2-\theta}}}{2}}^{-\frac{2r}{r-2}}\nonumber \\
    \widehat{\widehat{C}}(p) &\coloneqq \begin{cases} \min\inbraces{\widehat{C}(p,p/2), \widehat{C}(p,1)} & \text{ if } 1 \le p < 2 \\ \widehat{C}(p,p/2) & \text{ if } 0 < p < 1 \end{cases}\label{eq:hathat} \\
    C(p) &\coloneqq \widehat{\widehat{C}}(p)\begin{cases} r & \text{ if } \abs{S_i} = 1 \text{ for all } i \\ 2r + \frac{4r}{r-2} & \text{ otherwise} \end{cases}\label{eq:choosing_c_from_hathat},
\end{align}
where, in an abuse of notation, $r$ in \eqref{eq:choosing_c_from_hathat} is chosen according to the value of $\theta$ that is selected by $\widehat{\widehat{C}}(p)$ in \eqref{eq:hathat}.
\end{lemma}
\begin{proof}[Proof of \Cref{lemma:covering_two_to_one} and \Cref{lemma:covering_two_to_one_clean}]
This time, following \Cref{lemma:r_to_two_complicated}, we define
\begin{align*}
    M(\mF,r) &\coloneqq \begin{cases} r\Fstar\log\Fstar & \text{ if } \abs{S_i}=1 \\ (2r + \frac{4r}{r-2})\inparen{\max_i \min(p_i+1, \log\abs{S_i})}\Fstar\log\Fstar & \text{ otherwise} \end{cases}.
\end{align*}
Use \Cref{lemma:poor_mans_dualization} and \Cref{lemma:r_to_two} to write
\begin{align}
    \log\cN(B_p,\eta \widehat{B_2}) &\le \sum_{h\ge 0} \log\cN(\widehat{B_2},\delta_hB_r) \le M(\mF,r)\sum_{h \ge 0} \inparen{\frac{\delta_h}{2}}^{-\frac{2r}{r-2}}\nonumber \\
    &= M(\mF,r) \sum_{h \ge 0} \inparen{\eta^{\frac{\theta}{2-\theta}}\cdot 8^{(h+1)\cdot\inparen{\frac{\theta}{2-\theta}}} \cdot \frac{\inparen{2\cdot 8^{2/\theta}}^{-\frac{\theta}{2-\theta}}}{2}}^{-\frac{2r}{r-2}} \nonumber \\
    &= \eta^{-\frac{\theta}{2-\theta}\cdot\frac{2r}{r-2}} \cdot \underbrace{\inparen{\frac{\inparen{2\cdot 8^{2/\theta}}^{-\frac{\theta}{2-\theta}}}{2}}^{-\frac{2r}{r-2}}}_{= \widehat{C}(p,\theta)} M(\mF,r)\sum_{h \ge 0} 8^{(h+1) \cdot \inparen{-\frac{\theta}{2-\theta}}\cdot\frac{2r}{r-2}} \label{eq:substitute_theta}.
\end{align}
We now make the substitution $\theta=p/2$. By the formula in \Cref{lemma:poor_mans_dualization}, this means that $r = 4-p$ and $-\theta/(2-\theta) \cdot 2r/(r-2) = -2p/(2-p)$. We continue.
\begin{align*}
    \log\cN(B_p,\eta B_2) &\le \eta^{-\frac{\theta}{2-\theta}\cdot\frac{2r}{r-2}} \cdot \widehat{C}(p,\theta) M(\mF,r)\sum_{h \ge 0} 8^{(h+1) \cdot -\frac{\theta}{2-\theta}\cdot\frac{2r}{r-2}} \\
    &= \eta^{-\frac{2p}{2-p}} \cdot \widehat{C}(p,p/2) M(\mF,4-p)\sum_{h \ge 0} 8^{(h+1) \cdot -\frac{2p}{2-p}} \\
    &\lesssim \eta^{-\frac{2p}{2-p}} \cdot \widehat{C}(p,p/2) M(\mF,4-p).
\end{align*}
A regrettable consequence of the above calculation is that the ``constant'' $C(p,\theta)=C(p,p/2)$ explodes as $p\rightarrow 2$, as observed by \citet{sz01}. To fix this, we perform a slightly different variant of this calculation in the regime where $1 < p < 2$. We resume from \eqref{eq:substitute_theta} except we use $\theta = 1$. Here, again using \Cref{lemma:poor_mans_dualization}, we check that $r=p/(p-1)$ and $-\theta/(2-\theta) \cdot 2r/(r-2) = -2p/(2-p)$ (note that now $r$ is the conjugate exponent of $p$). This means that
\begin{align*}
    \log\cN(B_p,\eta \widehat{B_2}) &\le \eta^{-\frac{\theta}{2-\theta}\cdot\frac{2r}{r-2}} \cdot \widehat{C}(p,\theta) M(\mF,r)\sum_{h \ge 0} 8^{(h+1) \cdot -\frac{\theta}{2-\theta}\cdot\frac{2r}{r-2}} \\
    &\lesssim \eta^{-\frac{2p}{2-p}} \cdot \widehat{C}(p,1)\sum_{h \ge 0} 8^{(h+1) \cdot -\frac{2p}{2-p}} \lesssim \eta^{-\frac{2p}{2-p}}M\inparen{\mF,\frac{p}{p-1}}.
\end{align*}
Taking the minimum over all the cases (of course where applicable) and expanding out the definition of $M(\mF,r)$ concludes the proof of \Cref{lemma:covering_two_to_one} and \Cref{lemma:covering_two_to_one_clean}.
\end{proof}

\subsection{Covering numbers for \texorpdfstring{$p \ge 2$}{p > 2}}

We will see that compared to the previous section, our task when $p\ge 2$ is far easier. The main technical lemma we need is \Cref{lemma:covering_poly_to_two}, which we need for both regimes of $p$.

We first state and prove \Cref{lemma:covering_poly_to_two}.

\begin{lemma}
\label{lemma:covering_poly_to_two}
Let $\mW$ and $\mLambda$ be chosen according to \Cref{thm:general_covering}. Suppose that $H \ge 1$ is such that $H\vrho_i \ge \nfrac{\valpha_i^p}{\norm{\valpha}_p^p}$ for all $i \in [m]$. Then,
\begin{align*}
    \log\cN\inparen{\widehat{B_2}, \eta\inbraces{\vx\in\R^n \suchthat \polynorm{\vx} \le 1}} \lesssim \eta^{-2}\cdot{\underset{i \in S}{\max}\ \min\inbraces{p_i,\log\abs{S_i}}H^{2/p}\norm{\valpha}_p^2\log\mtilde}.
\end{align*}
\end{lemma}
\begin{proof}[Proof of \Cref{lemma:covering_poly_to_two}]
Following the proof of \Cref{lemma:r_to_two} and the references therein, our goal here is to analyze the quantity
\begin{align*}
    \log\cN\inparen{\inbraces{\vx\in\R^n \suchthat \norm{\mW^{1/2}\mLambda^{1/2-1/p}\mA\vx}_2 \le 1}, \eta\inbraces{\vx\in\R^n \suchthat \polynorm{\vx} \le 1}}.
\end{align*}
Using the same type of linear transformation argument as in \Cref{lemma:r_to_two} (so, replacing every $\vx$ above with $\mR\vx$), we find that it is in fact sufficient to analyze
\begin{align*}
    \log\cN \coloneqq \log\cN\inparen{\inbraces{\vx \suchthat \norm{\mU\vx}_2 \le 1}, \eta\inbraces{\vx \suchthat \max_{i \in S}\ \vrho_i^{-1/p}\norm{\mW_{S_i}^{-1/2}\mLambda^{1/p-1/2}_{S_i}\mU\vx}_{p_i} \le 1}}.
\end{align*}
Recall that $\norm{\mU\vx}_2 = \norm{\vx}_2$, so a natural plan is to apply the dual Sudakov inequality (\Cref{fact:sudakov}, \eqref{eq:sudakov_dual}). We first consider the quantity (when $2 \le p_i \le \log\abs{S_i}$)
\begin{align*}
    \exvv{\vg\sim\cN(0,\mI_n)}{\norm{\mW_{S_i}^{-1/2}\mLambda^{1/p-1/2}_{S_i}\mU\vg}_{p_i}} &= \exvv{\vg\sim\cN(0,\mI_n)}{\inparen{\sum_{j \in S_i} \abs{\vlambda_i^{1/p}\ip{\vw_{j}^{-1/2}\vlambda_i^{-1/2}\vu_j,\vg}}^{p_i}}^{1/p_i}} \\
    &\le \inparen{\sum_{j\in S_i} \inparen{\vlambda_i^{1/p}\norm{\vw_j^{-1/2}\vlambda_i^{-1/2}\vu_j}_2}^{p_i}\exvv{g\sim\cN(0,1)}{\abs{g}^{p_i}}}^{1/p_i} \\
    &\lesssim p_i^{1/2} \inparen{\sum_{j\in S_i} \inparen{\vlambda_i^{1/p}\norm{\vw_j^{-1/2}\vlambda_i^{-1/2}\vu_j}_2}^{p_i}}^{1/p_i} \\
    &= p_i^{1/2}\inparen{\sum_{j \in S_i} \inparen{\vlambda_i^{1/p}\norm{\vf_j}_2}^{p_i}}^{1/p_i} = p_i^{1/2}\vlambda_i^{1/p}\inparen{\sum_{j \in S_i} \norm{\vf_j}_2^{p_i}}^{1/p_i} = p_i^{1/2}\valpha_i.
\end{align*}
On the other hand, if $p_i \ge \log\abs{S_i}$, then we get
\begin{align*}
    \exvv{\vg\sim\cN(0,\mI_n)}{\norm{\mW^{-1/2}_{S_i}\mLambda^{1/p-1/2}_{S_i}\mU\vg}_{p_i}} &\lesssim \exvv{\vg\sim\cN(0,\mI_n)}{\norm{\mW^{-1/2}_{S_i}\mLambda^{1/p-1/2}_{S_i}\mU\vg}_{\infty}} \\
    &\lesssim  \vlambda_i^{1/p}\max_{j\in S_i}\norm{\vf_j}_2\sqrt{\log\abs{S_i}} \asymp\valpha_i\sqrt{\log\abs{S_i}}.
\end{align*}
Finally, when $p_1=\dots=p_m=p < 2$, we have
\begin{align*}
    \exvv{\vg\sim\cN(0,\mI_n)}{\norm{\widehat{\mLambda}_{S_i}^{1/p-1/2}\mU\vg}_p} &= \exvv{\vg\sim\cN(0,\mI_n)}{\inparen{\sum_{j \in S_i}\widehat{\vlambda}_j\abs{\ip{\widehat{\vlambda}_j^{-1/2}\vu_j,\vg}}^{p}}^{1/p}} \\
    &\le \inparen{\sum_{j\in S_i} \widehat{\vlambda}_j\norm{\widehat{\vlambda}_j^{-1/2}\vu_j}_2^p\exvv{g\sim\cN(0,1)}{\abs{g}^{p}}}^{1/p} \lesssim p^{1/2}\inparen{\sum_{j \in S_i} \widehat{\valpha}^p}^{1/p} = p^{1/2}\valpha_i.
\end{align*}
This means that throughout the rest of the proof, we assume without loss of generality that $p_i \le \log\abs{S_i}$. Next, observe that for any $\vx\in\R^n$ and when $p_1,\dots,p_m \ge 2$,
\begin{align*}
    \norm{\mW^{-1/2}_{S_i}\mLambda^{1/p-1/2}_{S_i}\mU\vx}_{p_i} = \inparen{\sum_{j \in S_i} \abs{\vlambda_i^{1/p}\ip{\vf_j,\vx}}^{p_i}}^{1/p_i} \le \inparen{\sum_{j \in S_i} \abs{\vlambda_i^{1/p}\norm{\vf_j}_2\norm{\vx}_2}^{p_i}}^{1/p_i} = \valpha_i\norm{\vx}_2.
\end{align*}
Similarly, when $p_1=\dots=p_m=p$, we first observe
\begin{align*}
    \abs{\ip{\widehat{\vlambda}_j^{1/p-1/2}\vu_j,\vx}} \le \widehat{\valpha}_j\norm{\vx}_2.
\end{align*}
We therefore get
\begin{align}
    \norm{\widehat{\mLambda}_{S_i}^{1/p-1/2}\mU\vx}_p \le \inparen{\sum_{j \in S_i} \widehat{\valpha}_j^p}^{1/p}\norm{\vx}_2 = \valpha_i\norm{\vx}_2.\label{eq:alpha_alleq_lipschitz}
\end{align}

Hence, after applying \Cref{fact:lip_concentration} and \Cref{fact:uncentering}, we notice that $\valpha_i(H\vrho_i)^{-1/p}\le\norm{\valpha}_p$ and get
\begin{align*}
    \subgnorm{\vrho_i^{-1/p}\norm{\mLambda^{1/p-1/2}_{S_i}\mU\vg}_{p_i}} \lesssim \inparen{1 + \sqrt{p_i}}H^{1/p}\norm{\valpha}_p.
\end{align*}
All of this implies that for any subset $S$ of size $\mtilde$ (see Exercise 2.5.10 of \cite{vershynin_2018}),
\begin{align*}
    \exvv{\vg\sim\cN(0,\mI_n)}{\max_{i \in S} \vrho_i^{-1/p}\norm{\mLambda^{1/p-1/2}_{S_i}\mU\vg}_{p_i}} \lesssim \max_{i \in S} p_i^{1/2}H^{1/p}\norm{\valpha}_p\sqrt{\log\mtilde}.
\end{align*}
Thus,
\begin{align*}
    \log\cN \lesssim \eta^{-2}\inparen{\inparen{\max_{i \in S} p_i+1}H^{2/p}\norm{\valpha}_p^2\log\mtilde}.
\end{align*}
This concludes the proof of \Cref{lemma:covering_poly_to_two}.
\end{proof}

We are finally ready to prove \Cref{thm:general_covering}.

\begin{proof}[Proof of \Cref{thm:general_covering}]
For notational simplicity in this proof, write
\begin{align*}
    K \coloneqq \inbraces{\vx\in\R^n\suchthat \polynorm{\vx} \le 1}.
\end{align*}
Let us first handle the case where $p \ge 2$. By our choice of $\mW$, we have for all $\vx\in\R^n$ that
\begin{align*}
     \norm{\mW^{1/2}\mLambda^{1/2-1/p}\mA\vx}_{2} \le \gnorml{\mLambda^{-1/p}\mA\vx}{p} = \gnorm{\mA\vx}{p}.
\end{align*}
This implies the containment $B_p \subseteq \widehat{B_2}$, and thus
\begin{align*}
    \log\cN\inparen{B_p, \eta K} \le \log\cN\inparen{\widehat{B_2}, \eta K} \le \eta^{-2}\max_{i}\min\inbraces{p_i,\log\abs{S_i}}H^{2/p}\Fstar\log\mtilde.
\end{align*}
The desired result now follows immediately from \Cref{lemma:covering_poly_to_two}.

For the case where $p < 2$ and $p_i \ge 2$ for all $i$, we require a bit more work. By our choice of $\mW$, we have
\begin{align*}
    \norm{\mW^{1/2}\mLambda^{1/2-1/p}\mA\vx}_{2} \le \gnorml{\mLambda^{-1/p}\mA\vx}{2} = \norm{\mLambda^{1/2-1/p}\mA\vx}_2,
\end{align*}
and for any $t > 0$ (after remembering \eqref{eq:alpha_to_fstar} which tells us $\norm{\valpha}_p \le \sqrt{\Fstar}$),
\begin{align*}
    \log\cN\inparen{B_p, \eta K} &\le \log\cN\inparen{B_p, t\widehat{B_2}} + \log\cN\inparen{t\widehat{B_2}, \eta K} = \log\cN\inparen{B_p, t\widehat{B_2}} + \log\cN\inparen{\widehat{B_2}, \frac{\eta}{t}\cdot K} \\
    &\le t^{-\frac{2p}{2-p}} \cdot C(p)\max_{i}\min\inbraces{p_i+1,\log\abs{S_i}}\Fstar\log\Fstar \\
    &\quad+\inparen{\frac{\eta}{t}}^{-2}\max_{i}\min\inbraces{p_i+1,\log\abs{S_i}}H^{2/p}\Fstar\log\mtilde,
\end{align*}
where the last line follows from \Cref{lemma:covering_two_to_one_clean} and \Cref{lemma:covering_poly_to_two}.
Choose $t=\eta^{1-p/2} \cdot H^{1/2-1/p}$, and for simplicity let $p^{\star} \coloneqq \max_i\min\inbraces{p_i+1,\log\abs{S_i}}$. We write
\begin{align*}
    \log\cN\inparen{B_p, \eta K} &\le t^{-\frac{2p}{2-p}} \cdot C(p)p^{\star}\Fstar\log\Fstar +\inparen{\frac{\eta}{t}}^{-2}p^{\star}H^{2/p}\Fstar\log\mtilde \\
    &\lesssim \eta^{-p}\cdot H \cdot C(p)\inparen{p^{\star}\Fstar\log\max\inbraces{\mtilde,\Fstar}}.
\end{align*}
Finally, we need to address the case where $p_1=\dots=p_m=p < 2$, as this is not covered by the construction of the block Lewis weights (\Cref{lemma:block_lewis}).

To do so, notice that $\gnorm{\mA\vx}{p} = \norm{\mA\vx}_p$ for all $\vx\in\R^n$. So, we reduce to the case where all the $S_i$ have size $1$. In particular, let $\widehat{\vlambda}$ be a probability measure over $[k]$ such that for all $j\in [k]$,
\begin{align*}
    \Fstar \ge \frac{\tau_j\inparen{\widehat{\mLambda}^{1/2-1/p}\mA}}{\widehat{\vlambda}_j},
\end{align*}
and let
\begin{align*}
    \widehat{B_2} = \inbraces{\vx\in\R^n\suchthat \norm{\widehat{\mLambda}^{1/2-1/p}\mA\vx}_2 \le 1}.
\end{align*}
Notice that this setup is in accordance with \Cref{defn:alpha}. By \Cref{lemma:covering_two_to_one_clean}, we get
\begin{align*}
    \log\cN(B_p,t\widehat{B_2}) \lesssim t^{-\frac{2p}{2-p}} \cdot C(p)\Fstar\log\Fstar.
\end{align*}
It remains to bound $\log\cN(\widehat{B_2},(\eta/t)K)$. Let $\vlambda$ be a probability measure over $[m]$, where $\vlambda_i = \sum_{j \in S_i} \widehat{\vlambda}_j$. Define $\widehat{\valpha} \in \R^k$ analogously to $\widehat{\vlambda}$. Notice that, for these choices of $\widehat{\valpha}$ and $\widehat{\vlambda}$, we see that the conclusion of \Cref{lemma:covering_poly_to_two} still holds, and we have
\begin{align*}
    \log\cN\inparen{\widehat{B_2},\inparen{\frac{\eta}{t}} \cdot K} \lesssim \inparen{\frac{\eta}{t}}^{-2}H^{2/p}\Fstar\log\mtilde.
\end{align*}
Now, the calculation is the same as before, and we conclude the proof of \Cref{thm:general_covering}.
\end{proof}

\subsection{Volume-based metric entropy}

In this subsection, we prove \Cref{lemma:covering_easy}, which is an easy consequence of a volume-based argument to obtain a covering number guarantee.

We start with \Cref{lemma:containments}.

\begin{lemma}
\label{lemma:containments}
Let $S \subseteq [m]$ have size $\mtilde$. Let $H \ge 1$ be such that $H\vrho_i \ge \nfrac{\valpha_i^p}{\norm{\valpha}_p^p}$ for all $i \in [m]$. For all $\vx\in\R^n$, we have
\begin{align*}
    \max_{i \in S} \frac{\norm{\mA_{S_i}\vx}_{p_i}}{\inparen{\Fstar}^{\max(1/2,1/p)}\vrho_i^{1/p}H^{1/p}} &\le \gnorm{\mA\vx}{p}.
\end{align*}
\end{lemma}
\begin{proof}[Proof of \Cref{lemma:containments}]
Consider the invertible mapping $\vx \mapsto \mR\vx$, and write
\begin{align}
    \norm{\mW^{-1/2}_{S_i}\mLambda^{1/p-1/2}_{S_i}\mU\vx}_{p_i} = \inparen{\sum_{j \in S_i} \abs{\vlambda_i^{1/p}\ip{\vf_j,\vx}}^{p_i}}^{1/p_i} \le \inparen{\sum_{j \in S_i} \abs{\vlambda_i^{1/p}\norm{\vf_j}_2\norm{\vx}_2}^{p_i}}^{1/p_i} = \valpha_i\norm{\vx}_2.\label{eq:alpha_importance_proof}
\end{align}
This means that when $p \ge 2$,
\begin{align*}
    \norm{\mA_{S_i}\vx}_{p_i} \le \valpha_i\norm{\mW^{1/2}\mLambda^{1/2-1/p}\mA\vx}_2 \le \valpha_i\gnorm{\mLambda^{1/2-1/p}\mA\vx}{2} \le \valpha_i\gnorm{\mA\vx}{p} \le \norm{\valpha}_p\vrho_i^{1/p}H^{1/p}\gnorm{\mA\vx}{p}.
\end{align*}
Dividing both sides by $\norm{\valpha}_p\vrho_i^{1/p}H^{1/p}$ and then recalling \eqref{eq:alpha_to_fstar} yields the desired conclusion (in particular, we see that $\norm{\valpha}_p \le \inparen{\Fstar}^{1/2}$).

We now analyze what happens when $p \le 2$ and $p_1,\dots,p_m \ge 2$. Let
\begin{align*}
    \delta_i^{1/2} \coloneqq \max_{\vx\in\R^n\setminus\inbraces{0}}\frac{\vlambda_i^{-1/p}\norm{\mA_{S_i}\vx}_{p_i}}{\gnorm{\mLambda^{1/2-1/p}\mA\vx}{2}}.
\end{align*}
Let $\delta \coloneqq \max_{i \in [m]} \delta_i$. We now see that
\begin{align*}
    &\ \gnorm{\mLambda^{1/2-1/p}\mA\vx}{2} = \inparen{\sum_{i=1}^m \vlambda_i\norm{\mLambda_{S_i}^{-1/p}\mA_{S_i}\vx}_{p_i}^2}^{1/2} = \inparen{\sum_{i=1}^m \vlambda_i\norm{\mLambda_{S_i}^{-1/p}\mA_{S_i}\vx}_{p_i}^{p}\cdot\norm{\mLambda_{S_i}^{-1/p}\mA_{S_i}\vx}_{p_i}^{2-p}}^{1/2} \\
    &\le \inparen{\gnorm{\mA\vx}{p}^p\inparen{\delta^{1/2} \cdot \gnorm{\mLambda^{1/2-1/p}\mA\vx}{2}}^{2-p}}^{1/2} = \gnorm{\mA\vx}{p}^{p/2}\inparen{\delta^{1/2} \cdot \gnorm{\mLambda^{1/2-1/p}\mA\vx}{2}}^{1-p/2}.
\end{align*}
Rearranging and taking the $2/p$-power gives
\begin{align*}
    \gnorm{\mLambda^{1/2-1/p}\mA\vx}{2} \le \gnorm{\mA\vx}{p} \delta^{1/p-1/2}.
\end{align*}
Observe that this yields the inequalities
\begin{align*}
     \frac{\norm{\mA_{S_i}\vx}_{p_i}}{\delta^{1/p-1/2}} &\le \frac{\valpha_i}{\delta^{1/p-1/2}}\norm{\mW^{1/2}\mLambda^{1/2-1/p}\mA\vx}_2 \le \frac{\valpha_i}{\delta^{1/p-1/2}}\gnorm{\mLambda^{1/2-1/p}\mA\vx}{2} \\
     &\le \valpha_i\gnorm{\mA\vx}{p} \le \norm{\valpha}_p\vrho_i^{1/p}H^{1/p}\gnorm{\mA\vx}{p}.
\end{align*}
It remains to bound the $\delta_i$. By the same sort of argument from the $p\ge 2$ case (i.e., \eqref{eq:alpha_importance_proof}), we get
\begin{align*}
    \vlambda_i^{1/p}\delta_i^{1/2} \le \valpha_i = \vlambda_i^{1/p}\inparen{\sum_{j \in S_i} \norm{\vf_j}_2^{p_i}}^{1/p_i}
\end{align*}
which means that
\begin{align*}
    \delta \le \max_{i \in [m]} \inparen{\sum_{j \in S_i} \norm{\vf_j}_2^{p_i}}^{2/p_i} \le \Fstar.
\end{align*}
Now, again using the fact that $\norm{\valpha}_p \le \inparen{\Fstar}^{1/2}$, we see that $\delta^{1/p-1/2}\norm{\valpha}_p \le \inparen{\Fstar}^{1/p}$. 

Finally, we analyze the case where $p_1=\dots=p_m=p < 2$. Using \eqref{eq:alpha_alleq_lipschitz}, and the same sort of argument from above, we have
\begin{align*}
    \norm{\mA_{S_i}\vx}_p \le \valpha_i\norm{\widehat{\mLambda}^{1/2-1/p}\mA\vx}_2 \le \valpha_i\inparen{\Fstar}^{1/p-1/2}\gnorm{\mA\vx}{p} \le \norm{\valpha}_pH^{1/p}\vrho_i^{1/p}\inparen{\Fstar}^{1/p-1/2}\gnorm{\mA\vx}{p}.
\end{align*}
Once again, we use $\norm{\valpha}_p \le \inparen{\Fstar}^{1/2}$.

We have covered all our cases and may conclude the proof of \Cref{lemma:containments}.
\end{proof}

\Cref{lemma:containments} also suggests a useful sanity check, as the denominator points to a sparsity of $\sum_{i \le m} \inparen{\inparen{\Fstar}^{\max(1/2,1/p)}\vrho_i^{1/p}H^{1/p}}^p = H\inparen{\Fstar}^{\max(1,p/2)}$. And, recall that we should be able to set $\Fstar \sim n$, which indeed gives us the dependence on $n$ we see in \Cref{thm:one_shot_lewis}. 

\begin{lemma}
\label{lemma:covering_easy}
Let $S \subseteq [m]$ have size $\mtilde$. We have
\begin{align*}
    \log\cN\inparen{B_p,\eta\inbraces{\vx\in\R^n\suchthat\polynorm{\vx}\le 1}} \le n\logv{\frac{4H^{1/p}\inparen{\Fstar}^{\max(1/p,1/2)}}{\eta}}
\end{align*}
\end{lemma}
\begin{proof}[Proof of \Cref{lemma:covering_easy}]
Define
\begin{align*}
    K &\coloneqq \inbraces{\vx\in\R^n\suchthat\polynorm{\vx}\le 1}.
\end{align*}
Let $C$ be a value such that for all $i \in S$ and $\vx\in\R^n$, we have $(C\vrho_i)^{-1/p}\norm{\mA_{S_i}\vx}_{p_i} \le \gnorm{\mA\vx}{p}$. This means that $B_p \subseteq C^{1/p} K$. Then,
\begin{align*}
    \log\cN\inparen{B_p,\eta K} \le \log\cN\inparen{C^{1/p} K, \eta K} = \log\cN\inparen{K, \frac{\eta}{C^{1/p}} \cdot K} \le n\logv{\frac{4C^{1/p}}{\eta}}.
\end{align*}
By \Cref{lemma:containments}, when $p\ge 2$, we can choose $C = H\inparen{\Fstar}^{\max(1/p,1/2)}$.
This concludes the proof of \Cref{lemma:covering_easy}. 
\end{proof}

\section{Concentration analysis}
\label{sec:generic_chaining}

In this section, we prove Theorem \ref{thm:concentration}. Theorem \ref{thm:concentration} states our main result in its fullest generality. \Cref{thm:one_shot_lewis} follows easily from this, as we show in \Cref{sec:applications}.

We first state \Cref{thm:concentration}.

\begin{restatable}[General concentration result]{mainthm}{concentration}
\label{thm:concentration}
Let $\cG = (\mA \in \R^{k \times n}, S_1,\dots,S_m, p_1,\dots,p_m)$ where $S_1,\dots,S_m$ form a partition of $[k]$. Suppose at least one of the following holds:
\begin{itemize}
    \item $1 \le p < \infty$ and $p_1,\dots,p_m \ge 2$;
    \item $1/\log n \le p_1 = \dots = p_m = p < \infty$;
    \item $p_1 = \dots = p_m = 2$ and $1/\log n \le p < \infty$.
\end{itemize}
Let $P \coloneqq \max\inparen{1, \max_{i \in [m]} \min(p_i,\log\abs{S_i})}$. Suppose that $\vlambda \in \R^m$ is a probability measure over $[m]$, let $\mW$ be a rounding matrix (\Cref{defn:rounding_matrix}) such that we get an $\Fstar$-block Lewis overestimate (\Cref{defn:block_lewis_overestimate}), and define $\valpha$ according to \Cref{defn:alpha}. Let $\cD = \inparen{\vrho_1,\dots,\vrho_m}$ be a probability distribution over $[m]$ and $H \ge 1$ be such that $H\vrho_i \ge \valpha_i^p/\norm{\valpha}_p^p$.

If
\begin{align*}
    \mtilde &= \Omega\inparen{\logv{\nfrac{1}{\delta}}\vbrho \cdot H \cdot P \inparen{\Fstar}^{\max(1,p/2)}},
\end{align*}
and if we sample $\cM \sim \cD^{\mtilde}$, then, with probability $\ge 1-\delta$, we have:
\begin{align*}
    \text{for all } \vx \in \R^n,\quad (1-\eps)\gnorm{\mA\vx}{p}^p \le \frac{1}{\mtilde}\sum_{i \in \cM} \frac{1}{\vrho_i} \cdot \norm{\mA_{S_i}\vx}_{p_i}^p \le (1+\eps)\gnorm{\mA\vx}{p}^p
\end{align*}
\end{restatable}
The goal of the rest of this section is to prove \Cref{thm:concentration}. It may be helpful to recall the argument sketch given in \Cref{sec:overview_concentration}. 

To formalize the idea given there, we first introduce the following notation (recall that $\vrho_i$ is the probability that we choose group $i$ in a round of sampling and that \Cref{defn:entropy_numbers} defines the $e_N$).
\begin{align}
    g_i(\vx) &\coloneqq \norm{\mA_{S_i}\vx}_{p_i}\label{eq:chaining_g_defn} \\
    \dtwo(\vx, \xhat) &\coloneqq \inparen{\sum_{h=1}^{\mtilde} \inparen{\frac{g_{i_h}(\vx)^p}{\vrho_{i_h}} - \frac{g_{i_h}(\xhat)^p}{\vrho_{i_h}}}^2}^{1/2}\label{eq:chaining_dtwo_defn} \\
    \gamma_2(B_{p},\dtwo) &\coloneqq \inf_{\substack{\abs{T_N} \le 2^{2^N};\\ T_N  \subset B_{p}}} \sup_{\vx \in B_{p}} \sum_{N \ge 0} 2^{N/2} \cdot d_2(\vx,T_N)\label{eq:chaining_gamma2_defn}
\end{align}
The goal is to control $\gamma_2(B_p, \dtwo)$. This quantity represents the worst-case approximation error that one incurs by using the discretization scheme given by the $T_N$, where the discretization is taken with respect to the $\dtwo$ metric.

Towards this goal, we first apply a standard symmetrization reduction. Informally, this reduction (\Cref{lemma:symmetrization}) states that it is enough to analyze the average fluctuations of a Rademacher average of any set of $\mtilde$ (not necessarily distinct) reweighted groups.
\begin{lemma}[Symmetrization reduction]
\label{lemma:symmetrization} 
Let $R_1,\dots,R_{\mtilde}$ be independent Rademacher random variables (i.e., $\mathsf{Unif}\inparen{\pm 1}$). We have
\begin{align*}
    \exvv{\cG'}{\abs{\gnorm{\mA\vx}{p}^p - \norm{\vx}_{\cG'}^p}} \le 2 \underset{\cG'}{\E} \exvv{R_1, \ldots, R_{\mtilde}}{\frac{1}{\mtilde}\abs{\sum_{h=1}^{\mtilde} R_h \frac{g_{i'_h}(\vx)^p}{\vrho_{i'_h}}}}.
\end{align*}
\end{lemma}
\begin{proof}[Proof of \Cref{lemma:symmetrization}]
We follow the proof of Lemma 3.1 due to \citet{lee22}. Let \(\tilde{\cG}\) be an independent copy of \(\cG'\). Fixing \(\cG'\), we have by Jensen's inequality that
\begin{align*}
    \abs{\gnorm{\mA\vx}{p}^p - \norm{\vx}_{\cG'}^p} &= \abs{\exvv{\tilde{\cG}}{\norm{\vx}_{\tilde{\cG}}^p} - \norm{\vx}_{\cG'}^p}
    \leq \exvv{\tilde{\cG}}{\abs{\norm{\vx}_{\tilde{\cG}}^p - \norm{\vx}_{\cG'}^p}}.
\end{align*}
Thus, taking expectation over \(\cG'\),
\begin{align*}
    \exvv{\cG'}{\abs{\gnorm{\mA\vx}{p}^p - \norm{\vx}_{\cG'}^p}} &\leq \exvv{\cG', \tilde{\cG}}{\abs{\norm{\vx}_{\tilde{\cG}}^p - \norm{\vx}_{\cG'}^p}}
    = \exvv{\cG', \tilde{\cG}}{\frac{1}{\mtilde}\abs{\sum_{h=1}^{\mtilde} \frac{g_{\tilde{i}_h}(\vx)^p}{\vrho_{\widetilde{i}_h}} -\frac{g_{i'_h}(\vx)^p}{\vrho_{i'_h}}}}.
\end{align*}
Observe that \(\frac{g_{\tilde{i}_h}(\vx)^p}{\vrho_{\widetilde{i}_h}} -\frac{g_{i'_h}(\vx)^p}{\vrho_{i'_h}}\) is symmetric,
and so is distributed the same as \(R_h \inparen{\frac{g_{\tilde{i}_h}(\vx)^p}{\vrho_{\widetilde{i}_h}} -\frac{g_{i'_h}(\vx)^p}{\vrho_{i'_h}}}\)
where \(R_h\) is an independent Rademacher variable. Then,
\begin{align*}
    \exvv{\cG', \tilde{\cG}}{\abs{\frac{1}{\mtilde}\sum_{h=1}^{\mtilde} \inparen{\frac{g_{\tilde{i}_h}(\vx)^p}{p_{\tilde{i}_h}} -\frac{g_{i'_h}(\vx)^p}{p_{i'_h}}}}} &= \underset{R_1, \ldots, R_{\mtilde}}{\E} \exvv{\cG', \tilde{\cG}}{\frac{1}{\mtilde}\abs{\sum_{h=1}^{\mtilde} R_h \inparen{\frac{g_{\tilde{i}_h}(\vx)^p}{\vrho_{\widetilde{i}_h}} -\frac{g_{i'_h}(\vx)^p}{\vrho_{i'_h}}}}} \\
    &\leq 2 \underset{\cG'}{\E} \exvv{R_1, \ldots, R_M}{\frac{1}{\mtilde}\abs{\sum_{h=1}^{\mtilde} R_h \frac{g_{i'_h}(\vx)^p}{\vrho_{i'_h}}}}.
\end{align*}
This concludes the proof of \Cref{lemma:symmetrization}.
\end{proof}
With \Cref{lemma:symmetrization} in hand, we set up our chaining argument in \Cref{lemma:chaining_setup}. We first confirm that our random process is subgaussian with respect to our choice of $\dtwo$. 
\begin{lemma}[Choosing the distance]
\label{lemma:chaining_setup}
The random process $\sum_{h=1}^{\mtilde} R_{i_h}\cdot\nfrac{g_{i_h}(\vx)^p}{\vrho_{i_h}}$  is subgaussian with respect to $\dtwo$ as defined in  \eqref{eq:chaining_dtwo_defn}.
\end{lemma}
\begin{proof}[Proof of \Cref{lemma:chaining_setup}]
Let
\begin{align*}
    P &\coloneqq \abs{\sum_{h=1}^{\mtilde} R_{h}\inparen{\frac{g_{i_h}(\vx)^p}{\vrho_{i_h}} - \frac{g_{i_h}(\xhat)^p}{\vrho_{i_h}}}}.
\end{align*}
Let us first calculate $\subgnorm{P}$. Using the fact that every term in this sum is independent and \Cref{fact:subgaussian_properties}, we get
\begin{align*}
    \subgnorm{P}^2 = \sum_{h=1}^{\mtilde} \subgnorm{R_{h}\inparen{\frac{g_{i_h}(\vx)^p}{\vrho_{i_h}} - \frac{g_{i_h}(\xhat)^p}{\vrho_{i_h}}}}^2 \le \sum_{h=1}^{\mtilde} 2\inparen{\frac{g_{i_h}(\vx)^p}{\vrho_{i_h}} - \frac{g_{i_h}(\xhat)^p}{\vrho_{i_h}}}^2
\end{align*}
By \Cref{fact:subgaussian_properties}, we have
\begin{align*}
    \prvv{R_{h}}{P \ge v\dtwo(\vx,\xhat)} &\le 2\expv{-\frac{v^2\dtwo(\vx,\xhat)^2}{4\sum_{h=1}^{\mtilde} \inparen{\frac{g_{i_h}(\vx)^p}{\vrho_{i_h}} - \frac{g_{i_h}(\xhat)^p}{\vrho_{i_h}}}^2}} = 2\expv{-\frac{v^2}{2}}.
\end{align*}
This concludes the proof of \Cref{lemma:chaining_setup}.
\end{proof}
\Cref{lemma:chaining_setup} tells us that $\dtwo$ is a choice of distance on $B_{p}$ that allows us to use the subgaussian form of chaining to analyze our random process. Along with the way we have set up our sampling process, we have enough to apply Theorem \ref{thm:generic_chaining_dirksen}. This is simply a restatement of Lemma 2.6 of \cite{jlls23} for our setting.

\begin{theorem}[Restatement of Lemma 2.6 due to \citet{jlls23}, $\alpha=2$]
\label{thm:generic_chaining_dirksen}
Recall $\dtwo$ (\ref{eq:chaining_dtwo_defn}) and $\gamma_2$ (\ref{eq:chaining_gamma2_defn}). Suppose that for some $D$ and for every choice of $i_1,\dots,i_{\mtilde}$, we have
\begin{align*}
    \gamma_2\inparen{B_p,\frac{\dtwo}{\mtilde}} \lesssim D\inparen{\max_{\vx\in B_p} \norm{\vx}_{\cG'}^p}^{1/2},
\end{align*}
where
\begin{align*}
    \norm{\vx}_{\cG'}^p = \frac{1}{\mtilde}\sum_{h=1}^{\mtilde} \frac{g_{i_h}(\vx)^p}{\vrho_{i_h}}.
\end{align*}
Then, we have the following.
\begin{align*}
    \exvv{\cD}{\sup_{\vx\in B_p} \abs{\gnorm{\mA\vx}{p}-\norm{\vx}_{\cG'}^p}} &\lesssim D.
\end{align*}
If we also have for all choices $i_1,\dots,i_{\mtilde}$ and for some $\widehat{D}$ that
\begin{align*}
    \mathsf{diam}\inparen{B_p, \frac{\dtwo}{\mtilde}} \lesssim \widehat{D}\inparen{\max_{\vx\in B_p} \norm{\vx}_{\cG'}^p}^{1/2},
\end{align*}
then there exists a universal constant $C > 0$ such that for all $0 \le t \le \nfrac{1}{2K\widehat{D}}$,
\begin{align*}
    \prvv{\cD}{\sup_{\vx\in B_p} \abs{\gnorm{\mA\vx}{p}^p-\norm{\vx}_{\cG'}^p} \ge C(D + t\widehat{D})} \le \expv{-\frac{Kt^2}{4}}.
\end{align*}
\end{theorem}

With Theorem \ref{thm:generic_chaining_dirksen} in our arsenal, our task becomes to compute $\gamma_2(B_p,\dtwo)$ (which we will then divide by $\mtilde$ so that we can apply \Cref{thm:generic_chaining_dirksen}). Hereafter, we will simply abbreviate $\gamma_2(B_p,\dtwo)$ as $\gamma_2$. We will first weaken the definition of $\gamma_2$, which is essentially equivalent to Dudley's integral. Recall the definition of the entropy numbers $e_N$ (\Cref{defn:entropy_numbers}) and notice that
\begin{align*}
    \gamma_2 \le \sum_{N \ge 0} 2^{N/2}e_N(B_{p},\dtwo). 
\end{align*}
We now rewrite $\dtwo$ in a form that will be more convenient for us.

\begin{lemma}
\label{lemma:concentration_g2_rewrite_d}
We have
\begin{align*}
    \dtwo(\vx, \xhat) \le \max(p,2) \inparen{\Fstar}^{\max(0,p/4-1/2)}{\mtilde}^{1/2} \cdot \polynorm{\vx-\xhat}^{\min(p/2,1)} \cdot \inparen{\max_{\vx \in B_{p}}\norm{\vx}_{\cG'}^p}^{1/2}
\end{align*}
and therefore
\begin{align*}
    \gamma_2 &\le {\mtilde}^{1/2}\max(p,2) \inparen{H^{1/p}\inparen{\Fstar}^{1/2}}^{\max(0,p/2-1)}\inparen{\max_{\vx \in B_{p}}\norm{\vx}_{\cG'}^p}^{1/2}\sum_{N \ge 0} 2^{N/2}e_N\inparen{B_p, \polynorm{\cdot}}^{\min(p/2,1)}.
\end{align*}
\end{lemma}
\begin{proof}[Proof of \Cref{lemma:concentration_g2_rewrite_d}]
We have two cases. We first address the case $0 < p < 2$. Recall that in this regime, we have $\abs{a^{p/2}-b^{p/2}} \le \abs{a-b}^{p/2}$. Since $g_{i_h}(\vx)$ is a norm, the triangle inequality tells us that $\abs{g_{i_h}(\vx) - g_{i_h}(\xhat)} \le g_{i_h}(\vx-\xhat)$. We use these and write
\begin{align*}
    \dtwo(\vx, \xhat) &= \inparen{\sum_{h=1}^{\mtilde} \inparen{\frac{g_{i_h}(\vx)^{p/2}}{\sqrt{\vrho_{i_h}}} - \frac{g_{i_h}(\xhat)^{p/2}}{\sqrt{\vrho_{i_h}}}}^2\inparen{\frac{g_{i_h}(\vx)^{p/2}}{\sqrt{\vrho_{i_h}}} + \frac{g_{i_h}(\xhat)^{p/2}}{\sqrt{\vrho_{i_h}}}}^2}^{1/2} \\
    &\le \inparen{\sum_{h=1}^{\mtilde} \inparen{\frac{g_{i_h}(\vx-\xhat)^{p/2}}{\sqrt{\vrho_{i_h}}}}^2\inparen{\frac{g_{i_h}(\vx)^{p/2}}{\sqrt{\vrho_{i_h}}} + \frac{g_{i_h}(\xhat)^{p/2}}{\sqrt{\vrho_{i_h}}}}^2}^{1/2} \\
    &\le \polynorm{\vx-\xhat}^{p/2}\cdot\inparen{\sum_{h=1}^{\mtilde} \inparen{\frac{g_{i_h}(\vx)^{p/2}}{\sqrt{\vrho_{i_h}}} + \frac{g_{i_h}(\xhat)^{p/2}}{\sqrt{\vrho_{i_h}}}}^2}^{1/2} \\
    &\le 2{\mtilde}^{1/2} \cdot \polynorm{\vx-\xhat}^{p/2}\inparen{\max_{\vx \in B_{p}}\norm{\vx}_{\cG'}^p}^{1/2}
\end{align*}
which concludes the proof in the range $0 < p < 2$.

We now move onto the case where $p \ge 2$. Recall that by Lipschitzness, we get
\begin{align*}
    \abs{a^{p/2}-b^{p/2}} \le \frac{p}{2} \cdot \max(a,b)^{p/2-1}\abs{a-b}.
\end{align*}
Next, by \Cref{lemma:containments}, we know that for all $i$, 
\begin{align*}
    \norm{\mA_{S_i}\vx}_{p_i} \le \inparen{\Fstar}^{1/2}\vrho_i^{1/p}H^{1/p}\gnorm{\mA\vx}{p} \le \inparen{\Fstar}^{1/2}\vrho_i^{1/p}H^{1/p}.
\end{align*}
We use these to rewrite $\dtwo(\vx,\xhat)$.
\begin{align*}
    \dtwo(\vx, \xhat) &= \inparen{\sum_{h=1}^{\mtilde} \inparen{\frac{g_{i_h}(\vx)^{p/2}}{\sqrt{\vrho_{i_h}}} - \frac{g_{i_h}(\xhat)^{p/2}}{\sqrt{\vrho_{i_h}}}}^2\inparen{\frac{g_{i_h}(\vx)^{p/2}}{\sqrt{\vrho_{i_h}}} + \frac{g_{i_h}(\xhat)^{p/2}}{\sqrt{\vrho_{i_h}}}}^2}^{1/2} \\
    &\le \frac{p}{2}\cdot H^{1/2-1/p}\inparen{\Fstar}^{p/4-1/2}\inparen{\sum_{h=1}^{\mtilde} \inparen{\vrho_{i_h}^{-1/p}g_{i_h}(\vx-\xhat)}^2\inparen{\frac{g_{i_h}(\vx)^{p/2}}{\sqrt{\vrho_{i_h}}} + \frac{g_{i_h}(\xhat)^{p/2}}{\sqrt{\vrho_{i_h}}}}^2}^{1/2} \\
    &\le \frac{p}{2}\cdot H^{1/2-1/p}\inparen{\Fstar}^{p/4-1/2}\polynorm{\vx-\xhat}\cdot\inparen{\sum_{h=1}^{\mtilde} \inparen{\frac{g_{i_h}(\vx)^{p/2}}{\sqrt{\vrho_{i_h}}} + \frac{g_{i_h}(\xhat)^{p/2}}{\sqrt{\vrho_{i_h}}}}^2}^{1/2} \\
    &\le p \cdot H^{1/2-1/p}\inparen{\Fstar}^{p/4-1/2}{\mtilde}^{1/2} \cdot \polynorm{\vx-\xhat}\inparen{\max_{\vx \in B_{p}}\norm{\vx}_{\cG'}^p}^{1/2}
\end{align*}
We conclude the proof of \Cref{lemma:concentration_g2_rewrite_d}.
\end{proof}

In light of \Cref{lemma:concentration_g2_rewrite_d}, observe that it is enough to calculate each term of the sum $\sum_{N \ge 0} 2^{N/2}e_N(B_p, \polynorm{\cdot})^{\min(p/2,1)}$. We begin this analysis with \Cref{lemma:concentration_g2_two_one_volume}.

\begin{lemma}
\label{lemma:concentration_g2_two_one_volume}
For all $N \ge 0$, we have
\begin{align*}
    2^{N/2}e_N\inparen{B_p, \polynorm{\cdot}}^{\min(p/2,1)} \le 2^{N/2 + \min(p/2,1)(2+(1/p)\log C - 2^N/n)},
\end{align*}
where $C$ is such that for all $i \in S$ and $\vx\in\R^n$, we have $\inparen{C\vrho_i}^{-1/p}\norm{\mA_{S_i}\vx}_{p_i} \le \gnorm{\mA\vx}{p}$.
\end{lemma}
\begin{proof}[Proof of \Cref{lemma:concentration_g2_two_one_volume}]
Recall that by \Cref{lemma:covering_easy}, we have
\begin{align*}
    \log\cN\inparen{B_p,\eta\inbraces{\vx\in\R^n\suchthat\polynorm{\vx}\le 1}} \le n\logv{\frac{4C^{1/p}}{\eta}}.
\end{align*}
We set $\eta = 2^{2+(1/p)\log C - 2^N/n}$ so that
\begin{align*}
    \log\cN\inparen{B_p,\eta\inbraces{\vx\in\R^n\suchthat\polynorm{\vx}\le 1}} \le 2^N.
\end{align*}
Then,
\begin{align*}
    2^{N/2}e_N\inparen{B_p, \polynorm{\cdot}}^{\min(p/2,1)} &\le 2^{N/2 + \min(p/2,1)(2+(1/p)\log C - 2^N/n)},
\end{align*}
concluding the proof of Lemma \ref{lemma:concentration_g2_two_one_volume}.
\end{proof}

Using \Cref{lemma:concentration_g2_two_one_volume}, we get a rapidly converging tail in our summation for large values of $N$. See \Cref{lemma:g2_tail}.

\begin{lemma}
\label{lemma:g2_tail}
Let $N_S \coloneqq \logv{6n\log n}$. We have
\begin{align*}
    \sum_{N \ge 0} 2^{N/2}e_N\inparen{B_p, \polynorm{\cdot}}^{\min(p/2,1)} &\lesssim \sqrt{nH^{1/\max(2,p)}\inparen{\Fstar}^{1/2}\log n}\\
    &\quad+\sum_{N \le N_S} 2^{N/2}e_N\inparen{B_p, \polynorm{\cdot}}^{\min(p/2,1)}.
\end{align*}
\end{lemma}
\begin{proof}[Proof of \Cref{lemma:g2_tail}]
For now, let $C$ be such that for all $i \in S$ and $\vx\in\R^n$, we have $\inparen{C\vrho_i}^{-1/p}\norm{\mA_{S_i}\vx}_{p_i} \le \gnorm{\mA\vx}{p}$.

Let $N_S$ be a threshold such that for all $N \ge N_S$, we use the entropy number bound given by \Cref{lemma:concentration_g2_two_one_volume} (as the volume-based covering number bound is much better for small values of $e_N$). Let us enforce the constraint $N_S \ge \ceil{\logv{\nfrac{3n}{\min(p,2)}}}$, so that the entropy number bound in \Cref{lemma:concentration_g2_two_one_volume} is decreasing in $N$ and is dominated from above by a geometric series with common ratio $\nfrac{1}{2}$.

We now set $N_S = \logv{6n\log n}$. Since $p \ge \nfrac{1}{\log n}$, we know that $N_S \ge \ceil{\logv{\nfrac{3n}{\min(p,2)}}}$. Now, since $2^{N/2}e_N$ is bounded above by a geometric series with common ratio $\nfrac{1}{2}$, the summation for all $N \ge N_S$ is dominated by the first term. Let us evaluate this. We first observe that
\begin{align*}
    \frac{N_S}{2} + \min\inparen{\frac{p}{2},1}\inparen{-\frac{2^{N_S}}{n} + 2 + \frac{\log C}{p}} &= \frac{\logv{6n\log n}}{2} + \min\inparen{\frac{p}{2},1}\inparen{-\frac{2^{\logv{6n\log n}}}{n} + 2 +\frac{\log C}{p}} \\
    &= \frac{\logv{6n\log n}}{2} + \min\inparen{\frac{p}{2},1}\inparen{-6\log n + 2 +\frac{\log C}{p}} \\
    &\le \frac{\logv{6n\log n}}{2} + \min\inparen{\frac{p}{2},1}\inparen{-6\log n + 2 +\frac{\log C}{p}} \\
    &\le \frac{\logv{6n\log n}}{2} + \frac{\log C}{\max(2,p)} \le \frac{6nC\log n}{2}
\end{align*}
By \Cref{lemma:concentration_g2_two_one_volume}, we see that
\begin{align*}
    2^{N_S/2}e_{N_S}\inparen{B_p, \polynorm{\cdot}}^{\min(p/2,1)} \lesssim \sqrt{nC^{1/\max(2,p)}\log n},
\end{align*}
and by \Cref{lemma:containments}, we can choose
\begin{align*}
    C = H\inparen{\Fstar}^{\max(1,p/2)}.
\end{align*}
We plug this in, account for the remaining terms in the summation, and conclude the proof of \Cref{lemma:g2_tail}.
\end{proof}

We now give another way to evaluate the terms of our summation when the indices $N$ are such that the entropy numbers are rather large. See \Cref{lemma:concentration_g2_two_one_sudakov}.

\begin{lemma}
\label{lemma:concentration_g2_two_one_sudakov}
For all $N \ge 0$, we have
\begin{align*}
    2^{N/2}e_N\inparen{B_p, \polynorm{\cdot}}^{\min(p/2,1)} \le \inparen{C(p) \cdot \max_i\min\inbraces{p_i,\log\abs{S_i}}H^{2/\max(2,p)}\Fstar\logv{\Fstar+\mtilde}}^{1/2},
\end{align*}
where $C(p)$ is a constant that only depends on $p$.
\end{lemma}
\begin{proof}[Proof of \Cref{lemma:concentration_g2_two_one_sudakov}]
For now, let
\begin{align*}
    f(\Fstar,\cG) \coloneqq C(p) \cdot \max_i\min\inbraces{p_i,\log\abs{S_i}}H^{2/\max(2,p)}\Fstar\logv{\Fstar+\mtilde}.
\end{align*}
By \Cref{thm:general_covering}, we have
\begin{align*}
    \log \cN\inparen{B_{p}, \eta \cdot \inbraces{\vy \in \R^n \suchthat \polynorm{\vx} \le 1}} \lesssim \eta^{-\min(2,p)} \cdot f(\Fstar,\cG).
\end{align*}
so if we choose, for some universal constant $C_0$,
\begin{align*}
    \eta = C_0 \cdot 2^{-N/\min(p,2)}\inparen{f(\Fstar,\cG)}^{\max(1/2,1/p)},
\end{align*}
then we get
\begin{align*}
    \log \cN\inparen{B_{p}, \eta \cdot \inbraces{\vy \in \R^n \suchthat \polynorm{\vx} \le 1}} \le 2^N.
\end{align*}
Thus, $e_N\inparen{B_p, \polynorm{\cdot}} \le \eta$. Exponentiating and substituting the definition of $f(\Fstar,\cG)$ concludes the proof of \Cref{lemma:concentration_g2_two_one_sudakov}.
\end{proof}

We now show how to complete the sum by combining \Cref{lemma:concentration_g2_two_one_sudakov} and \Cref{lemma:g2_tail}.

\begin{lemma}
\label{lemma:concentration_subgaussian}
We have
\begin{align*}
    \sum_{N \ge 0} 2^{N/2}e_N\inparen{B_p, \polynorm{\cdot}}^{\min(p/2,1)} &\lesssim \sqrt{nH^{1/\max(2,p)}\inparen{\Fstar}^{1/2}\log n} \\
    &\quad +\log n\inparen{C(p) \cdot \max_i\min\inbraces{p_i,\log\abs{S_i}}H\Fstar\logv{\Fstar+\mtilde}}^{1/2}.
\end{align*}
\end{lemma}
\begin{proof}[Proof of \Cref{lemma:concentration_subgaussian}]
Noting that $N_S = \logv{6n\log n} \lesssim \log n$, we combine the conclusions of \Cref{lemma:concentration_g2_two_one_sudakov} and \Cref{lemma:g2_tail} to obtain the statement of \Cref{lemma:concentration_subgaussian}.
\end{proof}

Finally, we translate Lemma \ref{lemma:concentration_subgaussian} into an upper bound on the process we started with by using \Cref{thm:generic_chaining_dirksen}. We then use this to complete the proof of Theorem \ref{thm:concentration}.

\begin{proof}[Proof of Theorem \ref{thm:concentration}]
As we have done in previous proofs, as a shorthand, we define
\begin{align*}
    p^{\star} \coloneqq \max_{i}\min\inbraces{p_i,\log\abs{S_i}}.
\end{align*}
We first weaken the statement of \Cref{lemma:concentration_subgaussian} to read
\begin{align*}
    &\quad\sum_{N \ge 0} 2^{N/2}e_N\inparen{B_p, \polynorm{\cdot}}^{\min(p/2,1)} \\
    &\lesssim \max\inbraces{n,\Fstar}H^{1/2\max(2,p)}\sqrt{\log n} \\
    &\quad+ \log n \inparen{C(p) \cdot H^{2/\max(2,p)}p^{\star}\max\inbraces{n,\Fstar}\logv{n + \Fstar + \mtilde}}^{1/2} \\
    &\lesssim \log n \inparen{C(p) \cdot H^{2/\max(2,p)}p^{\star}\max\inbraces{n,\Fstar}\logv{n + \Fstar + \mtilde}}^{1/2}.
\end{align*}
Let $V$ denote the right hand side of the above. Combining this rewrite with \Cref{lemma:concentration_g2_rewrite_d}, we get
\begin{align*}
    \gamma_2 \lesssim {\mtilde}^{1/2}\max(p,2) \inparen{H^{1/p}\inparen{\Fstar}^{1/2}}^{\max(0,p/2-1)}V\inparen{\max_{\vx \in B_{p}}\norm{\vx}_{\cG'}^p}^{1/2}.
\end{align*}
By the symmetrization reduction (\Cref{lemma:symmetrization}) and \Cref{thm:generic_chaining_dirksen}, we have
\begin{align*}
    \exv{\sup_{\vx\in B_p} \abs{\gnorm{\mA\vx}{p}-\norm{\vx}_{\cG'}^p}} \lesssim \frac{\max(p,2) \inparen{H^{1/p}\norm{\valpha}_p}^{\max(0,p/2-1)}V}{\mtilde^{1/2}}
\end{align*}
and so to make the RHS upper-bounded by $\eps$, it is sufficient to set $\mtilde$ according to
\begin{align*}
    \mtilde &\asymp \frac{\max(p,2)^2 \inparen{H^{1/p}\inparen{\Fstar}^{1/2}}^{\max(0,p-2)}V^2}{\eps^2}.
\end{align*}
This means that when $p < 2$, we have
\begin{align*}
    \mtilde &\asymp \frac{(\log n)^2\cdot C(p) \cdot Hp^{\star}\max\inbraces{n,\Fstar}\logv{n + \Fstar + \mtilde}}{\eps^2} \\
    &\asymp \frac{(\log n)^2\logv{\max\inbraces{n,\Fstar}/\eps}\cdot C(p) \cdot Hp^{\star}\max\inbraces{n,\Fstar}}{\eps^2},
\end{align*}
which is what we desired.

For $p \ge 2$, we have
\begin{align*}
    \mtilde &\asymp \frac{p^2H^{1-2/p}\inparen{\Fstar}^{p/2-1}\inparen{(\log n)^2H^{2/p}p^{\star}\max\inbraces{n,\Fstar}\logv{n+\Fstar+\mtilde}}}{\eps^2} \\
    &\asymp \frac{(\log n)^2\logv{\max\inbraces{n,\Fstar}/\eps}\cdot p^2 \cdot Hp^{\star}\max\inbraces{n,\Fstar}^{p/2}}{\eps^2}
\end{align*}
To bound $\mathsf{diam}(B_p,\dtwo)$, by the triangle inequality, it is enough to estimate $\dtwo(\vx,0)$ for all $\vx \in B_p$. Recalling \Cref{lemma:containments}, we have
\begin{align*}
    \dtwo(\vx,0) &= \inparen{\sum_{h=1}^{\mtilde} \inparen{\frac{g_{i_h}(\vx)^p}{\vrho_{i_h}} - \frac{g_{i_h}(0)^p}{\vrho_{i_h}}}^2}^{1/2} = \inparen{\sum_{h=1}^{\mtilde} \inparen{\frac{g_{i_h}(\vx)^p}{\vrho_{i_h}}}\cdot\inparen{\frac{g_{i_h}(\vx)^p}{\vrho_{i_h}}}}^{1/2} \\
    &\le \inparen{\sum_{h=1}^{\mtilde} H\inparen{\Fstar}^{\max(1,p/2)}\cdot\inparen{\frac{g_{i_h}(\vx)^p}{\vrho_{i_h}}}}^{1/2} \le \mtilde^{1/2}H^{1/2}\inparen{\Fstar}^{\max(1/2,p/4)}\max_{\vx\in B_p}\inparen{\norm{\vx}_{\cG'}^{p}}^{1/2},
\end{align*}
which means that we may set $\widehat{D}$ in \Cref{thm:generic_chaining_dirksen} according to
\begin{align*}
    \widehat{D} \asymp \frac{H^{1/2}\inparen{\Fstar}^{\max(1/2,p/4)}\max_{\vx\in B_p}\inparen{\norm{\vx}_{\cG'}^{p}}^{1/2}}{\mtilde^{1/2}}.
\end{align*}
We now verify that if we choose
\begin{align*}
    \mtilde \asymp \vbrho \cdot H\max_{i \in [m]} \min(p_i,\log\abs{S_i}) \inparen{\Fstar}^{\max(1,p/2)}\logv{\nfrac{1}{\delta}},
\end{align*}
that we indeed get for some universal constant $C$ that
\begin{align*}
    \prvv{\cD}{\max_{\vx\in B_p} \abs{\gnorm{\mA\vx}{p}^p-\norm{\vx}_{\cG'}^p} \ge C\eps} \lesssim \delta.
\end{align*}
We rescale $\eps$ appropriately and conclude the proof of \Cref{thm:concentration}.
\end{proof}

\section{Applications and algorithms}
\label{sec:applications}
At this point in the paper, we are ready to prove our main results (\Cref{thm:one_shot_lewis} and \Cref{thm:computeblw}).

\subsection{Block norm approximations via block Lewis weights (Proof of Theorem \ref{thm:one_shot_lewis})}
\label{sec:lewis}

We restate and prove the main result of the paper.
\oneshotlewis*
\begin{proof}[Proof of Theorem \ref{thm:one_shot_lewis}]
Observe that \Cref{lemma:block_instantiation_small} proves the existence of a probability measure $\vlambda$ over $[m]$ and a rounding $\mW$ that are $\Fstar$-block Lewis overestimates for $\Fstar = n$ if $p_i \ge 2$, and \Cref{lemma:block_instantiation_alleq} proves the existence of a probability measure $\widehat{\vlambda}$ over $[k]$ and corresponding $\widehat{\valpha} \in \R^k$ such that we get an $\Fstar = n$-Lewis overestimate.

We now apply \Cref{thm:concentration} and conclude the proof of Theorem \ref{thm:one_shot_lewis}.
\end{proof}

\subsection{Efficient computation of block Lewis weight overestimates (Proof of \texorpdfstring{\Cref{thm:computeblw}}{Theorem 2})}
\label{sec:alg}

In this subsection, we restate and prove \Cref{thm:computeblw}.
\computeblw*
We break up the proof into two sections -- one where $p_1=\dots=p_m=2$ and $p > 0$ and another where $p=2$ and $p_1,\dots,p_m \ge 2$.

\subsubsection{Special case -- \texorpdfstring{$p > 0, p_1=\dots=p_m=2$}{p>0,p1=...=pm=2}}

Following \Cref{defn:block_lewis_overestimate} and the discussion in \Cref{sec:blw_defs_covering}, observe that when $p_1=\dots=p_m=2$, we have \[\beta_i(\mV) \coloneqq \inparen{\sum_{j \in S_i} \va_j^{\top}(\mA^{\top}\mV\mA)^{-1}\va_j}^{1/2}.\]
As before, we call \(\beta_i(\mV)^p\) block Lewis weights.
Given \(\vb \in \R^m\), we let \(\mB \in \R^{k}\) be a diagonal matrix given by
\(\mB_{jj} = \vb_i\) for all \(i \in [m]\), \(j \in S_i\).
First, let us specialize the definition of block Lewis weight overestimates (recall \Cref{defn:block_lewis_overestimate}) to this special case.
\begin{definition}\label{def:lev_score_over}
    For \(\nu \ge 0\), we say \(\vb \in \R_{\geq 0}^m\) is a vector of \(\nu\)-bounded block Lewis weight overestimates for \(\mA\) if
    \[\|\vb\|_1 \le \nu,\]
    and for all \(i \in [m]\),
    \[\vb_i \ge \beta_i(\mB^{1-2/p})^p.\]
\end{definition}

We can think of the definition of block Lewis weight overestimates as being a relaxation of the fixed point condition for block Lewis weights that is described in \cite[Page 31, Proof of Lemma 4.2]{jlls23}. 

As a primitive, our algorithm will use \textit{leverage score overestimates} (see \cite[Definition 2.2]{jls22}). They are approximate forms of leverage scores \(\tau_j(\mM)\).
\begin{definition}
    For \(\nu \ge 0\), we say \(\widetilde{\tau} \in \R^k_{\ge 0}\) is a vector of \(\nu\)-bounded leverage score overestimates for \(\mM \in \R^{k \times n}\) if
    \[\|\widetilde{\tau}\|_1 \le \nu\]
    and for all \(i \in [k]\),
    \[\widetilde{\tau}_i \ge \tau_i(\mM).\]
\end{definition}

There are known efficient algorithms for computing leverage score overestimates or reducing leverage score computations to linear system solves.

\begin{theorem}[\protect{\cite[Theorem 3]{jls22}}]\label{thm:lev_ove}
There is an algorithm \(\textsc{OverLev}\) that, given \(\mM \in \R^{k \times n}\), produces \(O(n)\)-bounded leverage score overestimates for \(\mM\) in \(\widetilde{O}(\nnz{\mM} + n^\omega)\) time, where \(\omega\) is the matrix multiplication exponent.
\end{theorem}

We will use two different algorithms depending on the value of $p$. If $p \le 2$, then we present a contractive scheme reminiscent of the algorithm of \citet{cp15}. If $p > 2$, we present an algorithm similar to those of \citet{ccly19} and \citet{jls21}.

We begin with the case where $0 < p \le 2$ (in fact, we will see that this algorithm yields guarantees where $p < 4$). The main objects of interest here are \Cref{alg:blw_lewismap} and \Cref{lemma:contract_alg}.

\begin{algorithm}
\caption{Algorithm to compute block Lewis weight overestimates, $0 < p < 4$ and $p_1=\dots=p_m=2$.}
\begin{algorithmic}[1]\label{alg:blw_lewismap}
\State \textbf{Input:} \(\mA \in \R^{k \times n}\), outer norm $0 < p < 4$, group structure \((S_1,\dots,S_m,2,\dots,2)\).
\State Initialize \(\vb^{(0)} = \frac{n}{m} \cdot \onev_m\).
\State Define $\psi$ such that\label{line:psi_def}
\begin{align*}
    \psi_i(\vb) \coloneqq \inparen{\inparen{\sum_{j \in S_i} \inparen{\va_j^{\top}\inparen{\mA^{\top}\mB^{1-2/p}\mA}^{-1}\va_j}^{p_i/2}}^{2/p_i}}^{p/2}.
\end{align*}
\For{\(t = 1, \ldots, T\)}
    \State \(\vb^{(t)} = \max(\psi(\vb^{(t-1)}),\onev_{m} \cdot 1/m)\) \Comment{The $\max$ is taken elementwise here.}
\EndFor
\State \Return \(1.1\vb^{(T)}\)
\end{algorithmic}
\end{algorithm}

\begin{lemma}
\label{lemma:contract_alg}
Let
\begin{align*}
    T \ge \frac{\ln\inparen{\frac{\ln\inparen{\frac{m}{n}}}{\ln\inparen{1+\eps}}}}{\ln\inparen{\abs{\frac{2}{2-p}}}}.
\end{align*}
Then, the weights $1.1\vb^{(T)}$ output by \Cref{alg:blw_lewismap}  are a $1.1n$-block Lewis overestimate (\cref{defn:block_lewis_overestimate}). Furthermore, computing $\vb^{(T)}$ requires at most $T$ computations of the vector whose entries are the $\va_j^{\top}\inparen{\mA^{\top}\mD\mA}^{-1}\va_j$ for all $j$, where $\mD$ is a diagonal matrix.
\end{lemma}

To prove \Cref{lemma:contract_alg}, we first state and prove \Cref{lemma:lewis_contract_easy}.

\begin{lemma}
\label{lemma:lewis_contract_easy}
Let $\psi$ be as defined in Line \ref{line:psi_def} in \Cref{alg:blw_lewismap}. For all $\vu \in \R^{m}_{\ge 0}$ and $\vv \in \R^{m}_{\ge 0}$, we have
\begin{align*}
    \max_{1 \le i \le m} \abs{\lnv{\frac{\psi_i(\vu)}{\psi_i(\vv)}}} \le \abs{\frac{p}{2}-1}\max_{1 \le i \le m} \abs{\lnv{\frac{\vu_i}{\vv_i}}},
\end{align*}
and therefore $\psi(\vu)$ is a contraction whenever $0 < p < 4$.
\end{lemma}
\begin{proof}[Proof of \Cref{lemma:lewis_contract_easy}]
Fix some index $i \le m$. For notational simplicity in this proof, let $\alpha$ be such that $\lnv{\alpha} \coloneqq \triangle(\vu,\vv)$.

This easily implies that
\begin{align*}
    \frac{1}{\alpha^{\abs{1-2/p}}} \cdot \va_j^{\top}\inparen{\mA^{\top}\mV^{1-2/p}\mA}^{-1}\va_j \le \va_j^{\top}\inparen{\mA^{\top}\mU^{1-2/p}\mA}^{-1}\va_j \le \alpha^{\abs{1-2/p}} \cdot \va_j^{\top}\inparen{\mA^{\top}\mV^{1-2/p}\mA}^{-1}\va_j.
\end{align*}
We take the $p_i/2$-norm and take the $p/2$ power, which tells us that for all $1 \le i \le m$,
\begin{align*}
    \frac{1}{\alpha^{\abs{p/2-1}}} \cdot \psi_i(\vv) \le \psi_i(\vu) \le \alpha^{\abs{p/2-1}} \cdot \psi_i(\vv).
\end{align*}
Hence,
\begin{align*}
    \max_{1 \le i \le m} \abs{\lnv{\frac{\psi_i(\vu)}{\psi_i(\vv)}}} \le \abs{\frac{p}{2}-1}\lnv{\alpha} = \abs{\frac{p}{2}-1}\max_{1 \le i \le m} \abs{\lnv{\frac{\vu_i}{\vv_i}}},
\end{align*}
completing the proof of \Cref{lemma:lewis_contract_easy}.
\end{proof}

We are now ready to complete the proof of \Cref{lemma:contract_alg}.

\begin{proof}[Proof of \Cref{lemma:contract_alg}]
The computational complexity guarantee is immediate, so we focus on the approximation guarantee.

By \Cref{lemma:lewis_contract_easy} and the Banach fixed point theorem, we know that $\psi$ has a unique fixed point. Denote this fixed point by $\vb^{\star}$. We would like to argue that since the convergence in the $\ln$-metric is linear, it takes roughly $\log\log(1+\eps)/\log\abs{2/(2-p)}$ applications of $\psi$ to reach a $(1+\eps)$-multiplicative approximation to $\vb^{\star}$. An annoying technicality is that if $\vb^{\star}$ has some elements arbitrarily close to $0$, then the convergence rate could be very slow. To fix this, we simply enforce that the coordinates of the iterates never drop below $1/m$. It is easy to see that this only overestimates the true weights and therefore does not affect our sampling guarantees (in particular, $\norm{\max(\vb^{\star}, \onev_m \cdot 1/m)}_1 \le n + 1$).

More precisely, let $\vb \coloneqq \max(\vb^{\star}, \onev_m \cdot 1/m)$. Notice that for all $1 \le i \le m$, $\abs{\ln\inparen{\nfrac{\vb_i^{(0)}}{\vb_i}}} \le \ln\inparen{\nfrac{m}{n}}$. This means that after $T$ iterations, we have
\begin{align*}
    \max_{1 \le i \le m} \abs{\ln\inparen{\frac{\vb^{(T)}_i}{\vb_i}}} \le \abs{\frac{p}{2}-1}^T\max_{1 \le i \le m} \abs{\ln\inparen{\frac{\vb_i^{(0)}}{\vb_i}}} \le  \abs{\frac{p}{2}-1}^T\ln\inparen{\frac{m}{n}}.
\end{align*}
Choosing
\begin{align*}
    T \ge \frac{\ln\inparen{\frac{\ln\inparen{\frac{m}{n}}}{\ln\inparen{1+\eps}}}}{\ln\inparen{\abs{\frac{2}{2-p}}}}
\end{align*}
and observing that for sampling that it is sufficient to choose $\eps = 0.1$ implies that $\vb^{(T)}$ is an entrywise $1.1$-approximation to $\vb$. As this is sufficient to get the concentration in the setting of \Cref{thm:concentration}, we may  complete the proof of \Cref{lemma:contract_alg}.
\end{proof}

Now, we move onto the case where $p \ge 2$. This covers the cases of $p$ where \Cref{alg:blw_lewismap} is not a contraction (whenever $p \ge 4$). In this setting, we have \Cref{alg:blw}. At a high level, observe that Line \ref{line:alg_iter_next} of \Cref{alg:blw} performs a fixed point iteration on the stationary condition \(\vb_i = \beta_i(\mB^{1-2/p})^p\) that holds for the optimal choice of block Lewis weights (see \cite[proof of Lemma 4.2 and (4.5)]{jlls23}).

\begin{algorithm}[H]
\caption{Algorithm to compute block Lewis weight overestimates, $p \ge 2$}
\begin{algorithmic}[1]\label{alg:blw}
\State \textbf{Input:} \(\mA \in \R^{k \times n}\)
\State \textbf{Output:}
\State Initialize \(\vb^{(1)} = \frac{n}{m} \cdot \onev\) \label{line:alg_init}
\For{\(t = 1, \ldots, T - 1\)}
    \State \(\mB^{(t)}_{jj} = \vb^{(t)}_i\) for all \(i \in [m]\), \(j \in S_i\)
    \State \(\widetilde{\tau}^{(t)} = \textsc{OverLev}((\mB^{(t)})^{1/2 - 1/p} \mA)\)\label{line:alg_iter_lev}
    \State \(\vb^{(t+1)}_i = \sum_{j \in S_i} \widetilde{\tau_j}\) for all \(i \in [m]\)\label{line:alg_iter_next}
\EndFor
\State \(\overline{\vb} = \frac{1}{T} \sum_{t=1}^T \vb^{(t)}\)
\State \Return \(\vb = \frac{3}{2} \overline{\vb}\)
\end{algorithmic}
\end{algorithm}

The guarantee we obtain for \Cref{alg:blw} is captured by \Cref{thm:blw}.

\begin{lemma}
\label{thm:blw}
The return value \(\vb\) of  \Cref{alg:blw} is a vector of \(O(n)\)-bounded block Lewis weight overestimates.
Further, this vector of overestimates is found in \(\text{polylog}(k, m, n)\) leverage score overestimate computations.
\end{lemma}

The goal of the rest of this section is to analyze \Cref{alg:blw} and to prove \Cref{thm:blw}. Let us briefly describe the analysis of \Cref{alg:blw}. We first give a collection of potential functions \(\phi_i(\vb)\), with the goal of showing the potential of \(\phi_i(\overline{\vb})\) decreases with \(T\). A low potential will also imply that \(\overline{\vb}\) is nearly a vector of block Lewis weight overestimates.
For each \(i \in [m]\) define \(\phi_i \colon \R^m \to \R\) by
\[\phi_i(\vb) \coloneqq \ln\inparen{\frac{1}{\vb_i} \sum_{j \in S_i} \tau_j(\mB^{1/2 - 1/p} \mA)}.\]

The key property of this potential is convexity, which we now show.
\begin{lemma}
\label{lemma:alg_phi_convex}
    For each \(i \in [m]\), \(\phi_i\) is convex.
\end{lemma}
\begin{proof}[Proof of \Cref{lemma:alg_phi_convex}]
     Our argument for the convexity of this function is similar to the one given in \cite[Lemma A.2]{jls21}.
     First, notice that by the definition of \(\tau_j\), \(\phi(\vb)\) is equal to
     \begin{align*}
         \phi_i(\vb) = \ln\inparen{\frac{1}{\vb_i^{2/p}} \sum_{j \in S_i} \va_j^\top \inparen{\mA^\top \mB^{1 - 2/p} \mA}^{-1} \va_j}.
     \end{align*}
     Since \(- \frac{2}{p} \ln(\vb_i)\) is convex, it suffices to show the convexity of
     \[f(\vb) \coloneqq \ln\inparen{\sum_{j \in S_i} \va_j^\top \inparen{\mA^\top \mB^{1 - 2/p} \mA}^{-1} \va_j)}.\]
     Now we define for each \(j \in [k]\) the function
     \[h_j(\vb) \coloneqq \ln \inparen{\va_j^\top (\mA^\top \mB^{1-2/p} \mA)^{-1} \va_j}.\]
     A version of this function without repeated entries in \(\mB\) was shown to be convex in \cite[Lemma A.2]{jls21}, but \(h_j(\vb)\) is still convex.
     Next, notice that we may write \[
     f(\vb) = \ln\inparen{\sum_{j \in S_i} \exp(h_j(\vb))}.
     \]
     Now taking \(\vb, \vb' \in \R^m\) and \(\lambda \in [0, 1]\) we have
     \begin{align}
        f(\lambda \vb + (1-\lambda) \vb') &= \ln\inparen{\sum_{j \in S_i} \exp(h_j(\lambda \vb + (1-\lambda) \vb'))} \nonumber
        \\
     &\le \ln\inparen{\sum_{j \in S_i} \exp(\lambda h_j(\vb) + (1-\lambda) h_j(\vb'))} \label{eq:alg_ph_cvx_1}\\
     &\le \lambda \ln\inparen{\sum_{j \in S_i} \exp(h_j(\vb))} + (1-\lambda) \ln\inparen{\sum_{j \in S_i} \exp(h_j(\vb'))} \label{eq:alg_ph_cvx_2} \\
     &= \lambda f(\vb) + (1-\lambda) f(\vb'), \nonumber
     \end{align}
     where \eqref{eq:alg_ph_cvx_1} follows from the convexity of \(h_j\) and the monotonicity of log-sum-exp, and \eqref{eq:alg_ph_cvx_2} is due to the convexity of log-sum-exp (see e.g. \cite[Section 3.1.5]{boyd2004convex}). Hence, we may conclude the proof of \Cref{lemma:alg_phi_convex}.
\end{proof}

We now give an argument that \(\phi_i(\overline{\vb}) = \widetilde{O}(1/T)\) using the convexity of \(\phi_i\).

\begin{lemma}\label{lemma:alg_analysis1}
    Assume that \(\textsc{OverLev}\) returns \(\nu\)-bounded leverage score overestimates.
   Then after \Cref{alg:blw} runs for \(T\) iterations, we have for all \(i \in [m]\) that
    \[\phi_i(\overline{\vb}) \le \frac{1}{T} \ln\inparen{\frac{m \nu}{n}}.\]
\end{lemma}
\begin{proof}[Proof of \Cref{lemma:alg_analysis1}]
    We have
    \begin{align*}
        \phi_i(\overline{\vb}) &\le \frac{1}{T} \sum_{t=1}^T \phi_i(\vb^{(t)}) &\text{Jensen's inequality} \\
        &= \frac{1}{T} \sum_{t=1}^T \ln \inparen{\frac{1}{\vb_i^{(t)}} \sum_{j \in S_i} \tau_j((\mB^{(t)})^{1/2 - 1/p} \mA)} &\text{Definition of \(\phi_i\)}\\
        &\le \frac{1}{T} \sum_{t=1}^T \ln \inparen{\frac{1}{\vb_i^{(t)}} \sum_{j \in S_i} \widetilde{\tau}_j^{(t)}} &\text{Definition of \(\tau_j\)}\\
        &= \frac{1}{T} \sum_{t=1}^T \ln \inparen{\frac{\vb_i^{(t+1)}}{\vb_i^{(t)}}} &\text{Line \ref{line:alg_iter_next}} \\
        &= \frac{1}{T} \ln\inparen{\frac{\vb_i^{(T+1)}}{\vb_i^{(1)}}} \\
        &= \frac{1}{T} \ln\inparen{m/n} + \frac{1}{T} \ln(\vb_i^{(T+1)}). &\text{Line \ref{line:alg_init}}
    \end{align*}
    In the above, for the sake of analysis we define \(\vb_i^{(T+1)}\) to be as if the algorithm executed \(T+1\) iterations. Now, \(\vb_i^{(T+1)} \le \|\widetilde{\tau}\|_1 \le \nu\) by the fact that \(\textsc{OverLev}\) returns \(\nu\)-bounded leverage score overestimates. We therefore conclude the proof of \Cref{lemma:alg_analysis1}.
\end{proof}

Now, we show that \(\text{polylog}(k,m,n)\) leverage score overestimate computations suffice to find the \((\nfrac{3}{2}\cdot n\))-bounded block Lewis weight overestimates. This proves \Cref{thm:blw}.

\begin{proof}[Proof of \Cref{thm:blw}]
By \Cref{thm:lev_ove}, \textsc{OverLev} returns \(O(n)\)-bounded leverage score overestimates.
We take \(T = O\inparen{\ln\inparen{m}}\); clearly this satisfies the desired runtime guarantee. Further, using \Cref{lemma:alg_analysis1} and taking the constant in \(T\) large enough, we have for each \(i\) that \(\phi_i(\overline{\vb}) \le \ln\inparen{\frac{3}{2}}\).
Thus we have
\begin{equation}\label{eq:alg_b_guar}
    \frac{3}{2} \cdot \overline{\vb_i} \ge \sum_{j \in S_i} \tau_j(\overline{\mB}^{1/2 - 1/p} \mA).
\end{equation}
And so
\begin{align*}
    \vb_i &\ge \frac{3}{2} \overline{\vb_i} \\
    &\ge \sum_{j \in S_i} \tau_j(\overline{\mB}^{1/2 - 1/p} \mA) &\text{by \eqref{eq:alg_b_guar}} \\
    &= \sum_{j \in S_i} \tau_j(\mB^{1/2 - 1/p} \mA). &\text{\Cref{fact:lev_score_facts}}
\end{align*}
We now manipulate this guarantee into the desired guarantee of block Lewis weight overestimates
by some algebra.
Splitting powers on the left hand side of the above, we get
\[
\vb_i^{2/p} \ge \frac{1}{\vb_i^{1-2/p}} \sum_{j \in S_i} \tau_j(\mB^{1/2 - 1/p} \mA).
\]
Taking this to the \(p/2\)th power, we obtain
\begin{align*}
    \vb_i \ge \inparen{\frac{\sum_{j \in S_i} \tau_j(\mB^{1/2-1/p} \mA)}{\vb_i^{1-2/p}}}^{p/2} = \beta_i(\mB^{1-2/p})^p,
\end{align*}
as desired. To bound \(\norm{\vb}_1\), notice
\begin{align*}
    \norm{\vb}_1 = \frac{3}{2} \|\overline{\vb}\|_1 \le \frac{3}{2} \frac{1}{T} \sum_{t=1}^T \|\vb^{(t)}\|_1 \le \frac{3}{2} \nu \le O(n).
\end{align*}
Finally, to obtain the concentration statement, we observe that the above implies that we get a measure $\vlambda$ and a rounding $\mW$ that are $\Fstar$-block Lewis estimates for $\Fstar=O(n)$. This concludes the proof of \Cref{thm:blw}.
\end{proof}

\subsubsection{Special case -- \texorpdfstring{$p = 2, p_1,\dots,p_m\ge2$}{p=2,p1,...,pm>=2}}

Finally, we are ready to introduce and analyze the algorithm for the case where $p=2$ and $p_1,\dots,p_m \ge 2$. See \Cref{alg:blw_base}. The main property of \Cref{alg:blw_base} is given in \Cref{thm:blw_base}.

\begin{algorithm}[ht]
\caption{Algorithm to compute block Lewis weight overestimates, $p=2$ and $p_1,\dots,p_m \ge 2$}
\begin{algorithmic}[1]\label{alg:blw_base}
\State \textbf{Input:} \(\mA \in \R^{k \times n}\), group structure \((S_1,\dots,S_m,p_1,\dots,p_m)\).
\State Initialize \(\vb^{(0)} = \frac{n}{m} \cdot \onev\) and \(\vu^{(0)}\) such that for all $1 \le i \le m$, \(\vu^{(0)}_j = 1/\abs{S_i}\) for all \(j \in S_i\)
\For{\(t = 1, \ldots, T - 1\)}
    \State \(\widetilde{\tau}^{(t)} = \textsc{OverLev}(\mV(\vu^{(t-1)})^{1/2}\mA)\) \Comment{$\mV(\vu)$ is such that $\vv_j = \vu_j^{1-2/p_i}$ for all $1 \le i \le m$ and $j \in S_i$}
    \State \(\vb_{i}^{(t)} = \sum_{j \in S_i} \widetilde{\tau}_j^{(t)}\)
    \State \(\vu^{(t)}_j = \widetilde{\tau}_j^{(t)}/(\sum_{j' \in S_i} \widetilde{\tau}_{j'}^{(t)})\) for all $1 \le i \le m$ and all $j \in S_i$
\EndFor
\State \(\overline{\vb} = \abs{S_i}^{1/T}\cdot\frac{1}{T}\sum_{t=1}^T \vb^{(t)}\)
\State \(\overline{\vu} = \frac{1}{T}\sum_{t=0}^{T-1} \vu^{(t)}\)
\State \Return \((\overline{\vb},\overline{\vu})\)
\end{algorithmic}
\end{algorithm}

\begin{lemma}
\label{thm:blw_base}
If $\textsc{OverLev}$ is a routine that returns leverage score overestimates whose sum is at most $\nu$, then \Cref{alg:blw_base} returns $\Fstar$-block Lewis overestimates (\Cref{defn:block_lewis_overestimate}) with $\Fstar = \max_{1 \le i \le m} \abs{S_i}^{1/T} \cdot \nu$.

In particular, if $T = \max_{1 \le i \le m} \log\abs{S_i}$ and $\nu \le (4/e)n$, then we get $\Fstar = 4n$.
\end{lemma}

As with the analysis of \Cref{alg:blw}, we first prove that a relevant potential is convex and then show that we can control it effectively.

\begin{lemma}
\label{lemma:blw_convex_potential_general}
Let $\mU \in \R^{k \times k}$ be a nonnegative diagonal matrix. Let $\mV(\vu)$ denote the matrix such that for all $i$ and $j \in S_i$, we have $\vv_j = \vu_j^{1-2/p_i}$. Then, $\phi$ as defined below is log-convex in $\vu$.
\begin{align*}
    \text{for all } j \in S_1 \cup \dots \cup S_m:\quad\quad\phi_j(\vu) = \lnv{\frac{\tau_j\inparen{(\mV(\vu))^{1/2}\mA}}{\vu_j}}
\end{align*}
\end{lemma}
\begin{proof}[Proof of \Cref{lemma:blw_convex_potential_general}]
We expand the above definition after taking the $\ln$ of both sides.
\begin{align*}
    \phi_j(\vu) &= \lnv{\frac{\tau_j\inparen{(\mV(\vu))^{1/2}\mA}}{\vu_j}} = \lnv{\va_j^{\top}\inparen{\mA^{\top}\mV(\vu)\mA}^{-1}\va_j} + \lnv{\frac{\vv_j}{\vu_j}} \\
    &= \lnv{\va_j^{\top}\inparen{\mA^{\top}\mV(\vu)\mA}^{-1}\va_j} - \frac{2}{p_i}\lnv{\vu_j}.
\end{align*}
The last term is convex, so it is sufficient to argue that the first term is convex. To see this, first observe that by concavity of $x^{1-2/p_i}$, we have
\begin{align*}
    \lambda\mV(\vu^{(1)}) + (1-\lambda)\mV(\vu^{(2)}) \preceq \mV(\lambda\vu^{(1)} + (1-\lambda)\vu^{(2)}),
\end{align*}
which implies
\begin{align*}
    \lnv{\va_j^{\top}\inparen{\mA^{\top}\mV(\lambda\vu^{(1)} + (1-\lambda)\vu^{(2)})\mA}^{-1}\va_j} &\le \lnv{\va_j^{\top}\inparen{\lambda\mA^{\top}\mV(\vu^{(1)})\mA + (1-\lambda)\mA^{\top}\mV(\vu^{(2)})\mA}^{-1}\va_j} \\
    &\le \lambda\lnv{\va_j^{\top}\inparen{\mA^{\top}\mV(\vu^{(1)})\mA}^{-1}\va_j} \\
    &\quad + (1-\lambda)\lnv{\va_j^{\top}\inparen{\mA^{\top}\mV(\vu^{(2)})\mA}^{-1}\va_j}.
\end{align*}
Above, the last line follows from the well-known (see, e.g., \cite[Lemma 3.4]{ccly19}) fact that $\lnv{\va_j^{\top}\mM^{-1}\va_j}$ is convex in $\mM$ where $\mM$ is symmetric positive-semidefinite. This completes the proof of \Cref{lemma:blw_convex_potential_general}.
\end{proof}

Next, we will show that any choice of nonnegative $\vu$ such that $\sum_{j \in S_i} \vu_j = \vb_i^{\inparen{1-\frac{2}{p}}\cdot\inparen{\frac{p_i}{p_i-2}}}$ can be used to find a rounding matrix $\mW$.

\begin{lemma}
\label{lemma:blw_rounding_matrix}
For all $\vx\in\R^n$ and all nonnegative $\vu$ such that $\sum_{j \in S_i} \vu_j = \vb_i^{\inparen{1-\frac{2}{p}}\cdot\inparen{\frac{p_i}{p_i-2}}}$, we have
\begin{align*}
    \frac{\norm{\mV(\vu)^{1/2}\mA\vx}_2}{\inparen{\sum_{i' \le m} \vb_{i'}}^{1/2-1/p}} \le \gnorm{\mLambda^{1/2-1/p}\mA\vx}{2}.
\end{align*}
\end{lemma}
\begin{proof}[Proof of \Cref{lemma:blw_rounding_matrix}]
We start with the LHS.
\begin{align*}
    \norm{\mV(\vu)^{1/2}\mA\vx}_2^2 &= \sum_{i=1}^m \sum_{j \in S_i} \vu_j^{1-2/p_i}\abs{\ip{\va_j,\vx}}^2 \\
    &\le \sum_{i=1}^m \vb_i^{1-2/p}\norm{\mA_{S_i}\vx}_{p_i}^2 & \text{ by H\"older's Inequality with powers } \frac{p_i}{p_i-2}, \frac{p_i}{2}
\end{align*}
Normalizing concludes the proof of \Cref{lemma:blw_rounding_matrix}.
\end{proof}

Now, we show that finding uniform overestimates $\tau_j/\vu_j \le \vb_i^{\inparen{\frac{2}{p}-\frac{2}{p_i}}\cdot\frac{p_i}{p_i-2}}$ is enough to satisfy \Cref{defn:block_lewis_overestimate} with $\Fstar = \sum_{i \le m} \vb_i$.

\begin{lemma}
\label{lemma:blw_overestimates_conversion}
If we have $\vu$ such that $\tau_j(\mV(\vu)^{1/2}\mA)/\vu_j \le \vb_i^{\inparen{\frac{2}{p}-\frac{2}{p_i}}\cdot\frac{p_i}{p_i-2}}$ for all $i$, then the rounding matrix $\mW = \mV(\vu)\mB^{2/p-1}$ and measure $\vlambda_i = \vb_i/(\sum_{i'\le m} \vb_{i'})$ is an $\Fstar$-block Lewis overestimate (\Cref{defn:block_lewis_overestimate}) with $\Fstar = \sum_{i'\le m} \vb_{i'}$.
\end{lemma}
\begin{proof}[Proof of \Cref{lemma:blw_overestimates_conversion}]
Observe that we must have
\begin{align*}
    \frac{\tau_j\inparen{\mV^{1/2}\mA}}{\vv_j} \le \vv_j^{\frac{2}{p_i-2}} \cdot \vb_i^{\inparen{\frac{2}{p}-\frac{2}{p_i}}\cdot\frac{p_i}{p_i-2}}.
\end{align*}
Let $T = \sum_{i' \le m} \vb_{i'}$. Following \Cref{lemma:block_lewis}, let $\vlambda_i = \vb_i/T$ and $\mW$ be such that
\begin{align*}
    \frac{\mV^{1/2}}{T^{1/2-1/p}} = \mW^{1/2}\mLambda^{1/2-1/p}.
\end{align*}
In particular, this means that
\begin{align*}
    \vw_j = \frac{\vv_j}{(T\vlambda_i)^{1-2/p}} = \frac{\vv_j}{\vb_i^{1-2/p}}.
\end{align*}
Hence,
\begin{align*}
    \inparen{\sum_{j \in S_i} \inparen{\frac{\tau_j\inparen{\mV^{1/2}\mA}}{\vw_j}}^{p_i/2}}^{2/p_i} &= \vb_i^{1-2/p}\inparen{\sum_{j \in S_i} \inparen{\frac{\tau_j\inparen{\mV^{1/2}\mA}}{\vv_j}}^{p_i/2}}^{2/p_i} \\
    &\le \vb_i^{1-2/p} \cdot \vb_i^{\inparen{\frac{2}{p}-\frac{2}{p_i}}\cdot\frac{p_i}{p_i-2}} \cdot \inparen{\sum_{j \in S_i} \vv_{j}^{\frac{2}{p_i-2} \cdot \frac{p_i}{2}}}^{2/p_i} \\
    &= \vb_i^{1-2/p}\cdot\vb_i^{\inparen{\frac{2}{p}-\frac{2}{p_i}}\cdot\frac{p_i}{p_i-2}} \cdot \inparen{\sum_{j \in S_i} \vu_j}^{2/p_i} \\
    &= \vb_i^{1-2/p}\cdot\vb_i^{\inparen{\frac{2}{p}-\frac{2}{p_i}}\cdot\frac{p_i}{p_i-2}} \cdot \inparen{\sum_{j \in S_i} \vu_j}^{2/p_i} \\
    &= \vb_i^{1-2/p}\cdot\vb_i^{\inparen{\frac{2}{p}-\frac{2}{p_i}}\cdot\frac{p_i}{p_i-2}} \cdot \inparen{\vb_i^{\inparen{1-\frac{2}{p}}\cdot\inparen{\frac{p_i}{p_i-2}}}}^{2/p_i} = \vb_i = \vlambda_i\inparen{\sum_{i'\le m} \vb_{i'}},
\end{align*}
which is exactly the statement of \Cref{lemma:blw_overestimates_conversion}.
\end{proof}

We now have the tools we need to prove \Cref{thm:blw_base}.

\begin{proof}[Proof of \Cref{thm:blw_base}]
Note that \Cref{alg:blw_base} and \Cref{thm:blw_base} are reminiscent of \cite[Algorithm 2 and Theorem 4]{jls22}.

Using \Cref{lemma:blw_convex_potential_general}, we begin with using the convexity of the potential $\phi$. For all $j \in S_i$, we have
\begin{align*}
    \phi_j(\overline{\vu}) &\le \frac{1}{T}\sum_{t=0}^{T-1} \phi\inparen{\vu^{(t)}} = \frac{1}{T}\sum_{t=0}^{T-1} \lnv{\frac{\tau_j\inparen{\mV(\vu^{(t)})^{1/2}\mA}}{\vu_j^{(t)}}} \le \frac{1}{T}\sum_{t=0}^{T-1}\lnv{\frac{\widetilde{\tau}_j^{(t+1)}}{\vu_j^{(t)}}} \\
    &= \frac{1}{T}\sum_{t=0}^{T-1}\inparen{\lnv{\frac{\vu_j^{(t+1)}}{\vu_j^{(t)}}} + \lnv{\sum_{j' \in S_i} \widetilde{\tau}_{j'}^{(t+1)}}} = \frac{1}{T}\lnv{\frac{\vu_j^{(T)}}{\vu_j^{(0)}}} + \frac{1}{T}\sum_{t=0}^{T-1}\lnv{\sum_{j' \in S_i} \widetilde{\tau}_{j'}^{(t+1)}} \\
    &\le \frac{1}{T}\lnv{\frac{\vu_j^{(T)}}{\vu_j^{(0)}}} + \lnv{\frac{1}{T}\sum_{t=0}^{T-1}\sum_{j' \in S_i} \widetilde{\tau}_{j'}^{(t+1)}} = \frac{1}{T}\lnv{\frac{\vu_j^{(T)}}{\vu_j^{(0)}}} + \lnv{\abs{S_i}^{-1/T}\overline{\vb}_{i}} \le \ln\overline{\vb}_i.
\end{align*}
We now apply \Cref{lemma:blw_overestimates_conversion} and see that the measure $\vlambda_i = \overline{\vb}_i/\norm{\overline{\vb}}_1$ is a $\Fstar$-block Lewis overestimate (\Cref{defn:block_lewis_overestimate}) with $\Fstar = \norm{\overline{\vb}}_1$.

Now, observe that
\begin{align*}
    \norm{\overline{\vb}}_1 \le \max_{1 \le i \le m} \abs{S_i}^{1/T} \cdot \frac{1}{T}\sum_{t=1}^T\norm{\vb^{(t)}}_1 \le \max_{1 \le i \le m} \abs{S_i}^{1/T} \cdot \nu \le 4n,
\end{align*}
where we use our setting of $T$ and the fact that $\textsc{OverLev}$ returns leverage score estimates whose sum is at most $(4/e)n$. This completes the proof of \Cref{thm:blw_base}.
\end{proof}

We are finally ready to give the proof of \Cref{thm:computeblw}.

\begin{proof}[Proof of \Cref{thm:computeblw}]
We have three cases.
\begin{itemize}
    \item If $0 < p < 4$ and $p_1=\dots=p_m=2$, then the algorithm and guarantee on the weights follow from \Cref{alg:blw_lewismap} and \Cref{lemma:contract_alg}.
    \item If $p \ge 2$ and $p_1=\dots=p_m=2$, then the algorithm and guarantee on the weights follow from \Cref{alg:blw} and \Cref{thm:blw}.
    \item If $p = 2$ and $p_1,\dots,p_m\ge 2$, then the algorithm and guarantee on the weights follow from \Cref{alg:blw_base} and \Cref{thm:blw_base}.
\end{itemize}
We plug these guarantees into \Cref{thm:concentration} and conclude the proof of \Cref{thm:computeblw}.
\end{proof}

\subsection{Minimizing sums of Euclidean norms (Proof of \texorpdfstring{\Cref{thm:msnalg}}{Theorem 3})}
\label{sec:applications_msn}

Recall the minimizing sums of Euclidean norms (MSN) problem \eqref{eq:msn}.
Given \(\mA \in \R^{k \times n}\), a partition \(S_1, \ldots, S_m\) of \([k]\), and \(\vb_1 \in \R^{|S_1|}, \ldots, \vb_m \in \R^{|S_m|}\), we would like to find $\widehat{\vx}$ such that
\[\sum_{i=1}^m \|\mA_{S_i} \widehat{\vx} - \vb_i\|_2\le (1+\eps)\min_{\vx \in \R^n} \sum_{i=1}^m \|\mA_{S_i} \vx - \vb_i\|_2.\]
\citet{xy97} give an algorithm with iteration complexity $\widetilde{O}(\sqrt{m}\logv{\nfrac{1}{\eps}})$ for the above problem (though for an additive approximation guarantee instead of a multiplicative one), where each iteration reduces to solving linear systems in matrices $\mA^{\top}\mD\mA$ for block-diagonal matrices, where each block has size $(\abs{S_i}+1) \times (\abs{S_i}+1)$. Their algorithm is based on the primal-dual interior point method framework.


Each system solve takes the following form. Let \(\widetilde{\mA}\) be a block matrix with \(- \mI_m\) in one block, and \(\mA\) in another.
The goal is to find \(\vy\) in the system
\[\widetilde{\mA}^\top \mD \widetilde{\mA} \vy = \vz\]
where \(\mD\) is a block matrix, with one \((1 + |S_i|) \times (1 + |S_i|)\) sized block for each group
(see \cite[equation (4.13)]{xy97}).
Alternatively, using some algebra, \cite{xy97}
describe a reformulation of the system that can be set up and solved in time \(O(n^3 + m n^2 \max_i |S_i|)\).

Our main result of this section is \Cref{thm:msnalg}, which gives an improved iteration complexity for \eqref{eq:msn} when $m \gg n$.

\msnalg*



We prove \Cref{thm:msnalg} by sparsifying the objective \eqref{eq:msn} using \Cref{thm:computeblw} and then applying the primal-dual interior point method from \cite{xy97}. We state the guarantee of this algorithm in \Cref{lemma:ipm_log_barrier}.

\begin{lemma}
\label{lemma:ipm_log_barrier}
Let $\mA \in \R^{k \times n}$ and \(\vb \in \R^k\), and  $S_1, \dots, S_m$ be a partition of $k$. There exists an algorithm that returns $\widehat{\vx}$ such that
\begin{align*}
    \sum_{i=1}^m \norm{\mA_{S_i}\widehat{\vx}-\vb_{S_i}}_2 \le (1+\eps)\min_{\vx\in\R^n} \sum_{i=1}^m \norm{\mA_{S_i} \vx-\vb_{S_i}}_2.
\end{align*}
The algorithm runs in $\widetilde{O}\inparen{\sqrt{m}\logv{\nfrac{1}{\eps}}}$ calls to a linear system solver in matrices of the form $\mA^{\top}\mD\mA$ for block-diagonal matrices $\mD$, where each block has size $(\abs{S_i}+1) \times (\abs{S_i}+1)$.
\end{lemma}
\begin{proof}[Proof of \Cref{lemma:ipm_log_barrier}]
The guarantee we will reduce to is \cite[Theorem 5.2]{xy97}. However, the guarantee there is stated for an additive approximation, and we desire a multiplicative approximation. We therefore apply a few transformations to our problem so that we can apply this guarantee.

Let
\begin{align*}
    \vx_0 &\coloneqq \argmin{\vx\in\R^n} \norm{\mA\vx-\vb}_2^2.
\end{align*}
This can be found in one linear system solve. Let $V \coloneqq \norm{\mA\vx_0-\vb}_2$ and consider the following modified optimization problem:
\begin{align}
    \min_{\vx-\vx_0\in\R^n} \frac{1}{V}\sum_{i=1}^m \norm{\mA_{S_i}(\vx-\vx_0)-(\vb_{S_i}-\mA_{S_i} \vx_0)}_2.\label{eq:modified_msn}
\end{align}
We will invoke \cite[Theorem 5.2]{xy97} on the above problem \eqref{eq:modified_msn}, folding the \(\frac{1}{V}\) factor into \(\mA\) and \(\vb\). Clearly, this problem is equivalent to the problem we started with. Next, let $\xstar$ be given by
\begin{align*}
    \xstar &\coloneqq \argmin{\vx\in\R^n} \sum_{i=1}^m \norm{\mA_{S_i}\vx-\vb_{S_i}}_2,
\end{align*}
where we use \(\vb \in \R^k\) for the vector formed by stacking \(\vb_{S_1}, \ldots, \vb_{S_m}\).
Let $\opt \coloneqq \nfrac{1}{V} \cdot \sum_{i=1}^m \norm{\mA_{S_i}\xstar-\vb_{S_i}}_2$. 
Because \(\|\cdot\|_2 \le \|\cdot\|_1\), we have
\begin{align*}
    1=\frac{\norm{\mA\vx_0-\vb}_2}{V} \le \frac{\norm{\mA\xstar-\vb}_2}{V} \le \frac{1}{V}\sum_{i=1}^m \norm{\mA_{S_i}\xstar-\vb_{S_i}}_2 = \opt.
\end{align*}
Furthermore, we have by \(\|\cdot\|_1 \le \sqrt{m}\|\cdot\|_2\) (applied by considering the summation over \(i\) as an \(\ell_1\) norm) that
\begin{align*}
    \max_{1 \le i \le m} \frac{1}{V}\norm{\vb_{S_i}-\mA_{S_i}\vx_0}_2 \le \frac{1}{V}\sum_{i=1}^m \norm{\vb_{S_i}-\mA_{S_i}\vx_0}_2 \le \sqrt{m} \cdot \frac{\norm{\mA\vx_0-\vb}_2}{V} = \sqrt{m}.
\end{align*}
This implies that all the offset vectors of our transformed problem \eqref{eq:modified_msn} have polynomially bounded norm. This suffices for our iteration complexity bound, because the iteration complexity of the algorithm in \cite{xy97} depends logarithmically on the maximum norm of these vectors. Now, using their method we can solve \eqref{eq:modified_msn} up to $\eps$ additive error. This means we find $\widehat{\vx}$ such that
\begin{align*}
    \sum_{i=1}^m \norm{\mA_{S_i}(\widehat{\vx}-\vx_0)-(\vb_{S_i}-\mA_{S_i} \vx_0)}_2 &\le \sum_{i=1}^m \norm{\mA_{S_i}(\xstar-\vx_0)-(\vb_{S_i}-\mA_{S_i} \vx_0)}_2 + V\eps \\
    &= V \cdot \opt\inparen{1+\frac{\eps}{\opt}} \le V \cdot \opt\inparen{1+\eps}.
\end{align*}
where in the last inequality, we use that \(\opt \ge 1\).
Since $V \cdot \opt$ is the original optimal objective value, this completes the proof of \Cref{lemma:ipm_log_barrier}.
\end{proof}

We remark that instead of the $\sqrt{m}$ above, one can get a $\sqrt{n}$-factor relationship between the optimal objective for a least squares relaxation of our problem by using the block Lewis weights with $p_1=\dots=p_m=2$ and $p=1$. However, this will only impact lower order terms.

We are now ready to prove \Cref{thm:msnalg}.

\begin{proof}[Proof of \Cref{thm:msnalg}]
We apply \Cref{thm:computeblw} for \(p_1 = \cdots = p_m = 2\) and \(p = 1\) with approximation $\eps/3$ to the group matrices $\insquare{\mA_{S_i} \vert \vb_{S_i}}$ to find a sparsified objective
with \(\widetilde{m} = O(\eps^{-2} \cdot n(\log n)^2 \log(\nfrac{n}{\eps}))\) terms.
This requires \(\widetilde{O}(1)\) linear system solves.
This implies a \((1+\eps)\) approximation to an un-sparsified objective over any vector in \(\R^{n+1}\),
and the approximation we need comes by only considering vectors in \(\R^{n+1}\) whose last entry is \(-1\).
We plug this into the guarantee of \Cref{lemma:ipm_log_barrier} with approximation $\eps/3$. Since $(1+\eps/3)^2 \le 1+\eps$ and $(1+\eps/3)/(1-\eps/3) \le 1+\eps$, this returns a $(1+\eps)$-approximate minimizer to \eqref{eq:msn}, completing the proof of \Cref{thm:msnalg}.
\end{proof}

\paragraph{Acknowledgments} We thank Antares Chen, Meghal Gupta, Yang P. Liu, Yury Makarychev, Aaron Sidford, Darshan Thaker, Madhur Tulsiani, Taisuke Yasuda, and anonymous referees for useful discussions and feedback.

\printbibliography

\appendix

\section{Improved guarantees for sensitivity sampling}
\label{sec:sensitivity}

In this Appendix, we give a self-contained analysis of another natural row sampling scheme -- sampling proportional to $\ell_p$ sensitivities. This appeared in the body in an earlier version of the manuscript, but it was pointed out to us by anonymous referees that the relation between sensitivities and Lewis weights was already established by \citet[Theorem 2.1]{cd21} (though it was not directly used for sampling). Nonetheless, we include it, partly because our proof of the relationship between sensitivities and Lewis weights is different from that of \cite{cd21} and is, in our opinion, simpler. We remark that a subsequent work by \citet{mo24_notus} shows that the dependence we obtain here is in fact tight.

\paragraph{Improved $\ell_p$ sensitivity sampling for $1 \le p < 2$.} We consider the problem of matrix row sampling proportional to $\ell_p$ sensitivity scores where $1 \le p \le 2$. See \Cref{defn:sensitivity}.

\begin{definition}[$\ell_p$ sensitivity {\cite[Section 1.3.1]{mmwy21}}]
\label{defn:sensitivity}
Let $\mA \in \R^{m \times n}$. The $\ell_p$ sensitivity of row $i$, denoted as $\vs_{i,p}(\mA)$, is
\begin{align*}
    \vs_{i,p}(\mA) \coloneqq \max_{\vx\in\R^n\setminus \inbraces{0}} \frac{\abs{\ip{\va_i,\vx}}^p}{\norm{\mA\vx}_{p}^p}.
\end{align*}
When $p$ and $\mA$ are clear from context, we omit them and simply write $\vs_i$ in place of $\vs_{i,p}(\mA)$.
\end{definition}

Essentially, $\vs_i$ gives an upper bound to the maximum possible relative contribution that the component $\abs{\ip{\va_i,\vx}}^p$ could make to the function $\norm{\mA\vx}_{p}^p$. 

Our main result in this setting is Theorem \ref{thm:main_sensitivity}.

\begin{restatable}[Sensitivity sampling for $1 \le p \le 2$]{mainthm}{sensitivity}
\label{thm:main_sensitivity}
Let $\cG = (\mA \in \R^{k \times n}, S_1,\dots,S_m, 1,\dots,1)$ where $\abs{S_1}=\dots=\abs{S_m}$. Let $p \ge 1$ (note that $k=m$ here).

For the probability distribution $\cD = \inparen{\vrho_1,\dots,\vrho_m}$ where $\vrho_i \propto \vs_{i}$, if
\begin{align*}
    \mtilde &= \Omega\inparen{\logv{\nfrac{1}{\delta}}\vbrho \cdot\inparen{\sum_{i=1}^m \vs_i}n^{1-p/2}},
\end{align*}
and if we sample $\cM \sim \cD^{\mtilde}$, then, with probability $\ge 1-\delta$, we have:
\begin{align*}
    \text{for all } \vx \in \R^n,\quad (1-\eps)\norm{\mA\vx}_{p}^p \le \frac{1}{\mtilde}\sum_{i \in \cM} \frac{1}{\vrho_i} \cdot \abs{\ip{\va_i,\vx}}^p \le (1+\eps)\norm{\mA\vx}_{p}^p
\end{align*}
\end{restatable}

To prove \Cref{thm:main_sensitivity}, we first need the following lower bounds on sensitivities. These were already proven in \cite[Theorem 2.1]{cd21}, but we believe our proof is more straightforward.

\begin{lemma}
\label{lemma:sens_lb}
Let $\vlambda$ denote Lewis's measure for $\mA$. For all $i \in [m]$, we have $\vlambda_i \le \vs_i n^{-p/2}$.
\end{lemma}
\begin{proof}[Proof of \Cref{lemma:sens_lb}]
Let $\mW \coloneqq n\mLambda$. By monotonicity of $p$-norms under probability measures and the definition of the $\vs_i$, notice that we have for all $i \in [m]$ and all $\vx\in\R^n$ that
\begin{align*}
    \frac{\abs{\ip{\va_i,\vx}}}{\vs_i^{1/p}n^{1/p-1/2}} \le \frac{\norm{\mA\vx}_p}{n^{1/p-1/2}} \le \norm{\mW^{1/2-1/p}\mA\vx}_2.
\end{align*}
Next, recall that a standard (but fundamental) property of Lewis's measure is that it satisfies
\begin{align*}
    \max_{\vx\in\R^n\setminus\inbraces{0}} \frac{\abs{\ip{\va_i,\vx}}}{\norm{\mW^{1/2-1/p}\mA\vx}_2} = \tau_i(\mW^{1/2-1/p}\mA)^{1/2} \cdot \vw_i^{1/p-1/2} = \vw_i^{1/p}.
\end{align*}
This means that
\begin{align*}
    \vs_i^{1/p}n^{1/p-1/2} \ge \vw_i^{1/p},
\end{align*}
so taking the $p$th power of both sides yields $\vw_i \le \vs_i \cdot n^{1-p/2}$. Recalling the definition of $\vw_i$ concludes the proof of \Cref{lemma:sens_lb}.
\end{proof}

We are now ready to prove \Cref{thm:main_sensitivity}.

\begin{proof}[Proof of Theorem \ref{thm:main_sensitivity}]
Let $\mLambda$ be Lewis's measure on $[m]$, and recall the conclusion of \Cref{lemma:sens_lb}.

Next, in order to plug this into our framework, we observe that $\vrho_i = \vs_i/(\sum_{j \le m} \vs_j)$. We now establish the overestimate
\begin{align*}
    n^{-p/2}\vrho_i = \frac{n^{-p/2}\vs_i}{\sum_{j \le m} \vs_j} \ge \frac{\vlambda_i}{\sum_{j \le m} \vs_j},
\end{align*}
and so after rearranging, we see that
\begin{align*}
    \inparen{n^{-p/2}\sum_{j \le m} \vs_j}\vrho_i \ge \vlambda_i.
\end{align*}
In particular, this means that we can set $H =  C\inparen{n^{-p/2}\sum_{j \le m} \vs_j}$ for some universal constant $C$ (as \cite[Theorem 1.6]{wy23_sens} shows that $\sum_{j \le m} \vs_j \gtrsim n^{p/2}$) and choose $\vlambda$ according to Lewis's measure. Then, since $H \ge 1$ and $\valpha_i^p = \vlambda_i n^{p/2}$ (which in particular means that $\valpha_i^p/\norm{\valpha}_p^p = \vlambda_i$), by \Cref{thm:concentration}, we get
\begin{align*}
    \mtilde \lesssim \vbrho\inparen{\sum_{i=1}^m \vs_i}n^{1-p/2},
\end{align*}
which completes the proof of \Cref{thm:main_sensitivity}.
\end{proof}

\paragraph{Discussion.} Arguably the most natural measure of row importance is when $\vrho_i \propto \vs_i$ (recall \Cref{defn:sensitivity} for the definition of $\vs_i$). However, it is not known whether the sparsity of this instantiation of $\vrho_i$ has $n$-dependence $n^{\max(1,p/2)}$. This is the optimal dependence on $n$ and is achieved by choosing the $\vrho_i$ to be proportional to Lewis's measure. The previous best analysis for $\vrho_i \propto \vs_i$ yielded a sparsity $\widetilde{O}(\eps^{-2}(\sum \vs_i)^{2/p})$ when $1 \le p \le 2$, due to \citet{wy23_sens}. To contrast this against our \Cref{thm:main_sensitivity}, we start by noting that by \cite[Theorem 1.6]{wy23_sens}, we have $\sum_{i\le m} \vs_i \gtrsim n^{p/2}$, so after raising both sides to the $2/p-1$-power and then multiplying both sides by $\sum_{i \le m} \vs_i$, we get
\begin{align*}
    \inparen{\sum_{i=1}^m \vs_i}^{2/p} \gtrsim \inparen{\sum_{i=1}^m \vs_i}n^{1-p/2}.
\end{align*}
Notice that the bound given by \citet{wy23_sens} occurs on the left and ours occurs on the right. In turn, this shows that our \Cref{thm:main_sensitivity} further sharpens the best-known analysis of row sampling by $\ell_p$ sensitivities when $1 \le p \le 2$. Furthermore, our worst-case guarantees improve over that given in \cite{wy23_sens} by a factor of $n^{2/p+p/2-2}$; in particular, when $p=1$, this amounts to an improvement of a factor of $n^{1/2}$.

For $p=2$, we note that $\vs_i = \tau_i$ where $\tau_i$ is the $i$th leverage score of $\mA$. Therefore, setting $\vrho_i \propto \vs_{i,2}$ yields optimal sparsity bounds. Finally, a recent result due to \citet{pwz23} gives fast algorithms for estimating the $\vs_i$ in various parameter ranges.

\end{document}